\documentclass[11pt,a4paper]{amsart}
\usepackage[T1]{fontenc}
\usepackage[utf8]{inputenc}
\usepackage{lmodern}
\usepackage[english]{babel}
\usepackage{amsmath,amssymb,amsthm,mathtools,enumitem,microtype}
\usepackage[left=2.6cm,right=2.6cm,top=2.8cm,bottom=2.8cm]{geometry}
\usepackage{needspace,xcolor}
\usepackage[colorlinks=true,linkcolor=blue!45!black,citecolor=blue!45!black,urlcolor=blue!45!black]{hyperref}
\hypersetup{pdftitle={The consistency strength of the failure of amenable C-sequences},
 pdfauthor={Gabriel Fernandes},pdfsubject={Equiconsistency and transfinite indescribability},
 hypertexnames=false}
\newtheorem{theorem}{Theorem}[section]
\newtheorem{lemma}[theorem]{Lemma}

\newtheorem{corollary}[theorem]{Corollary}
\numberwithin{equation}{section}
\theoremstyle{definition}
\newtheorem{definition}[theorem]{Definition}
\theoremstyle{remark}
\newtheorem{remark}[theorem]{Remark}
\newcommand{\A}{\mathcal A}
\newcommand{\SA}{\mathcal{SA}}
\newcommand{\NS}{\mathrm{NS}}
\newcommand{\NAC}{\mathsf{NAC}}
\newcommand{\D}{\mathbb D}
\newcommand{\PP}{\mathbb P}

\newcommand{\EE}{\mathbb E}

\newcommand{\ZFCc}{\mathsf{ZFC}^{-}_{\mathrm{Coll}}}
\newcommand{\restr}{\mathbin{\upharpoonright}}
\DeclareMathOperator{\crit}{crit}
\DeclareMathOperator{\dom}{dom}
\DeclareMathOperator{\tc}{tc}
\DeclareMathOperator{\supp}{supp}
\DeclareMathOperator{\Ord}{Ord}
\DeclareMathOperator{\Inacc}{Inacc}
\DeclareMathOperator{\Reg}{Reg}
\DeclareMathOperator{\Comp}{Comp}
\DeclareMathOperator{\acc}{acc}
\DeclareMathOperator{\otp}{otp}
\DeclareMathOperator{\cf}{cf}
\DeclareMathOperator{\Hull}{Hull}
\newcommand{\TrV}{\operatorname{Tr}^{V}}

\title[Failure of amenable $C$-sequences]{The consistency strength of the failure\texorpdfstring{\\}{}
of amenable $C$-sequences}
\author{Gabriel Fernandes}
\address{Department of Mathematics, Institute of Mathematics and Computer
Sciences, University of S\~ao Paulo, S\~ao Carlos, SP, Brazil}
\email{fernandes@icmc.usp.br}
\thanks{The author was supported in part by the S\~ao Paulo Research
Foundation (FAPESP), Brazil, grant 2025/09425-1.}
\date{November 1, 2026}
\begin{document}
\raggedbottom
\begin{abstract}
We determine the consistency strength of the existence of a regular
uncountable cardinal carrying no amenable $C$-sequence. It is
equiconsistent with the existence of an inaccessible cardinal $\kappa$
which is $\Pi^1_\xi$-indescribable for every $\xi<\kappa^+$ in
the transfinite second-order hierarchy of indescribability. For the lower bound,
we prove that the failure of amenability at $\kappa$ implies that
$L$ contains such a cardinal at or below $\kappa$, using constructible
guessing sequences, correct constructible levels, and a transfinite
induction. For the upper bound, we prepare the cardinal so that every proper
initial segment of a full Hechler iteration of length $\kappa^+$
preserves weak compactness. Each final $C$-sequence appears at a
bounded stage and is captured at the next stage on a stationary set;
the remaining closed iteration preserves that capture.
\end{abstract}
\maketitle

\section{Introduction}\label{sec:introduction}

An amenable $C$-sequence prevents a single club from having its
initial segments contained in the prescribed clubs on a stationary
set of indices. We determine the consistency strength of the
nonexistence of such sequences in terms of the transfinite second-order
indescribability hierarchy below the successor of the cardinal.

\begin{definition}\label{def:amenability}
Let $\kappa$ be regular uncountable. A $C$-sequence on
$S\subseteq\kappa$ is a sequence
$\vec C=\langle C_\beta:\beta\in S\rangle$ such that
$C_\beta$ is club in $\beta$ at every nonzero limit index.
We set $C_0=\varnothing$ and $C_{\alpha+1}=\{\alpha\}$ when
these indices occur. For a club $D\subseteq\kappa$, define
\[
 \operatorname{Cap}(D,\vec C)
       =\{\beta\in S:D\cap\beta\subseteq C_\beta\}.
\]
The sequence is \emph{amenable} if every such capture set is
nonstationary, and \emph{strongly amenable} if every such capture
set is bounded. Let $\A_\kappa$ and $\SA_\kappa$ be the
families of domains admitting the respective sequences. We write
\[
 \NAC(\kappa)\quad\Longleftrightarrow\quad
 \kappa\text{ is regular uncountable and }\kappa\notin\A_\kappa.
\]
\end{definition}

These conventions agree with the equivalent formulations in
\cite[Definitions~3.1, 4.1 and Fact~4.2]{IR}; amenable sequences
also appear in \cite{BR}. Inamdar and Rinot asked whether
$\NAC(\kappa)$ implies that $\kappa$ is ineffable in $L$
\cite[Conjecture~4.13]{IR}. A related question asks whether
the existence of a cardinal with no amenable $C$-sequence is
equiconsistent with an ineffable cardinal
\cite[Question~5.5]{BNZRad}. A forcing construction from a
$\Pi^2_1$-indescribable cardinal was given in \cite{Fernandes}.

\begin{definition}\label{def:H}
For a cardinal $\kappa$, let $\mathsf H(\kappa)$ assert that
$\kappa$ is inaccessible and
\[
 (\forall\xi<\kappa^+)\,
       \kappa\text{ is }\Pi^1_\xi\text{-indescribable}.
\]
In a statement $M\models\mathsf H(\kappa)$, the successor
cardinal, syntax, subsets and truth predicates are all computed
in $M$. In particular, in $L$ the bound is $(\kappa^+)^L$.
\end{definition}

We use $\Pi^1_\xi$-indescribability in the sense of
\cite[Definitions~3.1, 3.2 and 3.9]{CodyHigher}.
The definition and the associated syntax are given explicitly in
Section~\ref{sec:notation}.

\begin{theorem}\label{thm:main}
The following theories are equiconsistent:
\begin{enumerate}[label=\textup{(\arabic*)}]
\item $\mathsf{ZFC}+\exists\kappa\,\mathsf H(\kappa)$;
\item $\mathsf{ZFC}+\exists\kappa\,\NAC(\kappa)$.
\end{enumerate}
\end{theorem}

For the lower bound we prove the following ideal inclusions.

\begin{theorem}\label{thm:lower}
Suppose $\NAC(\kappa)$. Then $\kappa$ is inaccessible in $L$ and
\[
 I_{\mathrm{ineff}}^L(\kappa)\subseteq\A_\kappa^V
 \quad\text{or}\quad
 (\forall\xi\in[1,(\kappa^+)^L))\,
           I_\xi^L(\kappa)\subseteq\SA_\kappa^V. \tag{1.1}\label{intro:ideal}
\]
Here $I_{\mathrm{ineff}}^L$ and $I_\xi^L$ are respectively the
constructible ineffability and $\Pi^1_\xi$-indescribability ideals.
Consequently,
\[
 L\models\text{``$\kappa$ is ineffable or $\mathsf H(\kappa)$''},
 \qquad
 L\models\exists\lambda\leq\kappa\,\mathsf H(\lambda).
                                                   \tag{1.2}\label{intro:cardinal}
\]
\end{theorem}

If the first clause of \eqref{intro:ideal} fails, a constructible nonineffable set which is positive for the
external amenability ideal supplies a constructible guessing
sequence. Stationary capture and the pressing-down lemma turn this sequence
into a function dominating all constructible $\kappa$-functions.
Condensation then localizes the domination on a set whose complement
belongs to $\SA_\kappa^V$. The resulting set is used to start
the induction on $\xi<(\kappa^+)^L$.

At each stage we also prove a reflection statement at smaller regular
cardinals and construct clubs in $V$ almost contained, on the chosen domain,
in the higher clubs computed in $L$.
At successor stages we choose the least constructible counterexamples
and the least levels correct for a fixed fragment of infinitary syntax.
Their Skolem hull collapses induce coherent elementary maps. If a club
captured unboundedly many entries, the direct limit of these maps would
give a constructible counterexample contradicting reflection at the
preceding stage. At limit stages the canonical functions determine the
earlier indices used in the diagonal intersections. Stationary capture
then gives the club needed for the next successor stage.

In the ineffable case of \eqref{intro:cardinal},
\cite[Corollary~3.5]{CodyHoly} gives a cardinal satisfying
$\mathsf H$ strictly below $\kappa$ in $L$.

The opposite consistency implication uses the following forcing result.

\begin{theorem}\label{thm:upper}
If $\mathsf H(\kappa)$, there is a set-forcing extension satisfying
$\NAC(\kappa)$.
\end{theorem}

Generalized Hechler forcing at a weakly compact cardinal captures each
ground-model $C$-sequence on a stationary set of inaccessible indices.
We use a lottery preparation to preserve weak compactness through
every proper initial segment of the full Hechler iteration of length
$\kappa^+$. For each such segment, its presentation and the assertions
needed to lift an embedding are expressed by a well-founded formula
of size $\kappa$. This formula has a normal form of level below
$\kappa^+$, so the required reflection follows from $\mathsf H(\kappa)$.
Every $C$-sequence in the final extension belongs to a bounded
intermediate extension. It is captured at the next coordinate, and
the remaining closed iteration preserves the stationary capture set.

Section~\ref{sec:notation} fixes the transfinite conventions.
Sections~\ref{sec:foundations}--\ref{lb:section} establish
Theorem~\ref{thm:lower}. Sections~\ref{sec:capture} and
\ref{sec:iteration} prove the capture and iteration facts.
Section~\ref{up:section} proves Theorem~\ref{thm:upper}.
Section~\ref{sec:conclusion} completes the proofs of the main theorems.

\section{Transfinite syntax and conventions}\label{sec:notation}

For a limit ordinal $\eta$ and $X\subseteq\eta$, put
$\acc(X)=\{\delta<\eta:0<\delta=\sup(X\cap\delta)\}$ and
$\operatorname{cl}_\eta(X)=X\cup\acc(X)$.
All elementary models used below are well-founded and transitive
after collapse. The notation $M^{<\kappa}\subseteq M$ always
specifies external closure, in the stated ambient universe.

We use the canonical reflection functions and transfinite syntax of
\cite[Sections~2--3]{CodyHigher} in Definition~\ref{def:H}. For a regular cardinal $\lambda$, fix bijections
$b_{\lambda,\zeta}:\lambda\to\zeta$ for
$\lambda\leq\zeta<\lambda^+$, with
$b_{\lambda,\lambda}=\mathrm{id}_\lambda$. Put
\[
 F^\lambda_\zeta(\alpha)=
 \begin{cases}
  \zeta,&\zeta<\lambda,\\
  b_{\lambda,\zeta}''\alpha,&\lambda\leq\zeta<\lambda^+,
 \end{cases}
 \qquad
 f^\lambda_\zeta(\alpha)=\otp(F^\lambda_\zeta(\alpha)).
                                                               \tag{2.1}
\]
In the constructible argument all these data are in $L$, and the
bijections are the $<_L$-least ones. In the forcing argument any
fixed family may be used. The resulting indescribability notions
are independent of this choice \cite[Corollary~3.7]{CodyHigher}.

Define the classes $\Pi^1_\xi$ and $\Sigma^1_\xi$, for
$\xi<\lambda^+$, recursively by
\[
 \Pi^1_0=\Sigma^1_0
 =\{\varphi:\varphi\text{ is a first-order formula}\}.
\]
At successor levels, put
\[
 \begin{aligned}
  \Pi^1_{\xi+1}
   &=\{\forall X_1\cdots\forall X_n\,\varphi:
        n<\omega,\ \varphi\in\Sigma^1_\xi\},\\
  \Sigma^1_{\xi+1}
   &=\{\exists X_1\cdots\exists X_n\,\varphi:
        n<\omega,\ \varphi\in\Pi^1_\xi\},
 \end{aligned}
\]
where the $X_i$ are second-order variables. At nonzero limit
levels $\xi$, put
\[
 \begin{aligned}
  \Pi^1_\xi
   &=\left\{\bigwedge_{\zeta<\xi}\varphi_\zeta:
        \varphi_\zeta\in\Pi^1_\zeta\text{ for every }\zeta<\xi\right\},\\
  \Sigma^1_\xi
   &=\left\{\bigvee_{\zeta<\xi}\varphi_\zeta:
        \varphi_\zeta\in\Sigma^1_\zeta\text{ for every }\zeta<\xi\right\}.
 \end{aligned}
\]
Formulas have only finitely many free variables and
second-order parameters. Their first-order parameters may vary
among the subformulas. Every particular formula of level below
$\lambda^+$, together with its required restriction data, has
hereditary size at most $\lambda$.

For a regular cardinal $\alpha<\lambda$, let
\[
 \pi^\lambda_{\xi,\alpha}:F^\lambda_\xi(\alpha)
       \longrightarrow f^\lambda_\xi(\alpha),\qquad
 \pi^\lambda_{\xi,\alpha}(\zeta)
       =\otp(F^\lambda_\xi(\alpha)\cap\zeta)
\]
be the increasing bijection. We use the restriction operation of
\cite[Definition~3.2]{CodyHigher}. If $\xi<\lambda$ and
$\varphi\in\Pi^1_\xi\cup\Sigma^1_\xi$, put
\[
 \varphi|^\lambda_\alpha
   =\varphi[\vec A\mapsto\vec A\cap V_\alpha],
 \qquad
 \vec A\cap V_\alpha=(A_1\cap V_\alpha,\ldots,A_m\cap V_\alpha),
\]
provided that every first-order parameter of $\varphi$ belongs to
$V_\alpha$; otherwise the restriction is undefined. Free variables
and bound variables are unchanged. For $\xi\geq\lambda$, define
the operation recursively. At a successor level,
\[
 \begin{aligned}
  (\forall X_1\cdots\forall X_n\,\psi)|^\lambda_\alpha
    &=\forall X_1\cdots\forall X_n\,(\psi|^\lambda_\alpha),\\
  (\exists X_1\cdots\exists X_n\,\psi)|^\lambda_\alpha
    &=\exists X_1\cdots\exists X_n\,(\psi|^\lambda_\alpha).
 \end{aligned}
\]
At a limit level $\xi\geq\lambda$, put
\[
 \begin{aligned}
  \left(\bigwedge_{\zeta<\xi}\psi_\zeta\right)|^\lambda_\alpha
    &=\bigwedge_{\eta<f^\lambda_\xi(\alpha)}
       \left(\psi_{(\pi^\lambda_{\xi,\alpha})^{-1}(\eta)}
                                  |^\lambda_\alpha\right),\\
  \left(\bigvee_{\zeta<\xi}\psi_\zeta\right)|^\lambda_\alpha
    &=\bigvee_{\eta<f^\lambda_\xi(\alpha)}
       \left(\psi_{(\pi^\lambda_{\xi,\alpha})^{-1}(\eta)}
                                  |^\lambda_\alpha\right).
 \end{aligned}
\]
The conjunction is defined only if, for every
$\eta<f^\lambda_\xi(\alpha)$, its displayed constituent is
defined and belongs to $\Pi^1_\eta$; the disjunction requires
$\Sigma^1_\eta$ instead. At a successor, the matrix must have
a defined restriction. For each fixed formula, these requirements
hold at the regular points of a club, and the restriction has
complexity $f^\lambda_\xi(\alpha)$ there. We use the corresponding
identities for canonical functions and iterated restrictions from
\cite[Lemmas~2.7--2.10 and Proposition~2.11]{CodyHigher}.
Inside $L$, all rank structures and parameter restrictions in these
definitions are computed in $L$; at inaccessible $\alpha$,
$(V_\alpha)^L=L_\alpha$. Section~\ref{lc:section} proves the
stronger composition identity needed for the chosen canonical
bijections in $L$.

For a formula $\varphi$ over $V_\lambda$, write
\[
 \operatorname{Ref}_\lambda(\varphi)
 =\{\alpha<\lambda:\alpha\text{ is regular},\
       \varphi|^\lambda_\alpha\text{ is defined},\
       V_\alpha\models\varphi|^\lambda_\alpha\}.
\]
Inside $L$, this notation uses $(V_\alpha)^L$ and satisfaction in $L$.

For a regular uncountable cardinal $\lambda$ and $\xi<\lambda^+$,
a set $S\subseteq\lambda$ is \emph{$\Pi^1_\xi$-indescribable} if
\[
 (\forall\varphi\in\Pi^1_\xi)\,
 \bigl(V_\lambda\models\varphi
       \ \Longrightarrow\
       S\cap\operatorname{Ref}_\lambda(\varphi)\ne\varnothing\bigr),
\]
with the parameters allowed above. We call $\lambda$
$\Pi^1_\xi$-indescribable when this holds for $S=\lambda$.

Let $\lambda$ be inaccessible in $L$ and
$1\leq\xi<(\lambda^+)^L$. The ideal $I_\xi^L(\lambda)$ consists
of the constructible sets $S\subseteq\lambda$ that are not
$\Pi^1_\xi$-indescribable in $L$. Explicitly, for
$S\in\mathcal P^L(\lambda)$,
\[
 \begin{split}
 S\in I_\xi^L(\lambda)\quad\Longleftrightarrow\quad
 L\models\bigl[\exists\varphi\in\Pi^1_\xi\,
  \bigl(V_\lambda\models\varphi\ \wedge\
         S\cap\operatorname{Ref}_\lambda(\varphi)=\varnothing\bigr)\bigr].
 \end{split}
\]
Here syntax, parameters and satisfaction are computed in $L$.
If $\lambda$ is not $\Pi^1_\xi$-indescribable in $L$,
then $I_\xi^L(\lambda)=\mathcal P^L(\lambda)$. For the
initial step of the induction, put $I_0^L(\lambda)=\NS_\lambda^L$.
When proper, these ideals are normal and contain
$\NS_\lambda^L$ \cite[Theorem~5.6]{CodyHigher}.

We also use the $\Pi^1_\xi$-clubs of
\cite[Definition~5.16]{CodyHigher}. All clauses in the following
recursion are computed in $L$. At level zero they are club subsets of
$\lambda$ in $L$. At a successor level,
\[
 \begin{split}
 Q\text{ is a }\Pi^1_{\eta+1}\text{-club in }\lambda
 \quad\Longleftrightarrow\quad
 &Q\in(I_\eta^L(\lambda))^+\ \wedge\ \exists B\subseteq\lambda
       \text{ club}\ \forall\alpha\in B\cap\Reg(\lambda)\\[-2pt]
 &\bigl[Q\cap\alpha\in(I_{f^\lambda_\eta(\alpha)}^L(\alpha))^+
                                      \ \Longrightarrow\ \alpha\in Q\bigr].
 \end{split}
\]
At a nonzero limit level,
\[
 \begin{split}
 Q\text{ is a }\Pi^1_\xi\text{-club in }\lambda
 \quad\Longleftrightarrow\quad
 &(\forall\eta<\xi)\ Q\in(I_\eta^L(\lambda))^+\\
 &{}\wedge\ \exists B\subseteq\lambda\text{ club}\
       \forall\alpha\in B\cap\Reg(\lambda)\\[-2pt]
 &\bigl[(\forall\eta<f^\lambda_\xi(\alpha))\,
          Q\cap\alpha\in(I_\eta^L(\alpha))^+
                                      \ \Longrightarrow\ \alpha\in Q\bigr].
 \end{split}
\]
We always intersect $B$ with the clubs on which the displayed local
indices and restrictions are defined. Intersecting $Q$ with such a
closure witness $B$ preserves its lower positivities and makes the
closure implication hold at every relevant regular cardinal.
The clubs of closure points used below ensure that these cardinals are
inaccessible in $L$. We use both parts of
\cite[Theorem~5.17]{CodyHigher}: characterization of positive sets
by meeting every higher club, and a formula defining higher clubs
whose restrictions are correct, on one club, for all local subsets.

Truth over $L_\lambda$ in the lower bound uses full second-order
semantics in $L$, except when an internal predicate domain is specified.
For satisfaction inside a set model, we require the first-order domain,
predicate domain, formula codes and satisfaction relation to belong to
the model.

\section{Capture ideals and constructible coherence}\label{sec:foundations}

Throughout the lower-bound argument, $\kappa$ is regular uncountable in
$V$. A superscript $L$ on an ideal, a trace, or an indescribability
assertion specifies its entire computation, including its second-order
semantics. A set or club called \emph{constructible} belongs to $L$; auxiliary
clubs, sets and ideals may belong only to $V$.
Write $\Reg^V(\kappa)$ for the set of regular uncountable cardinals below
$\kappa$. For $S\subseteq\kappa$, put
\[
 \TrV(S)=\{\beta\in \Reg^V(\kappa):S\cap\beta\text{ is stationary in }\beta
                                    \text{ in }V\}.
\]
For any ideal $J$, $J^*$ denotes its dual filter and $J^+$ its positive sets.

The $\kappa$-completeness, inclusions and trace assertion below, as well
as the normality of $\A_\kappa$, are due to Inamdar and Rinot
\cite[Lemma~3.7(1)--(5) and Lemma~4.8(1)--(2)]{IR}.
We include a proof, also recording the normality of $\SA_\kappa$
and the final consequence concerning inaccessibility in $L$.

\begin{lemma}\label{fd:SA}
The families $\A_\kappa$ and $\SA_\kappa$ are normal,
$\kappa$-complete ideals, possibly improper, and
\[
 \NS_\kappa\subseteq\SA_\kappa\subseteq\A_\kappa,
 \qquad \kappa\setminus \Reg^V(\kappa)\in\SA_\kappa.
\]
If $S$ is stationary, then $\kappa\setminus\TrV(S)\in\SA_\kappa$.
If $\kappa\notin\SA_\kappa$, then $\kappa$ is inaccessible in $L$.
\end{lemma}
\begin{proof}
For a limit $\beta>\omega$ with $\cf(\beta)<\beta$, choose
\[
 C_\beta\subseteq\beta\text{ club},\qquad
 \otp(C_\beta)=\cf(\beta)<\min(C_\beta).
\]
Fix a club $D\subseteq\kappa$ and $0<\xi\in D$. If
$\beta>\xi$ and $D\cap\beta\subseteq C_\beta$, then
$\cf(\beta)<\xi$. Choose $\eta<\kappa$ such that
$\otp(D\cap\eta)>\xi$. For a captured $\beta>\eta$ we would have
\[
 \xi<\otp(D\cap\beta)\leq\otp(C_\beta)=\cf(\beta)<\xi.
\]
Thus capture on this domain is bounded. For successor indices,
choose $\xi_0<\xi_1$ in $D$; no $\beta>\xi_1$ is captured by
$C_{\alpha+1}=\{\alpha\}$. The remaining indices form a bounded
set. Hence $\kappa\setminus \Reg^V(\kappa)\in\SA_\kappa$.

Both families are downward closed by restriction of their witnesses.
If $\mu<\kappa$ and $S_i\in\SA_\kappa$ have witnesses
$\vec C^{\,i}$, put
\[
 i(\beta)=\min\{i<\mu:\beta\in S_i\},\qquad
 C_\beta=C^{i(\beta)}_\beta.
\]
Then $\operatorname{Cap}(D,\vec C)\subseteq
\bigcup_{i<\mu}\operatorname{Cap}(D,\vec C^{\,i})$ is bounded.
The same inclusion proves $\kappa$-completeness of $\A_\kappa$,
using the $\kappa$-completeness of $\NS_\kappa$.

If $S\cap E=\varnothing$ for a club $E\subseteq\kappa$, put
$C_\beta=\{\eta<\beta:\eta>\sup(E\cap\beta)\}$ at limit
$\beta\in S$. Given $\xi\in D\cap E$, no $\beta>\xi$ is
captured. Thus $\NS_\kappa\subseteq\SA_\kappa$.

For normality, let
$S=\{\beta<\kappa:(\exists i<\beta)\,\beta\in S_i\}$ and,
for $\beta\in S\cap \Reg^V(\kappa)$, put
\[
 i_\beta=\min\{i<\beta:\beta\in S_i\},\qquad
 C_\beta=C^{i_\beta}_\beta\setminus(i_\beta+1).
\]
Fix $0<\xi\in D$. A captured $\beta>\xi$ satisfies
$i_\beta<\xi$. If $b_i<\kappa$ bounds
$\operatorname{Cap}(D,\vec C^{\,i})$, then
\[
 \operatorname{Cap}(D,\vec C)\subseteq
 (\xi+1)\cup\bigcup_{i<\xi}(b_i+1),
\]
which is bounded. The omitted indices lie in $\kappa\setminus \Reg^V(\kappa)$.
For $\A_\kappa$, a stationary capture set would, by Fodor's lemma,
have a stationary subset on which $i_\beta$ is constant,
contradicting amenability of $\vec C^{\,i}$.

If $S$ is stationary and $\beta\in \Reg^V(\kappa)\setminus\TrV(S)$,
choose $C_\beta$ club in $\beta$ with $C_\beta\cap S=\varnothing$.
For $\xi\in D\cap S$,
\[
 \beta>\xi\quad\Longrightarrow\quad D\cap\beta\nsubseteq C_\beta.
\]
Together with $\kappa\setminus \Reg^V(\kappa)\in\SA_\kappa$, this proves the
trace assertion.

Finally, if $\kappa\notin\SA_\kappa$, the inclusions
$\NS_\kappa\subseteq\SA_\kappa$ and
$\kappa\setminus\Reg^V(\kappa)\in\SA_\kappa$ imply that
$\Reg^V(\kappa)$ is stationary, hence unbounded, in $\kappa$.
Therefore $\kappa$ is strongly inaccessible in $L$.
\end{proof}

Fix the stage-ordered canonical well-order $<_L$ and the theory
\[
 T_0=\mathsf{ZFC}^{-}+\mathrm{Collection}+\mathrm{Choice}+V=L.
\]
Here Power Set is omitted, and Separation and Collection are included
\cite{GHJ}.
We use the following standard consequences of condensation
\cite[Chapter~13]{Jech}. For regular uncountable $\rho$ in $L$,
$H_\rho^L=L_\rho\models T_0$. Every constructible set of hereditary
$L$-size less than $\rho$ belongs to a level $L_\eta\models T_0$
with $\eta<\rho$. Any transitive set model of $T_0$ is an actual
level of $L$: induction identifies each of its internal constructible
levels with the external one. First-order satisfaction for a set
structure is absolute between the transitive models containing it.

Hulls use all first-order Skolem functions, choosing the $<_L$-least
witness. For a structure in a countable language, a regular uncountable
cardinal $\beta$, and a fixed finite tuple $p$, the set
\[
 \{\alpha<\beta:\Hull(\alpha\cup\{\beta,p\})\cap\beta=\alpha\}
                                                        \tag{\(*\)}
\]
is club. This follows by iterating the hull bound countably often;
continuity of finitary hulls proves closedness. It also applies to
expansions by finitely many external predicates, with the construction
performed in the universe containing those predicates.

If $H\prec L_\eta$, $\beta\in H$, and $H\cap\beta=\alpha$ is
limit, condensation and induction on constructible levels give
\[
 \pi:H\cong L_\tau,\quad \pi(\beta)=\alpha,\quad
 H\cap L_\beta=L_\alpha,\quad \pi\restr L_\alpha=\mathrm{id}.
                                                        \tag{\(**\)}
\]
Thus a predicate $A\in H$ on $L_\beta$ collapses to
$A\cap L_\alpha$. Every $<_L$-predecessor of an element of
$L_\tau$ is in $L_\tau$. Consequently an absolute property of
a collapsed object, together with its internal $<_L$-minimality,
identifies the actual least constructible object with that property.

\begin{lemma}\label{fd:minimality}
Let $p_\alpha,p_\beta\in L$ be finite tuples of parameters, and put
\[
 \eta_\gamma=\min\{\eta:L_\eta\models T_0\ \wedge\ p_\gamma\in L_\eta\}
 \qquad(\gamma\in\{\alpha,\beta\}).
\]
If a hull of $L_{\eta_\beta}$ collapses to $L_\tau$ and sends $p_\beta$
to $p_\alpha$, then
$\tau=\eta_\alpha$. The collapsed structure is generated by
the images of the hull generators. Elementary maps between such
generated structures are uniquely determined by their values on
those generators.
\end{lemma}
\begin{proof}
Let $\pi:H\cong L_\tau$ be the collapse. Since
$L_\tau\models T_0$ and $\pi(p_\beta)=p_\alpha$,
$\eta_\alpha\leq\tau$. Elementarity also gives
\[
 L_\tau\models\neg\exists\rho\,
       (L_\rho\models T_0\ \wedge\ p_\alpha\in L_\rho).
\]
If $\eta_\alpha<\tau$, then $L_{\eta_\alpha}\in L_\tau$;
set satisfaction is absolute and contradicts this assertion. Hence
$\eta_\alpha=\tau$. For each canonical Skolem term $t$,
\[
 \pi(t^H(\vec a))=t^{L_\tau}(\pi(\vec a)).
\]
These values exhaust the collapsed hull. Two elementary maps
agreeing on its generators therefore agree everywhere.
\end{proof}

A directed system of elementary embeddings between levels of $L$,
indexed by a nonempty set $A$ of ordinals, consists of structures
$M_\alpha=L_{\eta_\alpha}$ for $\alpha\in A$ and elementary embeddings
\[
 j_{\alpha\beta}:(M_\alpha,\in)\longrightarrow(M_\beta,\in)
 \qquad(\alpha\leq\beta\text{ in }A)
\]
satisfying
\[
 j_{\alpha\alpha}=\mathrm{id}_{M_\alpha},\qquad
 j_{\beta\gamma}\circ j_{\alpha\beta}=j_{\alpha\gamma}
 \qquad(\alpha\leq\beta\leq\gamma\text{ in }A).
\]
The second identity is the coherence condition.
We use the standard direct-limit construction for elementary embeddings
(see \cite[Section~3]{Cummings}), verifying its well-foundedness below.

\begin{lemma}\label{fd:limit}
Let $A\subseteq\kappa$ be unbounded and $\cf^V(\kappa)>\omega$.
Suppose $\langle M_\alpha,j_{\alpha\beta}:\alpha\leq\beta
\text{ in }A\rangle$ is such a system, with
$M_\alpha\models T_0$ and $\alpha\in M_\alpha$ for every $\alpha\in A$.
Assume also that
\[
 j_{\alpha\beta}\restr\alpha=\mathrm{id}_\alpha,
 \qquad j_{\alpha\beta}(\alpha)=\beta
 \qquad(\alpha\leq\beta\text{ in }A).
\]
Its direct limit is well-founded. Its transitive collapse $N$ is a
level of $L$, and the limit maps $j_\alpha:M_\alpha\to N$ satisfy
$j_\alpha(\alpha)=\kappa$ for every $\alpha\in A$.
\end{lemma}
\begin{proof}
On pairs $(\alpha,x)$ with $\alpha\in A$ and $x\in M_\alpha$,
define
\[
 \begin{aligned}
 (\alpha,x)\sim(\beta,y)
 &\iff (\exists\gamma\in A\setminus\max\{\alpha,\beta\})\,
          j_{\alpha\gamma}(x)=j_{\beta\gamma}(y),\\
 [\alpha,x]\ E\ [\beta,y]
 &\iff (\exists\gamma\in A\setminus\max\{\alpha,\beta\})\,
          j_{\alpha\gamma}(x)\in j_{\beta\gamma}(y).
 \end{aligned}
\]
Coherence makes both definitions independent of the comparison stage. Induction
on first-order formulas proves that $x\mapsto[\alpha,x]$ is
elementary. If $[\alpha_{n+1},x_{n+1}]\ E\ [\alpha_n,x_n]$
for $n<\omega$, choose $\gamma\in A$ above all representative
and comparison indices; this is possible since $\cf^V(\kappa)>\omega$.
Then
\[
 j_{\alpha_{n+1}\gamma}(x_{n+1})
       \in j_{\alpha_n\gamma}(x_n)\qquad(n<\omega),
\]
contradicting Foundation in the transitive $M_\gamma$. Thus the
direct limit is well-founded. Its transitive collapse $N$ satisfies
$T_0$, and therefore $N=L_\eta$ for some $\eta$.

Write $j_\alpha:M_\alpha\to N$ for the limit maps. Coherence
gives $j_\alpha(\alpha)=j_\beta(\beta)$. For $\xi<\alpha$,
$j_\alpha(\xi)=\xi$. Conversely, every member of
$j_\alpha(\alpha)$ is represented at a later stage by an ordinal
below that stage's distinguished ordinal. Consequently
\[
 j_\alpha(\alpha)=\bigcup_{\beta\in A}\beta=\kappa.
\]
\end{proof}

For $\xi=1$, $I_1^L(\lambda)$ is the ideal of sets that are not
weakly compact in $L$.

The cardinal implication
$\kappa\notin\SA_\kappa\Longrightarrow
\kappa\text{ is weakly compact in }L$
is due to Inamdar and Rinot
\cite[Lemma~3.4 and Corollary~3.8]{IR}.
The next lemma gives the corresponding ideal inclusion, with
constructible entries in the witnessing sequences.

\begin{lemma}\label{fd:robust}
If $\lambda$ is regular uncountable in $V$ and inaccessible in
$L$, then
\[
 I_1^L(\lambda)\subseteq\SA_\lambda^V.
\]
The witnessing sequence can be chosen with constructible entries
at every index regular in $V$.
\end{lemma}
\begin{proof}
Let $S\in I_1^L(\lambda)$. In $L$ choose $A\subseteq L_\lambda$
and a first-order $\chi$ such that
\[
 (L_\lambda,\in,A)\models\forall X\,\chi(X,A),\qquad
 (L_\alpha,\in,A\cap L_\alpha)\not\models\forall X\,\chi(X,A)
\]
for every $\alpha\in S$ inaccessible in $L$. For each limit
$\alpha$ at which the second assertion holds, define
\[
 \begin{aligned}
 X_\alpha&=\min_{<_L}\{X\in L\cap\mathcal P(L_\alpha):
             (L_\alpha,\in,X,A\cap L_\alpha)\models\neg\chi\},\\
 p_\alpha&=(\alpha,A\cap L_\alpha,X_\alpha),\\
 \eta_\alpha&=\min\{\eta:p_\alpha\in L_\eta\models T_0\}.
 \end{aligned}
\]
These ordinals exist below $\lambda$. For $\beta\in S$ regular
in $V$ and inaccessible in $L$, use the constructible club
\[
 C_\beta=\{\alpha<\beta:
     \Hull^{L_{\eta_\beta}}(\alpha\cup\{p_\beta\})\cap\beta=\alpha\}.
\]
At a limit $\alpha\in C_\beta$, let
$H=\Hull^{L_{\eta_\beta}}(\alpha\cup\{p_\beta\})$
and let $\pi$ collapse $H$.
Equation $(**)$ and absoluteness of first-order satisfaction give
\[
 \pi(A\cap L_\beta)=A\cap L_\alpha,
 \qquad \pi(X_\beta)=X_\alpha,
 \qquad \operatorname{coll}(H)=L_{\eta_\alpha}.
\]
For the middle identity, any $<_L$-earlier witness lies in the
collapsed level and would contradict its internal minimality. The
last identity follows from Lemma~\ref{fd:minimality}.

Suppose $J=\operatorname{Cap}(D,\vec C)$ is unbounded. For
limit $\alpha<\alpha'$ in $D$, choose $\beta\in J\setminus(\alpha'+1)$.
The inclusion of the two hulls induces
\[
 j_{\alpha\alpha'}:L_{\eta_\alpha}\longrightarrow L_{\eta_{\alpha'}},
 \qquad j_{\alpha\alpha'}(p_\alpha)=p_{\alpha'},
 \qquad j_{\alpha\alpha'}\restr L_\alpha=\mathrm{id}.
\]
Skolem generation makes these maps independent of $\beta$ and gives
$j_{\alpha'\alpha''}\circ j_{\alpha\alpha'}=j_{\alpha\alpha''}$.
Lemma~\ref{fd:limit} yields $X\in L\cap\mathcal P(L_\lambda)$
with $X\cap L_\alpha=X_\alpha$ for cofinally many $\alpha$.
Elementarity of first-order set satisfaction now gives
\[
 (L_\lambda,\in,X,A)\models\neg\chi,
\]
a contradiction.

Choose a constructible club $F\subseteq\lambda$ of strong limit
cardinals of $L$. If $\gamma<\lambda$ is regular in $V$ and
not inaccessible in $L$, then $\gamma\notin F$. Use
$C_\gamma=\{\eta<\gamma:\eta>\sup(F\cap\gamma)\}$.
For $\xi\in D\cap F$, no $\gamma>\xi$ in this domain is captured.
Lemma~\ref{fd:SA} handles the nonregular indices. Completeness
combines these three witnesses.
\end{proof}

In particular, $\kappa\notin\SA_\kappa^V$ implies that
$\kappa$ is weakly compact in $L$.

For functions on an ordinal, $f\leq^*g$ means that
$f(\xi)\leq g(\xi)$ at every sufficiently large argument.
For clubs, $D\subseteq^*E$ means that $D\setminus E$ is bounded.
At a regular uncountable $\lambda$, the following are equivalent:
\[
 \begin{gathered}
 (\exists d\in{}^\lambda\lambda)
          (\forall f\in L\cap{}^\lambda\lambda)\ f\leq^*d;\\
 (\exists D\subseteq\lambda\text{ club})
          (\forall E\in L\text{ club in }\lambda)\ D\subseteq^*E.
 \end{gathered}                                                   \tag{DC}
\]
For the forward implication take the limit closure points of $d$
and apply domination to $\xi\mapsto\min(E\setminus(\xi+1))$.
For the reverse implication take $d(\xi)=\min(D\setminus(\xi+1))$
and apply the club condition to the closure club of each constructible $f$.

The guessing ingredient in the next proof is Jensen's
$\diamondsuit^*$ theorem for nonineffable sets in $L$, used in
\cite[proof of Corollary~4.10]{IR}.
The diagonal intersection of the guessed clubs is the construction
in \cite[proof of Proposition~4.9]{IR}; here it is used to obtain a
club in $V$ almost contained in every constructible club.

\begin{lemma}\label{fd:dominator}
Suppose $\NAC(\kappa)$ and there is
$T\in L\cap\mathcal P(\kappa)$ with $T\notin\A_\kappa^V$
which is nonineffable in $L$. Then a function
$d:\kappa\to\kappa$ eventually dominates every constructible
function from $\kappa$ to $\kappa$.
\end{lemma}
\begin{proof}
Work first in $L$. If $S\subseteq\Reg^L(\kappa)$ is
nonineffable, choose a list $\langle a_\beta:\beta\in S\rangle$
with no stationary homogeneous set, and put
\[
 \eta_\beta=\min\{\eta:\beta,a_\beta\in L_\eta\models T_0\},
 \qquad \mathcal D_\beta=\mathcal P(\beta)\cap L_{\eta_\beta}.
\]
Condensation gives $\eta_\beta<\beta^+$, so
$|\mathcal D_\beta|\leq\beta$. For $B\subseteq\kappa$ in $L$,
choose $L_\theta\models T_0$ containing $\kappa,B$ and the list,
with $|L_\theta|=\kappa$, and a bijection $g:\kappa\to L_\theta$.
On a club $E$,
\[
 g''\beta\prec L_\theta,\quad g''\beta\cap\kappa=\beta,
 \quad\pi_\beta:g''\beta\cong L_{\rho_\beta},\quad
 \pi_\beta(B)=B\cap\beta.
\]
If $B\cap\beta\notin\mathcal D_\beta$, then
$\eta_\beta<\rho_\beta$ and $a_\beta\in L_{\rho_\beta}$.
Let $b_\beta=\pi_\beta^{-1}(a_\beta)\in g''\beta$; then
$b_\beta\cap\beta=a_\beta$. If these failures were stationary,
Fodor's lemma applied to $\beta\mapsto g^{-1}(b_\beta)<\beta$
would give a stationary set $U$ and a fixed $b$ with
\[
 a_\beta=b\cap\beta\qquad(\beta\in U),
\]
contradicting the choice of the list. Consequently
\[
 \forall B\in\mathcal P^L(\kappa)\ \exists E_B\in L\text{ club}\
 \forall\beta\in S\cap E_B\quad B\cap\beta\in\mathcal D_\beta.
\]
This proves the required instance of Jensen's $\diamondsuit^*(S)$ theorem.

Apply this to $S=T\cap\Reg^L(\kappa)$. At each $\beta\in T\cap \Reg^V(\kappa)$,
enumerate the clubs in $\mathcal D_\beta$ as
$\langle E_i^\beta:i<\beta\rangle$, using $\beta$ for any
unused indices, and let
\[
 B_\beta=\triangle_{i<\beta}E_i^\beta
 =\{\delta<\beta:(\forall i<\delta)\,\delta\in E_i^\beta\}.
\]
At indices in $T\setminus \Reg^V(\kappa)$, use Lemma~\ref{fd:SA}.
Since $T\notin\A_\kappa^V$, there are a club $D$ and a
stationary $U\subseteq T\cap \Reg^V(\kappa)$ with
$D\cap\beta\subseteq B_\beta$ for $\beta\in U$.

Fix a constructible club $E\subseteq\kappa$. On a stationary
subset $U'$ of $U\cap E_E\cap\acc(E)$, Fodor's lemma gives
a fixed $i<\kappa$ such that $E_i^\beta=E\cap\beta$.
For $\delta\in D\setminus(i+1)$ choose $\beta\in U'\setminus(\delta+1)$.
Then
\[
 \delta\in D\cap\beta\subseteq B_\beta,
 \qquad B_\beta\setminus(i+1)\subseteq E_i^\beta=E\cap\beta,
\]
so $D\subseteq^*E$. Equivalence (DC) gives the required
eventual dominator.
\end{proof}

\begin{lemma}\label{fd:coherent}
Suppose $\kappa\notin\SA_\kappa^V$ and $d$ is as above.
There are a club $D_0\subseteq\kappa$ in $V$ and an
$\SA_\kappa^V$-dual set
$P_1\subseteq\Reg^V(\kappa)\cap\Inacc^L(\kappa)$ such that
for $\gamma\in P_1$, $D_0\cap\gamma$ is club in $\gamma$ in $V$ and
\[
 D_0\cap\gamma\subseteq^* B
       \qquad(B\text{ is club in }\gamma\text{ in }L).
                                                               \tag{O}
\]
\end{lemma}
\begin{proof}
Let
\[
 B=\{\beta\in \Reg^V(\kappa):d''\beta\subseteq\beta\ \wedge\
       (\exists g\in L\cap{}^\beta\beta)\,g\not\leq^*d\restr\beta\}.
\]
For each limit $\alpha$ closed under $d$ for which a witness exists,
put
\[
 g_\alpha=\min_{<_L}\{g\in L\cap{}^\alpha\alpha:
                         g\not\leq^*d\restr\alpha\},\qquad
 \eta_\alpha=\min\{\eta:\alpha,g_\alpha\in L_\eta\models T_0\}.
\]
In the expansion $(L_{\eta_\beta},\in,d\restr\beta)$, use
the club of $\alpha<\beta$ satisfying
\[
 \Hull(\alpha\cup\{\beta,g_\beta\})\cap\beta=\alpha.
\]
The collapse sends the external predicate to $d\restr\alpha$.
The defining condition is the first-order assertion
\[
 g:\alpha\to\alpha\quad\wedge\quad
 \forall\eta<\alpha\ \exists\xi\in[\eta,\alpha)\ d(\xi)<g(\xi).
\]
It is absolute for the supplied graph. Initiality of $<_L$ and
Lemma~\ref{fd:minimality} identify the collapsed function and
level with $g_\alpha,L_{\eta_\alpha}$.

If a club captured unboundedly many entries, the hull inclusions
would give the coherent maps of Lemma~\ref{fd:limit}. Applied to
the reducts to the language of set theory, that lemma yields
$g\in L\cap{}^\kappa\kappa$ with
$g\restr\alpha=g_\alpha$ at cofinally many $\alpha$. Thus
\[
 \forall\eta<\kappa\ \exists\xi\geq\eta\ d(\xi)<g(\xi),
\]
contradicting $g\leq^*d$. Hence $B\in\SA_\kappa^V$.

Let $D_0=\{\delta<\kappa:\delta>0\text{ is limit and }
d''\delta\subseteq\delta\}$, and let $F$ be a constructible
club of strong limit cardinals of $L$. Set
\[
 P_1=(\Reg^V(\kappa)\cap F\cap\acc(D_0))\setminus B.
\]
Lemma~\ref{fd:SA} gives $\kappa\setminus P_1\in\SA_\kappa^V$.
Each $\gamma\in P_1$ is inaccessible in $L$, and
$d\restr\gamma$ dominates every constructible $\gamma$-function.
For a constructible club $E\subseteq\gamma$, apply this to
$\xi\mapsto\min(E\setminus(\xi+1))$. If $u<\gamma$ is its
domination bound and $\delta\in D_0\cap\gamma\setminus(u+1)$,
then
\[
 u\leq\xi<\delta\ \Longrightarrow\
 \xi<\min(E\setminus(\xi+1))\leq d(\xi)<\delta.
\]
Therefore $\delta\in\acc(E)\subseteq E$, proving (O).
\end{proof}

\begin{lemma}\label{fd:relativeone}
The set in Lemma~\ref{fd:coherent} may be refined so that
every member is weakly compact in $L$. If $\beta\leq\kappa$ is regular in $V$,
inaccessible in $L$, and $P_1\cap\beta$ is stationary in $V$,
then
\[
 X\in I_1^L(\beta)\quad\Longrightarrow\quad
 P_1\cap\beta\cap X\in\NS_\beta^V.                         \tag{$R_1$}
\]
Such a $\beta$ is weakly compact in $L$. Moreover, every constructible
$1$-club $Q\subseteq\beta$ contains the points of $P_1$ in
the accumulation set of some club in $\beta$ in $V$.
\end{lemma}
\begin{proof}
First take the set $P_1$ of Lemma~\ref{fd:coherent}. Suppose
$Y=P_1\cap\beta\cap X$ is stationary, where $X\in I_1^L(\beta)$.
Lemma~\ref{fd:robust} supplies a strongly amenable sequence on
$X$ with constructible entries at the indices in $Y$. By (O),
choose $u_\gamma<\gamma$ such that
\[
 (D_0\cap\gamma)\setminus(u_\gamma+1)\subseteq C_\gamma
 \qquad(\gamma\in Y).
\]
Fodor's lemma gives stationary $Y'\subseteq Y$ and $u<\beta$
with $u_\gamma=u$ on $Y'$. Since $P_1\subseteq\acc(D_0)$,
$D_0\cap\beta$ is club, and
\[
 \bigl((D_0\cap\beta)\setminus(u+1)\bigr)\cap\gamma
       \subseteq C_\gamma\qquad(\gamma\in Y').
\]
This contradicts strong amenability. Taking $X=\beta$ shows that
the ideal cannot be improper.

Replace $P_1$ by $P'_1=P_1\cap\TrV(P_1)$. Its complement
belongs to $\SA_\kappa^V$. Applying the statement just proved
at each $\beta\in P'_1$ shows that $\beta$ is weakly compact in
$L$. The relative statement persists: if $P'_1\cap\beta$ is
stationary, so is $P_1\cap\beta$, and
$P'_1\cap\beta\cap X\subseteq P_1\cap\beta\cap X$.

For a constructible $1$-club $Q\subseteq\beta$,
$\beta\setminus Q\in I_1^L(\beta)$ by
\cite[Theorem~5.17]{CodyHigher}. Thus some club
$F_Q\subseteq\beta$ in $V$ satisfies $P'_1\cap F_Q\subseteq Q$.
Since $\acc(F_Q)\subseteq F_Q$,
\[
 P'_1\cap\acc(F_Q)\subseteq Q.
\]
\end{proof}

The club $F_Q$ need not belong to $L$, so property (O) does not
apply to it.

\section{Correct constructible levels and canonical restrictions}
\label{lc:section}

We adapt the correct constructible levels and traces of Skolem hulls
from \cite[Definitions~3.6--3.7 and Section~3.2]{BW} to the hierarchy of
\cite[Sections~2--3]{CodyHigher}, including the restriction data among
the parameters. Throughout this section we work in $L$, except where
$V$ is mentioned.

Fix an inaccessible $\delta$.  A formula is represented by a well-founded
set-sized syntax tree, with its ordinal complexity labels, immediate subformula-indexing
functions, and finite quantifier blocks.  First-order leaves are finite formulas.  Each formula has finitely many free predicate variables;
assignments and finitely many predicate parameters are recorded separately.
At a limit complexity $\eta$, the immediate subformulas are indexed by $\eta$, as in the
definition of $\Pi^1_\eta$ and $\Sigma^1_\eta$ formulas used here. Labels decrease
along infinitary edges.  Thus validity of a presented code is absolute
between the transitive models considered below.  Formula codes of levels
below $\delta^+$ have hereditary size at most $\delta$.

Throughout this section, levels $N$ used to evaluate infinitary
formulas satisfy
\begin{equation}
 \begin{gathered}
 N=L_\nu\models T_0,\qquad \delta,L_\delta\in N,\qquad
 N\models\text{``$\delta$ is inaccessible and $\delta^+$ exists''},\\
 S_N=\mathcal P^N(L_\delta)\in N,\qquad
 C_N=H^N_{(\delta^+)^N}\in N.
 \label{lc:domains}
 \end{gathered}
\end{equation}
When a formula code belongs to $C_N$, it is always required to be
a valid formula over $L_\delta$. We do not assume
$(\delta^+)^N=(\delta^+)^L$.

\begin{lemma}\label{lc:set-satisfaction}
There is a uniform first-order definition of satisfaction of a formula
code in $C_N$ over the two-sorted set structure $(L_\delta,S_N)$, where predicate
quantifiers range over $S_N$.  A set family of codes and assignments has
a satisfaction relation which is a set of $N$.  If two transitive models satisfying \eqref{lc:domains}
have the same first-order domain, predicate domain, code and assignment, they
compute the same satisfaction value.
\end{lemma}
\begin{proof}
Fix a code $c\in C_N$. Its subformulas, the finite assignments
to their free variables, $L_\delta$ and $S_N$ are sets in $N$.
Recursion on the well-founded subformula relation defines a unique
relation $T$ satisfying the atomic clauses and
\[
 \begin{aligned}
 T(\neg d,a)&\iff\neg T(d,a),\\
 T(\forall x\,d,a)&\iff
       (\forall u\in L_\delta)\,T(d,a[x\mapsto u]),\\
 T(\forall X\,d,a)&\iff
       (\forall A\in S_N)\,T(d,a[X\mapsto A]),\\
 T\left(\bigwedge_{i\in I}d_i,a\right)&\iff
       (\forall i\in I)\,T(d_i,a).
 \end{aligned}
\]
Existential quantifiers and disjunctions have the dual clauses.
Collection and well-founded recursion give $T\in N$; its value
at the code $c$ is $\operatorname{Sat}(L_\delta,S_N,c,a)$.
The assertion that $T$ satisfies these clauses is first order and
uniform in $c$. Replacement collects the relations for a set of
codes. If two transitive models have the same indicated domains,
code and assignment, induction on the subformula relation gives
$T_1(d,a)=T_2(d,a)$ at every subformula, proving absoluteness.
\end{proof}

Fix $\xi<\delta^+$ and a tuple $p\in H_{\delta^+}$ containing
the formula codes, canonical bijections and parameters to be
preserved.
A level $N$ is \emph{correct through $\xi$ for $p$} if it satisfies
\eqref{lc:domains},
\begin{equation}
 \xi<(\delta^+)^N,\qquad p\in C_N,
 \label{lc:parameters}
\end{equation}
and, for every $c\in C_N\cap
\bigcup_{\eta\leq\xi}(\Pi^1_\eta\cup\Sigma^1_\eta)$ and every
assignment $a\in N$ to its free variables,
\[
 \operatorname{Sat}(L_\delta,S_N,c,a)
 \quad\Longleftrightarrow\quad
 (L_\delta,\mathcal P^L(L_\delta))\models c[a].
\]
The tuple $p$ will be represented by a predicate on $\delta$;
it is not required to be an element of $L_\delta$.

\begin{lemma}\label{lc:correct-levels}
For every $(\delta,\xi,p)$ as above there is a correct level $L_\eta$ of size
$\delta$, with $\eta<\delta^+$.  In a level $N$ correct through
$\xi$, one first-order test correctly decides whether a smaller level
$M=L_\rho\in N$ is correct through $\xi$ for the same parameter tuple.
\end{lemma}
\begin{proof}
Work in $L$. In $L_{\delta^{++}}$, let
\[
 S=\mathcal P(L_\delta),\qquad C=H_{\delta^+}=L_{\delta^+},
 \qquad T=\operatorname{Sat}(L_\delta,S)\restr C.
\]
Here $T$ includes the finite assignments to the free variables
of each valid code. GCH gives
$|\tc(\{S,C,T\})|\leq\delta^+$, so these sets belong to
$L_{\delta^{++}}$. Choose
\[
 H\prec L_{\delta^{++}},\qquad |H|=\delta,\qquad
 L_\delta\subseteq H,\qquad \delta,\xi,p,S,C,T\in H.
\]
Condensation gives $\pi:H\cong N=L_\nu$ with $\nu<\delta^+$.
For $x\in H\cap H_{\delta^+}$, a surjection
$\delta\twoheadrightarrow\tc(\{x\})$ belongs to $H$.
Since $\delta\subseteq H$,
\[
 \tc(\{x\})\subseteq H,\qquad \pi(x)=x.
\]
In particular $\xi,p$ and every formula code and assignment
in $H\cap H_{\delta^+}$ are fixed. Elementarity gives
\[
 \pi(S)=S_N,\qquad\pi(C)=C_N,\qquad
 \pi(T)=\operatorname{Sat}(L_\delta,S_N)\restr C_N.
\]
For every code and assignment in the collapse, its truth value
therefore equals its value in $T$, which is full truth in $L$.
This proves correctness, and in particular correctness through
$\xi$ for $p$.

For $M=L_\rho\in N$ satisfying \eqref{lc:domains} and
\eqref{lc:parameters}, compare, for every formula code
$c\in C_M\cap\bigcup_{\eta\leq\xi}(\Pi^1_\eta\cup\Sigma^1_\eta)$
and assignment $a\in M$,
\begin{equation}
 \operatorname{Sat}(L_\delta,S_M,c,a)
 \quad\Longleftrightarrow\quad
 \operatorname{Sat}(L_\delta,S_N,c,a).
 \label{lc:comparison}
\end{equation}
An internal hereditary-size witness in $M$ remains one in $N$,
so $C_M\subseteq C_N$. Correctness of $N$ identifies the
right side with full truth. Consequently this comparison is
equivalent to correctness of $M$ through $\xi$. By
Lemma~\ref{lc:set-satisfaction}, it is a uniform first-order
test about the sets in $N$.
\end{proof}

The next lemma is the version of the minimal-level argument in
\cite[Lemma~3.10]{BW} for the correctness fragment and parameter codes
specified above.

\begin{lemma}\label{lc:skolem-hull}
The least level correct through $\xi$ for $p$ is the closure of $\delta$
pointwise together with the ordered tuple $(\delta,\xi,p)$ under the
canonical Skolem functions.
\end{lemma}
\begin{proof}
Let $N=L_\nu$ be least among the levels correct through $\xi$
for $p$, and let
\[
 K=\Hull^N(\delta\cup\{\delta,\xi,p\}),\qquad
 \pi:K\cong L_\mu.
\]
For $x\in K\cap C_N$, an internal surjection
$\delta\twoheadrightarrow\tc(\{x\})$ belongs to $K$.
Since $\delta\subseteq K$, it follows that
\[
 \tc(\{x\})\subseteq K,\qquad \pi(x)=x.
\]
In particular $\pi(\delta,\xi,p)=(\delta,\xi,p)$ and
$\pi\restr L_\delta=\mathrm{id}$. Elementarity transfers
\eqref{lc:domains}. Every formula code and assignment in the
collapse lifts to one fixed by $\pi$, so elementarity of set
satisfaction transfers correctness through $\xi$. Thus
$\nu\leq\mu\leq\nu$, and $\mu=\nu$.
For a canonical Skolem term $t$ and a tuple of generators $\vec a$,
\[
 \pi(t^N(\vec a))=t^{L_\nu}(\vec a)=t^N(\vec a).
\]
These values exhaust $K$, and $\pi$ is onto $N$. Hence $K=N$.
\end{proof}

For every infinite cardinal $\lambda$ of $L$, set
\[
 b_{\lambda,\lambda}=\operatorname{id}_\lambda,
 \qquad
 b_{\lambda,\eta}=\text{the $<_L$-least bijection }\lambda\longrightarrow\eta
 \quad(\lambda<\eta<\lambda^+).
\]
Define $F^\lambda_\eta$ and $f^\lambda_\eta$ by (2.1).
For a fixed upper index
$\theta$, include among the parameters the family
\begin{equation}
 \mathcal B_{\lambda,\theta}
   =\langle b_{\lambda,\eta}:\lambda\leq\eta\leq\theta\rangle,
 \label{lc:bijections}
\end{equation}
which has hereditary size at most $\lambda$; it is empty if $\theta<\lambda$.

\begin{lemma}\label{lc:canonical-collapse}
Let $N=L_\nu\models T_0$, and let $H\prec N$ contain
$\lambda,\theta,\mathcal B_{\lambda,\theta}$.  Suppose
$H\cap\lambda=\alpha$ is an infinite cardinal of $L$, and write
$\pi:H\cong L_\tau$ for the collapse.  For every
$\eta\in H\cap[\lambda,\theta]$,
\begin{equation}
 H\cap\eta=b_{\lambda,\eta}[\alpha],\quad
 \pi(\eta)=f^\lambda_\eta(\alpha),\quad
 \pi(b_{\lambda,\eta})=b_{\alpha,\pi(\eta)}.
 \label{lc:canonical-identities}
\end{equation}
Consequently the full family of canonical bijections collapses to
$\mathcal B_{\alpha,\pi(\theta)}$.  For every $\rho<\alpha$ and
every $\eta\in H\cap(\theta+1)$,
\begin{equation}
 f^\alpha_{\pi(\eta)}(\rho)=f^\lambda_\eta(\rho).
 \label{lc:pointwise-indices}
\end{equation}
\end{lemma}
\begin{proof}
For $\eta\in H\cap[\lambda,\theta]$, let $b=b_{\lambda,\eta}\in H$.
Elementarity and evaluation give
\[
 \begin{aligned}
 i<\alpha&\Longrightarrow b(i)\in H\cap\eta,\\
 \zeta\in H\cap\eta&\Longrightarrow b^{-1}(\zeta)\in H\cap\lambda=\alpha.
 \end{aligned}
\]
Thus $H\cap\eta=b[\alpha]$ and
$\pi(\eta)=\otp(b[\alpha])=f^\lambda_\eta(\alpha)$.
The collapse of $b$ is a bijection
$\alpha\to\pi(\eta)$. If $\eta=\lambda$, it is the identity.
Otherwise, $L_\tau$ regards $\pi(b)$ as the $<_L$-least such
bijection. Every $<_L$-predecessor of $\pi(b)$ lies in $L_\tau$;
bijectivity is absolute. Hence
\[
 \pi(b_{\lambda,\eta})=b_{\alpha,\pi(\eta)},\qquad
 \pi(\eta)<\alpha^+.
\]
The domain of $\pi(\mathcal B_{\lambda,\theta})$ is
$[\alpha,\pi(\theta)]$, proving the assertion about this family.
For $\rho<\alpha$,
\[
 \begin{aligned}
 b_{\alpha,\pi(\eta)}[\rho]
    &=\pi[b_{\lambda,\eta}[\rho]],\\
 f^\alpha_{\pi(\eta)}(\rho)
    &=\otp(\pi[b_{\lambda,\eta}[\rho]])
      =\otp(b_{\lambda,\eta}[\rho])=f^\lambda_\eta(\rho).
 \end{aligned}
\]
For $\eta\in H\cap\lambda$, we have $\eta<\alpha$ and
$\pi(\eta)=\eta$, so the constant-function clause gives the
same identity.
\end{proof}

To represent a tuple $q$ of hereditary size $\lambda$ by a
predicate on $\lambda$, choose the $<_L$-least bijection
$e:\lambda\to\tc(\{q\})$ with $e(0)=q$. Fix an ordinal pairing
function whose restriction to each relevant infinite cardinal
maps its square into itself, and put
\[
 R(q)=\{\langle i,j\rangle:i,j<\lambda\text{ and }e(i)\in e(j)\}.
\]
The distinguished index $0$ represents $q$. The Mostowski collapse
recovers $q$ uniquely. A Skolem hull containing $\lambda$ pointwise
and $R(q)$ as an element contains $\tc(\{q\})$ pointwise. Under
the hypotheses of Lemma~\ref{lc:canonical-collapse}, the collapse
of $e$ is again the least enumeration of the collapsed transitive
closure, by the same $<_L$-minimality argument. Therefore
\begin{equation}
 R(\pi(q))=\pi(R(q))=R(q)\cap\alpha.
 \label{lc:hereditary-code}
\end{equation}
Thus successive hull collapses represent the same restricted
parameter tuple, independently of the intervening hull.

\begin{lemma}
\label{lc:strong-composition}
In Lemma~\ref{lc:canonical-collapse}, suppose that $\lambda,\alpha$
are inaccessible and $H\cap L_\lambda=L_\alpha$, with
$L_\lambda\in H$.
Every formula $c\in H$ for which the required families of
canonical bijections belong to $H$ has a defined restriction, and
\[
 \pi(c)=c|^\lambda_\alpha,
 \qquad \pi(A)=A\cap L_\alpha
 \quad(A\in H,\ A\subseteq L_\lambda).
\]
For every regular $\rho<\alpha$ where the direct restriction is
defined,
\begin{equation}
 (c|^\lambda_\alpha)|^\alpha_\rho=c|^\lambda_\rho.
 \label{lc:strong-composition-equation}
\end{equation}
For a fixed source and parameter tuple, there is a club $B\subseteq\lambda$
such that every regular $\alpha\in B$ yields a hull with these properties.
\end{lemma}
\begin{proof}
Induct on the subformula relation. For a predicate $A\in H$ on
$L_\lambda$,
\[
 \pi(A)=\pi[A\cap H]=A\cap L_\alpha.
\]
First-order parameters belong to $H\cap L_\lambda=L_\alpha$
and are fixed. If the formula has complexity below $\lambda$,
that complexity belongs to $H\cap\lambda=\alpha$ and its
syntax has size less than $\alpha$; it is unchanged except for
the displayed predicate restrictions. Finite quantifier blocks
commute with $\pi$.

Suppose $c=\bigwedge_{\zeta<\eta}c_\zeta$ and $\eta\geq\lambda$.
Let $F=F^\lambda_\eta(\alpha)$ and
$u=\pi^\lambda_{\eta,\alpha}$. By Lemma~\ref{lc:canonical-collapse},
\[
 \pi(c)=\bigwedge_{\tau<f^\lambda_\eta(\alpha)}
                \pi(c_{u^{-1}(\tau)})
       =\bigwedge_{\tau<f^\lambda_\eta(\alpha)}
                (c_{u^{-1}(\tau)}|^\lambda_\alpha)
       =c|^\lambda_\alpha.
\]
Elementarity gives the required complexity $\tau$ of each
constituent. The disjunction case is identical with $\bigvee$.

For a further regular $\rho<\alpha$, Lemma~\ref{lc:canonical-collapse}
gives
\[
 \pi[F^\lambda_\eta(\rho)]
      =F^\alpha_{\pi(\eta)}(\rho),\qquad
 \pi^\alpha_{\pi(\eta),\rho}(\pi(\zeta))
      =\pi^\lambda_{\eta,\rho}(\zeta)
 \quad(\zeta\in F^\lambda_\eta(\rho)).
\]
The constituent with index $\zeta$ in the direct restriction
therefore has the same index in the iterated restriction.
The induction hypothesis identifies the constituent formulas,
proving \eqref{lc:strong-composition-equation}.

For the club assertion, set
$H_\alpha=\Hull^N(\alpha\cup\{q,\lambda\})$.
Choose an enumeration of $L_\lambda$ and close $\alpha$ under
its bounds, its inverse indices and the hull bounds. Continuity
gives a club $B$ such that
\[
 \alpha\in B\ \Longrightarrow\
 H_\alpha\cap\lambda=\alpha,
 \qquad H_\alpha\cap L_\lambda=L_\alpha.
\]
Include closure under $\beta\mapsto(2^\beta)^L$ below $\lambda$.
If $\alpha\in B$ is regular in $L$, then
$(2^\beta)^L<\alpha$ for all $\beta<\alpha$, so $\alpha$
is inaccessible in $L$. The preceding calculation applies to
every formula code in $H_\alpha$.
\end{proof}

For arbitrary choices of bijections, composition of restrictions holds
on the clubs specified in \cite[Proposition~5.12]{CodyHigher}.

The next lemma adapts \cite[Lemmas~3.11 and~3.37]{BW} to the
hierarchy and restrictions of \cite{CodyHigher}, for $\xi<\delta^+$.

\begin{lemma}\label{lc:exact-trace}
Suppose $\delta$ is $\Pi^1_\xi$-indescribable, $\xi\geq1$, and
$N$ is a level of size $\delta$ correct through $\xi$ for a tuple
$p$ containing the required canonical bijections. Let
\[
 H_\alpha=\operatorname{Hull}^N
             (\alpha\cup\{\delta,\xi,p\}),\qquad
 \pi_\alpha:H_\alpha\cong M_\alpha.
\]
Let $B_N$ be the club of limit ordinals $\alpha<\delta$ such that
\[
 H_\alpha\cap\delta=\alpha,\qquad
 H_\alpha\cap L_\delta=L_\alpha,\qquad
 (\forall\eta<\alpha)\ (2^\eta)^L<\alpha.
\]
The set
\begin{equation}
 Q_\xi(N)=\{\alpha\in B_N:\alpha\text{ is regular and }
 M_\alpha\text{ is correct through }f^\delta_\xi(\alpha)
                      \text{ for }\pi_\alpha(p)\}
 \label{lc:trace-definition}
\end{equation}
is a $\Pi^1_\xi$-club. Its closure implication holds at every
regular cardinal in $B_N$.
\end{lemma}
\begin{proof}
Enumerate the true formula--assignment pairs in $N$ of complexity
at most $\xi$ as $\langle(c_i,a_i):i<\delta\rangle$, including
the negation of each false instance. For each $i$, choose a true
$\Pi^1_{\eta_i}$ instance $d_i$, $\eta_i\leq\xi$, and a
club $B_i\subseteq\delta$ in $L$ such that
\[
 \operatorname{Ref}_\delta(d_i)\cap B_i
                 \subseteq\operatorname{Ref}_\delta(c_i[a_i]).
\]
For a universal instance take the instance itself. For a true
existential successor, use a witnessing tuple in $N$ and its
universal matrix. For a limit disjunction, choose a true
constituent $c_\zeta$ and regard it as $\Pi^1_{\zeta+1}$;
$\zeta+1$ is below the limit. The clubs retain these parameters
and give the required restriction identities. At level zero use
the first-order reflection club.

Each $\operatorname{Ref}_\delta(d_i)$ is
$(I_\xi^L(\delta))^*$-large. Let $B\subseteq B_N$ also ensure that each
formula--assignment pair in $H_\alpha$ has index below $\alpha$
and that its chosen witnesses belong to $H_\alpha$. Normality gives
\[
 D=B\cap\triangle_{i<\delta}
       (B_i\cap\operatorname{Ref}_\delta(d_i))
       \in(I_\xi^L(\delta))^*.
\]
For $\alpha\in D$, elementarity of set satisfaction and
Lemma~\ref{lc:strong-composition} show, for each formula code
$c$ in the collapsed domain and each internal assignment $a$,
\[
 \operatorname{Sat}(L_\alpha,S_{M_\alpha},c,a)
   \iff (L_\alpha,\mathcal P^L(L_\alpha))\models c[a].
\]
Hence $D\subseteq Q_\xi(N)$; in particular the latter has all
the lower positivities required of a $\Pi^1_\xi$-club.

For closure, let $\alpha\in B_N$ be a regular cardinal such that
$Q_\xi(N)\cap\alpha$ has the lower positivities required in
Definition~5.16 of \cite{CodyHigher}. We prove correctness of $M_\alpha$ by induction
on $\theta\leq f^\delta_\xi(\alpha)$. Level zero follows from
absoluteness of first-order set satisfaction. At a successor
$\theta=\zeta+1$, lower correctness implies
\[
 L_\alpha\models\forall X\,\chi(X,a)
   \ \Longrightarrow\
 (L_\alpha,S_{M_\alpha})\models\forall X\,\chi(X,a).
\]
Suppose the converse fails and choose $X\in L$ with
$L_\alpha\models\neg\chi(X,a)$, where $\neg\chi\in\Pi^1_\zeta$.
The lower positivity reflects this instance to
$\gamma\in Q_\xi(N)\cap\alpha$, in its restriction club.
Write the internal code and $a$ as finitely many Skolem terms
in ordinal arguments $\vec b\in\alpha^{<\omega}$, and require
$\gamma>\max(\vec b)$. The hull inclusion gives an elementary
map $M_\gamma\to M_\alpha$. It transfers the internally true
universal statement to $M_\gamma$, while correctness of $M_\gamma$
and Lemma~\ref{lc:strong-composition} give
\[
 L_\gamma\models(\forall X\,\chi(X,a))|^\alpha_\gamma,
 \qquad
 L_\gamma\models(\neg\chi(X,a))|^\alpha_\gamma,
\]
a contradiction. Negation proves the existential case.
At a limit, each constituent belongs to the transitive model
$M_\alpha$, and the induction hypothesis identifies its truth
value in both semantics. It therefore identifies the conjunction
or disjunction. Thus $\alpha\in Q_\xi(N)$.
\end{proof}

The uniform defining formula and restriction club in the next lemma
are given by \cite[Theorem~5.17(2)]{CodyHigher}.

\begin{lemma}\label{lc:transported-tests}
Suppose $\kappa$ is $\Pi^1_\xi$-indescribable. Fix the
$\Pi^1_\xi$ formula $\chi^\kappa_\xi(X)$ defining the property of being a higher club,
and include its auxiliary parameters in a fixed constructible tuple.  There
is a constructible club $D_\xi$ such that for every regular $\gamma\in D_\xi$
and every constructible $C\subseteq\gamma$,
\begin{equation}
 L_\gamma\models(\chi^\kappa_\xi)|^\kappa_\gamma(C)
 \quad\Longleftrightarrow\quad
 C\text{ is a }\Pi^1_{f^\kappa_\xi(\gamma)}\text{-club in }L.
 \label{lc:uniform-club-test}
\end{equation}
The closure implication at every relevant regular cardinal is first order
over an inaccessible $L_\gamma$, using the code
\eqref{lc:hereditary-code} for the family of canonical bijections.
It can therefore be required in addition without increasing this
complexity.
\end{lemma}
\begin{proof}
Apply \cite[Theorem~5.17(2)]{CodyHigher} inside $L$. It
gives one club $D_\xi$ such that, for every regular
$\gamma\in D_\xi$ and every $A\subseteq\kappa$ in $L$,
\[
 L_\gamma\models\chi^\kappa_\xi(A)|^\kappa_\gamma
 \quad\Longleftrightarrow\quad
 A\cap\gamma\text{ is a }\Pi^1_{f^\kappa_\xi(\gamma)}
                       \text{-club in }\gamma.
\]
For a local $C\subseteq\gamma$ take $A=C$. The auxiliary
parameters are the restrictions of the fixed global parameters,
as represented in \eqref{lc:hereditary-code}.

For the closure assertion, work in an inaccessible $L_\gamma$.
For inaccessible $\alpha<\gamma$ and $\eta<(\alpha^+)^L$,
the code domain, $\mathcal P^L(L_\alpha)$ and their satisfaction
relations all belong to $L_\gamma$. Consequently the test
\[
 Q\cap\alpha\in(I_\eta^L(\alpha))^+
 \quad\Longleftrightarrow\quad
 \forall\varphi\in\Pi^1_\eta\,
 [L_\alpha\models\varphi\ \to\
           Q\cap\alpha\cap\operatorname{Ref}_\alpha(\varphi)
                              \ne\varnothing]
\]
is first order over $L_\gamma$ when its code and predicate
quantifiers are written over those sets. The canonical
bijections determine the indices $\eta$ by
\eqref{lc:canonical-identities}. Substituting these tests
into the successor and limit closure implications in
Section~\ref{sec:notation} gives the required first-order
formula for closure at every relevant local regular cardinal.
\end{proof}

If $B$ is a club in $L$ witnessing the closure condition for $Q$,
replace $Q$ by $Q\cap B$,
also intersecting $B$ with the clubs ensuring the restriction
identities. A positive restriction of $Q\cap B$ to $\alpha$
is unbounded, so $\alpha\in B$; the original closure condition
then gives $\alpha\in Q\cap B$. Thus the closure implication
holds at every relevant regular cardinal. These subsets still
form a base for the same higher club filter.

\section{Coherent witnesses and strongly amenable sequences}
\label{lc:coherent-witnesses}

Let $\kappa$
be regular uncountable in $V$ and inaccessible in $L$.  Fix
$1\leq\xi<(\kappa^+)^L$ and put
$\xi_\beta=f^\kappa_\xi(\beta)$.  A fixed constructible club of closure points
$C$ ensures the canonical-function identities and parameter
restrictions of Lemma~\ref{lc:strong-composition}.  For $\beta=\kappa$, set $\xi_\kappa=\xi$.

Assume that $P\subseteq C$ consists of $V$-regular, $L$-inaccessible
cardinals and that the following two properties hold.  First, whenever
$\beta\in C$ is regular in $V$, inaccessible in $L$, and
$P\cap\beta$ is stationary in $V$,
\begin{equation}
 X\in I^L_{\xi_\beta}(\beta)
 \quad\Longrightarrow\quad
 P\cap\beta\cap X\text{ is nonstationary in }V.
 \tag{$R_\xi$}\label{lc:R}
\end{equation}
We assume the same assertion at $\beta=\kappa$.  Second, for
$\delta\in P\cap\TrV(P)$ and every constructible
$\Pi^1_{\xi_\delta}$-club $Q\subseteq\delta$, there is a
club $\Phi_\delta(Q)\subseteq\delta$ in $V$ such that
\begin{equation}
 \gamma\in P\cap\operatorname{acc}(\Phi_\delta(Q))
 \quad\Longrightarrow\quad\gamma\in Q.
 \tag{$C_\xi$}\label{lc:club-in-V}
\end{equation}
The first property, applied to $X=\beta$ if the ideal is
improper, implies that every such $\beta$ with stationary
$P\cap\beta$ is $\Pi^1_{\xi_\beta}$-indescribable in $L$.

Write a local statement of complexity $\Pi^1_{\xi_\delta+1}$ as
\[
 \psi_\delta(Z,U_\delta)
       =\forall\vec X\,\chi_\delta(\vec X,Z,U_\delta),
 \qquad\chi_\delta\in\Sigma^1_{\xi_\delta}.
\]
If it is false in $L$, let
\[
 X_\delta(Z)=\min_{<_L}\{\vec X\in
       (\mathcal P^L(L_\delta))^{<\omega}:
       L_\delta\models\neg\chi_\delta(\vec X,Z,U_\delta)\}.
\]
Let $p_\delta(Z)$ be the tuple consisting of $\delta,\xi_\delta$,
this witness, the parameters, the formula codes and the required
canonical bijections, and put
\[
 \begin{aligned}
 \nu_\delta(Z)&=\min\{\nu:L_\nu\text{ is correct through }
                          \xi_\delta\text{ for }p_\delta(Z)\},\\
 N_\delta(Z)&=L_{\nu_\delta(Z)},\qquad
 Q_\delta(Z)=Q_{\xi_\delta}(N_\delta(Z)).
 \end{aligned}
\]
Intersect the last set with the constructible clubs ensuring the
restriction identities. At the relevant ordinals,
$f^\delta_{\xi_\delta}(\gamma)=\xi_\gamma$.

The next lemma adapts \cite[proof of Theorem~3.2, Section~3.2]{BW}
to the canonical restrictions and expansions by external predicates.

\begin{lemma}\label{lc:coherent-minima}
At every $\gamma\in Q_\delta(Z)$, the collapsed witness and
level are respectively
\begin{equation}
 X_\gamma(Z\cap L_\gamma),\qquad
 N_\gamma(Z\cap L_\gamma).
 \label{lc:minima-identity}
\end{equation}
The elementary maps between two such lower levels are independent of
the larger source level, whenever their collapsed parameter tuples agree, and they
compose.

The same conclusions hold for the least constructible candidate satisfying a
uniform test of complexity at most $\xi_\delta$, including the
restricted formula defining higher clubs with the stated closure requirement, together with first-order requirements in
finitely many external predicates, provided that the quantifiers of
these additional requirements are restricted to $L_\delta$ or $\delta$, the hull is elementary in that expansion, and
each predicate restricts to its intersection with the smaller domain.
\end{lemma}
\begin{proof}
Let $\pi:H\cong M$ be the collapse at $\gamma\in Q_\delta(Z)$.
By Lemma~\ref{lc:strong-composition},
\[
 \pi(p_\delta(Z))=p_\gamma(Z\cap L_\gamma),\qquad
 \pi(\chi_\delta)=\chi_\gamma.
\]
The assertion that the witness is $<_L$-least is first order when
written using $\operatorname{Sat}$. Elementarity transfers it to $M$,
and correctness of $M$ gives
\[
 L_\gamma\models
   \neg\chi_\gamma(\pi(X_\delta(Z)),Z\cap L_\gamma,U_\gamma).
\]
If $Y<_L\pi(X_\delta(Z))$ were a witness in $L$, initiality of
$<_L$ would give $Y\in M$, and correctness would contradict
internal minimality. Hence
$\pi(X_\delta(Z))=X_\gamma(Z\cap L_\gamma)$.

The first-order comparison test \eqref{lc:comparison} also transfers.
A smaller correct level belongs to $M$ and passes that test there,
contrary to transferred minimality. Thus
$M=N_\gamma(Z\cap L_\gamma)$, proving
\eqref{lc:minima-identity}. With external predicates, the same
argument applies to the first-order requirements because the hull
is elementary in the expanded language and its predicates restrict
to their actual intersections with $L_\gamma$.

Write $p_\gamma$ for the resulting tuple. By
Lemma~\ref{lc:skolem-hull},
\[
 N_\gamma=\Hull^{N_\gamma}(\gamma\cup\{p_\gamma\}).
\]
Every induced map is therefore determined by
\[
 j_{\gamma\gamma'}(t(\vec a,p_\gamma))
       =t(\vec a,p_{\gamma'})\qquad(\vec a\in\gamma^{<\omega}).
\]
This formula is independent of the larger source level. It also
gives $j_{\gamma'\gamma''}\circ j_{\gamma\gamma'}=j_{\gamma\gamma''}$.
\end{proof}

The expanded hulls required in this lemma agree with the original
hulls on a club subset of $\delta$ in $V$. Indeed, let
\[
 H_\alpha=\Hull^N(\alpha\cup\{p\}),\qquad
 H'_\alpha=\Hull^{(N,\vec A)}(\alpha\cup\{p\}),
\]
where $N=\Hull^N(\delta\cup\{p\})$ and $\delta$ is regular
in $V$. Both filtrations are continuous, have union $N$, and
have size less than $\delta$ at stages below $\delta$.
Choose $f:\delta\to\delta$ with
$H'_\alpha\subseteq H_{f(\alpha)}$. At a limit closure point
$\gamma$ of $f$,
\[
 H_\gamma\subseteq H'_\gamma
   =\bigcup_{\alpha<\gamma}H'_\alpha
   \subseteq\bigcup_{\alpha<\gamma}H_{f(\alpha)}
   \subseteq H_\gamma.
\]
Thus $H_\gamma=H'_\gamma$. This club belongs to $V$; it is
included explicitly in the entries of the sequence and is not
assumed to be constructible in applications of (O).

\begin{lemma}\label{lc:constructible-limit}
Suppose a stationary set $H\subseteq\beta$, with
$\operatorname{cf}^V(\beta)>\omega$, carries the local minima of
Lemma~\ref{lc:coherent-minima}.  Suppose every finite collection of
indices has a common transitive source level realizing the canonical maps,
and common source levels exist above every bounded collection of indices.
The directed limit is well founded.  The common image of the distinguished $L_\gamma$
is $L_\beta$, and a distinguished predicate whose local restrictions agree
has limit equal to their union.  This limit predicate belongs to $L$.
\end{lemma}
\begin{proof}
Use the equivalence relation and membership relation from
Lemma~\ref{fd:limit}. Coherence follows from
Lemma~\ref{lc:coherent-minima}. A descending $\omega$-sequence
uses countably many representatives and comparison indices.
These are bounded in $\beta$; by hypothesis one larger transitive
source level realizes all of them. Their images would form a
descending membership sequence there. Thus the direct limit is
well-founded.

Let $j_\gamma:N_\gamma\to N$ be the maps to its transitive
collapse, and let $U$ be the common image of the distinguished
$L_\gamma$. Each $j_\gamma$ fixes $L_\gamma$ pointwise, so
\[
 U=\bigcup_{\gamma\in H}L_\gamma=L_\beta.
\]
If $X$ is the common image of the distinguished predicates,
the same calculation gives
\[
 X\cap L_\gamma=X_\gamma,
 \qquad X=\bigcup_{\gamma\in H}X_\gamma.
\]
The reduct of $N$ to the language of set theory satisfies $T_0$.
Hence $N=L_\eta$ for some $\eta$, and in particular $X\in L$.
\end{proof}

\begin{theorem}\label{lc:strong-amenability}
Assume \eqref{lc:R} and \eqref{lc:club-in-V}.  Let $\beta$
be an ordinal to which \eqref{lc:R} applies, with $P\cap\beta$
stationary.  Let $\psi(A,U)$ be any constructible
$\Pi^1_{\xi_\beta+1}$ statement true over $L_\beta$.
There is a constructible club of closure points $A_\psi\subseteq\beta$ such that,
on every domain
\[
 B\subseteq P\cap\TrV(P)\cap A_\psi\cap\beta
\]
where $\psi|^\beta_\delta(A\cap L_\delta,U\cap L_\delta)$
is false, the entries
\begin{equation}
 K_\delta=\Phi_\delta
  \bigl(Q_\delta(A\cap L_\delta)\cap A_\psi\bigr)
 \qquad(\delta\in B)
 \label{lc:strong-amenability-entries}
\end{equation}
form a strongly amenable sequence.  In the argument of $\Phi_\delta$,
$A_\psi$ means $A_\psi\cap\delta$. The domain $B$ belongs to $V$;
its local codes, parameters, witnesses and traces belong to $L$.
Finitely many further clubs in $\delta$ computed in $L$ may be
intersected with $K_\delta$ to ensure the restriction identities.
\end{theorem}
\begin{proof}
Choose a constructible club $A_\psi\subseteq\beta$, closed
under the hull bounds for the code, parameters and canonical
bijections of $\psi$, such that
\[
 f^\beta_{\xi_\beta}(\delta)=\xi_\delta,
 \qquad (\psi|^\beta_\delta)|^\delta_\gamma
                    =\psi|^\beta_\gamma
\]
at the relevant regular cardinals. Replace $A_\psi$ by
$\acc(A_\psi)$. Lemma~\ref{lc:exact-trace} makes
$Q_\delta(A\cap L_\delta)\cap A_\psi$ a constructible
$\Pi^1_{\xi_\delta}$-club in $\delta$.

Suppose $F$ is club in $\beta$ and
$J=\{\delta\in B:F\cap\delta\subseteq K_\delta\}$ is unbounded.
Let $H=P\cap\beta\cap\acc(F)$, which is stationary. For
$\gamma\in H$ and $\delta\in J\setminus(\gamma+1)$,
\[
 \gamma\in P\cap\acc(K_\delta)
 \ \Longrightarrow\
 \gamma\in Q_\delta(A\cap L_\delta)\cap A_\psi
\]
by \eqref{lc:club-in-V}. Lemmas~\ref{lc:coherent-minima}
and~\ref{lc:constructible-limit} yield $X\in L$ with
\begin{equation}
 X\cap L_\gamma=X_\gamma(A\cap L_\gamma)
                  \qquad(\gamma\in H).
 \label{lc:witness-restrictions}
\end{equation}
For a tuple of predicates the intersection is taken componentwise.

Write $\psi=\forall\vec X\,\chi(\vec X,A,U)$, where
$\chi\in\Sigma^1_{\xi_\beta}$. Since $X\in L$ and
$L_\beta\models\psi$, we have $L_\beta\models\chi(X,A,U)$.
Moreover
\[
 \beta\setminus\operatorname{Ref}_\beta(\chi(X,A,U))
                         \in I^L_{\xi_\beta}(\beta).
\]
Indeed, at a successor choose existential witnesses and use the
lower universal matrix; at a limit choose a true constituent of
lower complexity and include its restriction club. At level zero
use the first-order reflection club.
By \eqref{lc:R}, some $\gamma\in H$ reflects this instance.
Composition and \eqref{lc:witness-restrictions} then give
\[
 \begin{aligned}
 L_\gamma&\models
    \chi_\gamma(X_\gamma(A\cap L_\gamma),A\cap L_\gamma,U\cap L_\gamma),\\
 L_\gamma&\models
    \neg\chi_\gamma(X_\gamma(A\cap L_\gamma),A\cap L_\gamma,U\cap L_\gamma),
 \end{aligned}
\]
where the second line is the defining property of the least
witness. This contradiction proves bounded capture. Intersecting
an entry with finitely many additional clubs in its index ordinal in $V$ preserves
bounded capture.
\end{proof}

\section{The consistency lower bound: a transfinite induction in \texorpdfstring{$L$}{L}}
\label{lb:section}

In this section every indescribability ideal, canonical function,
infinitary formula and higher club is computed in $L$. Stationarity without a superscript, clubs used for capture, and
the ideals $\A_\kappa$ and $\SA_\kappa$ are computed in $V$.

\begin{theorem}\label{lb:relative-bound}
Suppose that $\kappa$ is regular and uncountable in $V$, that
$\NAC(\kappa)$ holds, and that there is a set $T\in L\cap\mathcal P(\kappa)$
such that
\begin{equation}\label{lb:T}
 T\notin\A_\kappa^V
 \quad\text{and}\quad
 L\models\text{``$T$ is not an ineffable subset of $\kappa$''}.
\end{equation}
For every $1\leq\xi<(\kappa^+)^L$, put
\[
 W_\xi=\{\beta<\kappa:
 L\models\text{``$\beta$ is
 $\Pi^1_{f^\kappa_\xi(\beta)}$-indescribable''}\}.
\]
Then
\[
 I^L_\xi(\kappa)\subseteq\SA_\kappa^V
 \qquad\text{and}\qquad
 \kappa\setminus W_\xi\in\SA_\kappa^V.
\]
Consequently $L\models\mathsf H(\kappa)$.
\end{theorem}

The auxiliary sets and pseudointersections are chosen in $V$.
Replacing $T$ by $T\cap\operatorname{Reg}^L(\kappa)$ preserves
\eqref{lb:T}: the deleted set is contained in the complement of the
$V$-regular cardinals and is therefore strongly amenable. We make this
replacement before applying the $\diamondsuit^*$ construction on $T$.

Write $\xi_\beta=f^\kappa_\xi(\beta)$ for $\beta<\kappa$ and
$\xi_\kappa=\xi$. We shall choose
$P_\xi\in(\SA_\kappa^V)^*$. By Lemma~\ref{fd:SA}, this is
a proper normal filter, since
$\SA_\kappa^V\subseteq\A_\kappa^V$ and $\NAC(\kappa)$.

By Lemma~\ref{fd:SA}, if $P\in(\SA_\kappa^V)^*$ and $D$ is
club in $\kappa$ in $V$, then
$P\cap\TrV(P)\cap D\in(\SA_\kappa^V)^*$.

Lemmas~\ref{fd:dominator}, \ref{fd:robust} and~\ref{fd:coherent} give a constructible diamond-star sequence on $T$, an eventual
dominator for the constructible functions, and a club $D_0\subseteq\kappa$ together
with a dual set $P_1$ of regular cardinals of $V$ that are inaccessible
in $L$. The set may be chosen so that
\begin{equation}\label{lb:clubs-in-L}
 D_0\cap\gamma\subseteq^* B
 \quad\text{whenever }\gamma\in P_1
 \text{ and }B\text{ is club in }\gamma\text{ in }L.
\end{equation}
Lemma~\ref{fd:relativeone} gives the relative reflection statement
at rank one.

At rank $\xi\geq1$ we maintain a constructible club of closure points
$C_\xi$ and a dual set $P_\xi\subseteq P_1\cap W_\xi\cap C_\xi$.
The relative reflection assertion is
\begin{equation}\label{lb:relative}
 \begin{gathered}
 \beta\in C_\xi,\quad
 \beta\text{ regular in }V\text{ and inaccessible in }L,\\
 P_\xi\cap\beta\text{ stationary},\quad
 X\in I^L_{\xi_\beta}(\beta)
 \quad\Longrightarrow\quad
 P_\xi\cap\beta\cap X\text{ is nonstationary}.
 \end{gathered}
 \tag{$R_\xi$}
\end{equation}
We also maintain its global version at $\beta=\kappa$, with rank $\xi$.
Since $P_\xi$ is dual and nonstationary sets are strongly amenable,
this global version gives
\begin{equation}\label{lb:ideal-from-relative}
 I^L_\xi(\kappa)\subseteq\SA_\kappa^V.
\end{equation}
In particular, the constructible ideal is proper, so $\kappa$ is
$\Pi^1_\xi$-indescribable in $L$.

The other inductive data are clubs $F_\xi,E_\xi\subseteq\kappa$ in $V$
and an auxiliary set $S_\xi\subseteq\kappa$ in $V$, with
\begin{equation}\label{lb:auxiliary-data}
 P_\xi\subseteq S_\xi\cap\operatorname{acc}(F_\xi),
 \qquad E_\xi\subseteq F_\xi\cap D_0.
\end{equation}
The auxiliary sets define the following clubs in $V$. If
$\delta\in P_\xi\cap\TrV(P_\xi)$ and $Q$ is a constructible
$\Pi^1_{\xi_\delta}$-club in $\delta$, set
\begin{equation}\label{lb:club-in-V}
 \Phi^\xi_\delta(Q)=
 \operatorname{cl}_\delta(F_\xi\cap\delta\cap S_\xi\cap Q)
       \cap B_\delta(Q),
\end{equation}
where $B_\delta(Q)$ is a club in $\delta$ in $L$ on which the required
closure and restriction identities hold. We require that
$\Phi^\xi_\delta(Q)$ be club in $\delta$ in $V$ and that
\begin{equation}\label{lb:recognition}
 \gamma\in P_\xi\cap\operatorname{acc}(\Phi^\xi_\delta(Q))
       \quad\Longrightarrow\quad\gamma\in Q.
\end{equation}
Finitely many further clubs in $\delta$ computed in $L$ may be
intersected with $B_\delta(Q)$.

The club $E_\xi$ will satisfy
\begin{equation}\label{lb:global-containment}
 E_\xi\cap P_\xi\subseteq^* C
 \quad\text{for every constructible $\Pi^1_\xi$-club $C$ in $\kappa$}.
 \tag{$G_\xi$}
\end{equation}
After a further dual refinement, its local counterpart will hold at
every $\gamma\in P_{\xi+1}$:
\begin{equation}\label{lb:local-containment}
 E_\xi\cap\gamma\cap P_\xi\subseteq^* C
 \quad\text{for every constructible $\Pi^1_{\xi_\gamma}$-club $C$ in $\gamma$}.
 \tag{$L_\xi$}
\end{equation}
The sets and clubs are decreasing modulo bounded sets. To ensure that
the bounds needed at an ordinal $\beta$ lie below $\beta$, refine by
a club in $V$ as follows. Suppose $u_\zeta^P,u_\zeta^E<\kappa$ bound
the previous almost-inclusions used at a stage $\xi\geq\kappa$.
Using $b=b_{\kappa,\xi}$, close the new set under
\[
 h_\xi(i)=\max\{u_{b(i)}^P,u_{b(i)}^E\}+1,\qquad
 J_\xi=\{\beta<\kappa:\beta\text{ is limit and }h_\xi''\beta\subseteq\beta\}.
\]
For $\beta\in J_\xi$ and $i<\beta$, both bounds at the index $b(i)$
lie below $\beta$. If the refined set is stationary in $\beta$, then
$\beta\in J_\xi$ by closedness. The club $J_\xi$ need not belong to
$L$ and is not used in \eqref{lb:clubs-in-L}. When $\xi<\kappa$,
work above $\xi$ and a common bound for the fewer than $\kappa$ requirements.

For each earlier rank, choose in $L$ the $<_L$-least formula
codes, witnesses for the restriction identities, and
clubs with the required closure properties. The canonical
bijections are fixed independently. If the relevant ideal is
improper, use fixed default values for the higher-club test;
such a test is used only after properness has been proved.
For each $\theta<(\kappa^+)^L$, Replacement in $L$ gives the
sequence of these choices through $\theta$. Its diagonal
closure club is therefore constructible. The sets $P_\eta$,
clubs $E_\eta$ and their inclusion bounds, which may belong
only to $V$, are not used to make these choices.
We prove $R_\xi$ using the canonical bijections and restriction
identities; only after \eqref{lb:ideal-from-relative} establishes
indescribability do we use the corresponding formula from
\cite[Theorem~5.17]{CodyHigher}.

We also refine the constructible clubs of closure points by the limit points of the
constructible inaccessibles below $\kappa$. This is a constructible club, since
Lemma~\ref{fd:robust} gives weak compactness in $L$. Consequently every
ordinal in such a club that is regular in $L$ is inaccessible in $L$.
When verifying a higher club's closure in $L$, we will consider all
these regular cardinals of $L$, including ones that are singular in $V$.

For the induction put $P_0=\kappa$ and $E_0=D_0$. The relative assertion
$R_0$ is the preservation of a constructible club witnessing nonstationarity,
and $L_0$ is~\eqref{lb:clubs-in-L}. At rank one take $F_1=D_0$ and
$S_1=\kappa$. To verify~\eqref{lb:club-in-V}, $R_1$ makes
$D_0\cap Q$ cofinal. If $\gamma\in P_1\cap\acc(D_0\cap Q)$ and
$Q\cap\gamma$ were nonstationary in $L$, a constructible club
disjoint from it would contradict~\eqref{lb:clubs-in-L}. The constructible closure
factor then gives $\gamma\in Q$.

\begin{lemma}\label{lb:global-lemma}
Once $R_\xi$ and the rank-$\xi$ auxiliary data have been constructed, there is
a club $E_\xi\subseteq F_\xi\cap D_0$ in $V$ satisfying $G_\xi$.
\end{lemma}

\begin{proof}
For $\beta\in T\cap P_\xi\cap\TrV(P_\xi)$, enumerate
\[
 \{Q\in\mathcal D_\beta:Q\text{ is a constructible }
                  \Pi^1_{\xi_\beta}\text{-club in }\beta\}
       =\{Q_i^\beta:i<\mu_\beta\},\qquad \mu_\beta\leq\beta.
\]
By $R_\xi$ and \cite[Theorem~5.17(1)]{CodyHigher}, choose
clubs $B_i^\beta\subseteq\beta$ in $V$ with
$P_\xi\cap B_i^\beta\subseteq Q_i^\beta$. Complete the
family with $B_i^\beta=\beta$, and put
\[
 R_\beta=\triangle_{i<\beta}B_i^\beta,
 \qquad (R_\beta\cap P_\xi)\setminus(i+1)\subseteq Q_i^\beta.
\]
The domain remains $\A_\kappa^V$-positive. Hence a club $D$
and a stationary set $H$ satisfy
\[
 H\subseteq T\cap P_\xi\cap\TrV(P_\xi),\qquad
 D\cap\beta\subseteq R_\beta\quad(\beta\in H).
\]

Fix a constructible $\Pi^1_\xi$-club $C\subseteq\kappa$.
Let $\chi_\xi(C)$ be its uniform defining formula. Then
\[
 X_C=\kappa\setminus\operatorname{Ref}_\kappa(\chi_\xi(C))
                      \in I_\xi^L(\kappa).
\]
By $R_\xi$, $P_\xi\cap X_C$ is nonstationary. Intersect $H$
with a club avoiding this set, the uniform restriction club of
\cite[Theorem~5.17(2)]{CodyHigher}, and the
$\diamondsuit^*$ guessing club for $C$. On the resulting
stationary set $H_C$, there is $u_\beta<\beta$ such that
\[
 (D\cap\beta\cap P_\xi)\setminus(u_\beta+1)\subseteq C\cap\beta.
\]
Fodor's lemma gives stationary $H'_C\subseteq H_C$ and a fixed
$u<\kappa$ with $u_\beta=u$. Since $H'_C$ is unbounded,
$(D\cap P_\xi)\setminus(u+1)\subseteq C$. The same $D$
works for every $C$. Set $E_\xi=D\cap F_\xi\cap D_0$.
\end{proof}

\begin{lemma}\label{lb:bad-ordinals}
Among the ordinals in $P_\xi\cap\TrV(P_\xi)$, the failures
of $L_\xi$ form a member of $\SA_\kappa^V$.
\end{lemma}

\begin{proof}
By the argument at the end of Section~\ref{lc:section}, intersecting
a higher club with a club in $L$ witnessing its closure gives a higher
club whose closure implication holds at every relevant regular cardinal.
Passing to this subset preserves failure of almost inclusion.

Let
\[
 U=\{\delta\in P_\xi\cap\TrV(P_\xi):
       \exists C\in L\,[C\text{ is a }\Pi^1_{\xi_\delta}
       \text{-club in }\delta\ \wedge\
       \sup((E_\xi\cap\delta\cap P_\xi)\setminus C)=\delta]\}.
\]
For $\delta\in U$, choose the $<_L$-least such $C_\delta$
with the closure property just specified, and the least correct
level containing its parameters. Expand this level by
$E_\xi\cap\delta$ and $P_\xi\cap\delta$. Let $A_\delta$ be
the club in $\delta$ in $V$ where the two hull filtrations agree, and let
$Q_\delta$ be its constructible correctness trace. Use
\[
 K_\delta=\Phi^\xi_\delta(Q_\delta)\cap A_\delta.
\]
Suppose $J=\operatorname{Cap}(D,\vec K)$ is unbounded.
The set $H=P_\xi\cap\TrV(P_\xi)\cap\acc(D)$ is stationary.
For $\alpha\in H$ choose $\delta\in J\setminus(\alpha+1)$.
Then
\[
 \alpha\in\acc(K_\delta)\cap P_\xi
 \ \Longrightarrow\ \alpha\in Q_\delta\cap A_\delta.
\]
The uniform higher-club test is correct at this restriction.
The other defining condition is the first-order assertion
\[
 \forall u<\delta\ \exists\gamma\in(u,\delta)\
        (\gamma\in E_\xi\cap P_\xi\ \wedge\ \gamma\notin C_\delta),
\]
which transfers in the expanded hull. Lemmas
\ref{lc:coherent-minima} and~\ref{lc:constructible-limit} give
$C\in L\cap\mathcal P(\kappa)$ such that
\[
 C\cap\alpha=C_\alpha\qquad(\alpha\in H).
\]

We verify the required lower positivities of $C$ in $L$. If
$\xi=\eta+1$, let $L_\kappa\models\varphi\in\Pi^1_\eta$.
By $R_\eta$ and $P_\xi\subseteq^*P_\eta$, choose
$\alpha\in H\cap\operatorname{Ref}_\kappa(\varphi)$ in
the clubs ensuring the composition identities. Since
$C_\alpha\in(I_{f^\kappa_\eta(\alpha)}^L(\alpha))^+$,
there is a further reflector $\rho\in C_\alpha$ in the
composition club. Hence
\[
 L_\rho\models(\varphi|^\kappa_\alpha)|^\alpha_\rho
                    =\varphi|^\kappa_\rho,
 \qquad \rho\in C.
\]
Thus $C\in(I_\eta^L(\kappa))^+$. If $\xi$ is a limit,
the same argument applies separately to each $\eta<\xi$,
choosing $\alpha$ with $\eta\in F^\kappa_\xi(\alpha)$.
The level-zero case uses first-order reflection in $L$.

For closure, let $\rho$ be regular in $L$ in the global
restriction club, and suppose $C\cap\rho$ has the lower
positivities required by Definition~5.16 of \cite{CodyHigher}.
This $\rho$ is inaccessible in $L$. Choose $\alpha\in H$
above $\rho$. Lemma~\ref{lc:strong-composition} identifies
the local indices at $\rho$, and
\[
 C\cap\rho=C_\alpha\cap\rho
 \ \Longrightarrow\ \rho\in C_\alpha\subseteq C
\]
by the closure of $C_\alpha$. Thus $C$ is a constructible
$\Pi^1_\xi$-club. But for every $\alpha\in H$,
\[
 \sup((E_\xi\cap\alpha\cap P_\xi)\setminus C)=\alpha,
\]
so $\sup((E_\xi\cap P_\xi)\setminus C)=\kappa$, contrary to
$G_\xi$. Therefore $\vec K$ is strongly amenable and
$U\in\SA_\kappa^V$.
\end{proof}

Remove the ordinals where $L_\xi$ fails and refine to a dual preliminary set
\[
 P_{\xi+1}\subseteq
 P_\xi\cap\TrV(P_\xi)
 \cap\operatorname{acc}(E_\xi\cap P_\xi)
\]
on which $L_\xi$ and the required restriction identities hold. Put
$D_\xi=\operatorname{cl}_\kappa(E_\xi\cap P_\xi)$.

\begin{lemma}\label{lb:one-bound}
For every $\delta\in P_{\xi+1}$ and every correctness trace $Q_\delta$
used in the rank-$\xi$ construction,
\[
 D_\xi\cap\delta\subseteq^*\Phi^\xi_\delta(Q_\delta).
\]
The assertion also holds when the entry includes finitely many
additional clubs in $\delta$ computed in $L$ ensuring the restriction identities.
For the application of Theorem~\ref{lc:strong-amenability}, replace
$Q_\delta$ by $Q_\delta\cap A_\psi$ before applying $\Phi^\xi_\delta$,
as in \eqref{lc:strong-amenability-entries}.
\end{lemma}

\begin{proof}
By $L_\xi$, choose $u<\delta$ such that
\[
 (E_\xi\cap\delta\cap P_\xi)\setminus(u+1)
             \subseteq F_\xi\cap\delta\cap S_\xi\cap Q_\delta.
\]
Taking closure above $u+1$ gives
\[
 (D_\xi\cap\delta)\setminus(u+1)
   \subseteq\operatorname{cl}_\delta
              (F_\xi\cap\delta\cap S_\xi\cap Q_\delta).
\]
Moreover, $E_\xi\subseteq D_0$ and $D_0$ is closed, so
$D_\xi\subseteq D_0$. For the finite intersection $B$ of all
clubs in $\delta$ computed in $L$ used as factors, \eqref{lb:clubs-in-L} gives
$v<\delta$ with $(D_0\cap\delta)\setminus(v+1)\subseteq B$.
Thus $\max\{u,v\}+1$ bounds the required inclusion.
\end{proof}

\begin{lemma}\label{lb:successor}
The rank-$\xi$ data produce $R_{\xi+1}$ and, after dual refinement,
the set and auxiliary data defining the clubs in $V$ at rank $\xi+1$.
\end{lemma}

\begin{proof}
Let $\beta$ satisfy the hypotheses of $R_{\xi+1}$, and suppose
$Y=P_{\xi+1}\cap\beta\cap X$ is stationary for some
$X\in I^L_{(\xi+1)_\beta}(\beta)$. On the restriction club,
$(\xi+1)_\beta=\xi_\beta+1$ by
\cite[Lemma~2.10]{CodyHigher}. Choose a true
$\psi\in\Pi^1_{\xi_\beta+1}$ with
$X\cap\operatorname{Ref}_\beta(\psi)=\varnothing$.
Theorem~\ref{lc:strong-amenability} supplies a strongly
amenable sequence $\vec K$ on $Y\cap A_\psi$.
By Lemma~\ref{lb:one-bound}, choose $u_\delta<\delta$ with
\[
 (D_\xi\cap\delta)\setminus(u_\delta+1)\subseteq K_\delta
                    \qquad(\delta\in Y\cap A_\psi).
\]
Fodor's lemma gives stationary $Y'\subseteq Y\cap A_\psi$
and a constant value $u$. Since the preliminary set lies in
$\acc(D_\xi)$ and reflects stationarily to $\beta$,
$D_\xi\cap\beta$ is club. Therefore
\[
 Y'\subseteq\operatorname{Cap}
       ((D_\xi\cap\beta)\setminus(u+1),\vec K),
\]
contradicting strong amenability. This also proves the case
$\beta=\kappa$.

Intersect the preliminary set with its stationary trace. If
the local ideal at a remaining $\beta$ were improper, take
$X=\beta$ in the just established relative assertion; its
conclusion would contradict stationarity below $\beta$.
Hence the refined set lies in $W_{\xi+1}$.

Put $F_{\xi+1}=E_\xi$, $S_{\xi+1}=P_\xi$, and
\[
 \Phi^{\xi+1}_\delta(Q)
   =\operatorname{cl}_\delta(E_\xi\cap\delta\cap P_\xi\cap Q)
                 \cap B_\delta(Q).
\]
The new relative assertion makes the set inside the closure
cofinal in $\delta$. If $\gamma\in P_{\xi+1}\cap
\acc(\Phi^{\xi+1}_\delta(Q))$, then
$E_\xi\cap\gamma\cap P_\xi\cap Q$ is unbounded in $\gamma$.
Were $Q\cap\gamma\in I^L_{\xi_\gamma}(\gamma)$, a constructible
$\Pi^1_{\xi_\gamma}$-club $C$ disjoint from $Q\cap\gamma$
would exist by \cite[Theorem~5.17(1)]{CodyHigher}. But $L_\xi$
gives
\[
 E_\xi\cap\gamma\cap P_\xi\subseteq^* C,
\]
contradicting that unboundedness. Thus $Q\cap\gamma$ is
positive at the predecessor index. The factor $B_\delta(Q)$
ensures the index identity and the closure condition for $Q$,
so $\gamma\in Q$.
\end{proof}

Let $\lambda<(\kappa^+)^L$ be a nonzero limit, and suppose that all
lower stages have been constructed. Define the auxiliary diagonal sets
\[
 P^\Delta=\{\alpha<\kappa:
      (\forall\zeta\in F^\kappa_\lambda(\alpha))\ \alpha\in P_\zeta\},
 \qquad
 E^\Delta=\{\alpha<\kappa:
      (\forall\zeta\in F^\kappa_\lambda(\alpha))\ \alpha\in E_\zeta\}.
\]
If $\lambda<\kappa$, these are intersections of fewer than $\kappa$
sets, and we work above $\lambda$. If $\lambda\geq\kappa$, they are
normal diagonals indexed by the constructible bijection $b_{\kappa,\lambda}$.
Thus $P^\Delta$ is dual and $E^\Delta$ contains a club in $\kappa$ in $V$.
Choose such a club $F_\lambda\subseteq E^\Delta$, and refine the auxiliary
set by $\operatorname{acc}(F_\lambda)$ and the required constructible
club of closure points.

The restriction identities used at this stage are as follows. On the constructible club
$C_\lambda$, the rank $\lambda_\beta$ is a limit; for every
$\zeta\in F^\kappa_\lambda(\beta)$,
\[
 F^\kappa_\zeta(\beta)
      =F^\kappa_\lambda(\beta)\cap\zeta,
 \qquad
 f^\kappa_\zeta(\beta)=\pi_{\lambda,\beta}(\zeta),
\]
where $\pi_{\lambda,\beta}$ is the order collapse of
$F^\kappa_\lambda(\beta)$. This set of indices is successor closed,
and $\beta$ belongs to each relevant lower club of closure points. For
regular $\beta\in C_\lambda$, there is also a constructible club
$A^\lambda_\beta\subseteq\beta$ on which
\begin{equation}\label{lb:local-indices}
 \pi_{\lambda,\beta}[F^\kappa_\lambda(\alpha)]
     =F^\beta_{\lambda_\beta}(\alpha),
 \qquad
 f^\kappa_\lambda(\alpha)
     =f^\beta_{\lambda_\beta}(\alpha).
\end{equation}
These are \cite[Lemmas~2.7--2.9, Proposition~2.11,
and Corollary~2.12]{CodyHigher}, together with closure and diagonal
refinements. Write $P$ for the preliminary refined auxiliary set.

\begin{lemma}\label{lb:limit-relative}
The set $P$ satisfies $R_\lambda$ locally and globally.
\end{lemma}

\begin{proof}
Suppose $P\cap\beta$ is stationary and
$X\in I^L_{\lambda_\beta}(\beta)$. Choose
\[
 L_\beta\models\psi=\bigwedge_{\tau<\lambda_\beta}\psi_\tau,
 \qquad\psi_\tau\in\Pi^1_\tau,
 \qquad X\cap\operatorname{Ref}_\beta(\psi)=\varnothing.
\]
For $\zeta\in F^\kappa_\lambda(\beta)$, put
\[
 \tau_\zeta=\pi_{\lambda,\beta}(\zeta),\qquad
 X_{\tau_\zeta}=\beta\setminus
               \operatorname{Ref}_\beta(\psi_{\tau_\zeta}).
\]
Then $X_{\tau_\zeta}\in I^L_{\tau_\zeta}(\beta)$.
The set $P\cap\beta$ is eventually contained in $P_\zeta$,
so $P_\zeta\cap\beta$ is stationary. The relative $R_\zeta$
gives a club $B_\zeta\subseteq\beta$ such that
\[
 B_\zeta\cap P_\zeta\cap X_{\tau_\zeta}=\varnothing.
\]
When $\lambda\geq\kappa$, write $b=b_{\kappa,\lambda}$ and let
\[
 B=A^\lambda_\beta\cap\triangle_{i<\beta}B_{b(i)}
\]
intersected with the clubs of definition for this formula.
For $\zeta\in F^\kappa_\lambda(\beta)$,
$i=b^{-1}(\zeta)<\beta$ and
$P^\Delta\cap(i,\beta)\subseteq P_\zeta$.
When $\lambda<\kappa$, use
$B=A^\lambda_\beta\cap\bigcap_{\zeta<\lambda}B_\zeta$
and work above $\lambda$.

Let $\alpha\in P\cap B$. For every
$\tau\in F^\beta_{\lambda_\beta}(\alpha)$,
\eqref{lb:local-indices} gives $\zeta\in F^\kappa_\lambda(\alpha)$
with $\pi_{\lambda,\beta}(\zeta)=\tau$. In the diagonal case,
$b^{-1}(\zeta)<\alpha$, so
\[
 \alpha\in P_\zeta\cap B_\zeta,
 \qquad L_\alpha\models\psi_\tau|^\beta_\alpha.
\]
Consequently
\[
 L_\alpha\models\psi|^\beta_\alpha
     =\bigwedge_{\eta<f^\beta_{\lambda_\beta}(\alpha)}
       \psi_{(\pi^\beta_{\lambda_\beta,\alpha})^{-1}(\eta)}
                   |^\beta_\alpha.
\]
Thus $P\cap B\cap X=\varnothing$, proving $R_\lambda$.
For $\beta=\kappa$, take all $\zeta<\lambda$, with
$\tau_\zeta=\zeta$, and use the same diagonal intersection.
\end{proof}

Refine $P$ by its stationary trace. Applying $R_\lambda$ to
$X=\beta$ at any ordinal where the local ideal would be improper
shows that a dual refinement $P_\lambda$ is contained in $W_\lambda$.
The assertion $R_\lambda$ at $\beta=\kappa$ proves
$\Pi^1_\lambda$-indescribability of $\kappa$ in $L$. We may therefore
choose the uniform formula defining $\Pi^1_\lambda$-clubs. Let $S_\lambda$ be the
preliminary auxiliary set, so that $P_\lambda\subseteq S_\lambda
\subseteq P^\Delta$.

\begin{lemma}\label{lb:limit-club-in-V}
The sets $F_\lambda,S_\lambda$ satisfy
\eqref{lb:club-in-V} and \eqref{lb:recognition} at rank $\lambda$.
\end{lemma}

\begin{proof}
Let $\delta\in P_\lambda\cap\TrV(P_\lambda)$ and let $Q$ be
a constructible $\Pi^1_{\lambda_\delta}$-club in $\delta$.
Since $\delta\setminus Q\in I^L_{\lambda_\delta}(\delta)$,
$R_\lambda$ gives
\[
 \sup(F_\lambda\cap\delta\cap S_\lambda\cap Q)=\delta.
\]
Its closure, intersected with $A^\lambda_\delta$ and a
club in $L$ witnessing the closure condition for $Q$, is a club
$\Phi^\lambda_\delta(Q)$ in $\delta$ in $V$.

Let $\gamma\in P_\lambda\cap\acc(\Phi^\lambda_\delta(Q))$
and fix $\tau<\lambda_\gamma$. Choose
$\zeta\in F^\kappa_\lambda(\gamma)$ with
$f^\kappa_\zeta(\gamma)=\tau$. Successor closure of this
index set gives $\zeta+1\in F^\kappa_\lambda(\gamma)$, hence
$\gamma\in P_{\zeta+1}$. Property $L_\zeta$ therefore holds
at $\gamma$. A tail of
$F_\lambda\cap\gamma\cap S_\lambda\cap Q$ lies in
$E_\zeta\cap P_\zeta$ and is unbounded. If
$Q\cap\gamma\in I^L_\tau(\gamma)$, choose a constructible
$\Pi^1_\tau$-club $C$ disjoint from $Q\cap\gamma$. Then
\[
 E_\zeta\cap\gamma\cap P_\zeta\subseteq^*C
\]
contradicts the unbounded tail in $Q$. Thus
\[
 (\forall\tau<\lambda_\gamma)\quad
               Q\cap\gamma\in(I^L_\tau(\gamma))^+.
\]
Equation~\eqref{lb:local-indices} identifies these with the local
indices in the closure condition for $Q$. The club in $L$ used as a
closure factor gives $\gamma\in Q$.
\end{proof}

After constructing $\Phi^\lambda_\delta$, apply
Lemma~\ref{lb:global-lemma} to obtain $E_\lambda$ and
Lemma~\ref{lb:bad-ordinals} to obtain $L_\lambda$.

\begin{proof}[Proof of Theorem~\ref{lb:relative-bound}]
The preceding recursion constructs, for every
$1\leq\xi<(\kappa^+)^L$, the assertions $R_\xi,G_\xi,L_\xi$
and a set $P_\xi\in(\SA_\kappa^V)^*$ with $P_\xi\subseteq W_\xi$.
Lemma~\ref{fd:relativeone} gives the initial step;
Lemmas~\ref{lb:global-lemma} and~\ref{lb:bad-ordinals} give
$G_\xi,L_\xi$; Lemma~\ref{lb:successor} handles successors;
and Lemmas~\ref{lb:limit-relative} and
\ref{lb:limit-club-in-V} handle limits. Normality preserves
membership in $(\SA_\kappa^V)^*$ at each refinement.

For $X\in I_\xi^L(\kappa)$, the global $R_\xi$ gives
$X\cap P_\xi\in\NS_\kappa^V$. Hence
\[
 X=(X\cap P_\xi)\cup(X\setminus P_\xi)\in\SA_\kappa^V,
 \qquad \kappa\setminus W_\xi\subseteq\kappa\setminus P_\xi
                                      \in\SA_\kappa^V.
\]
Since $\kappa\notin\SA_\kappa^V$, the ideal
$I_\xi^L(\kappa)$ is proper for every such $\xi$. Together
with inaccessibility in $L$, this gives
\[
 L\models\text{``$\kappa$ is inaccessible and }
        (\forall\xi<\kappa^+)\,
               \kappa\text{ is }\Pi^1_\xi\text{-indescribable''}.
\]
\end{proof}

\begin{corollary}\label{lb:dichotomy}
If $\kappa$ is regular and uncountable and $\NAC(\kappa)$ holds, then
\[
 L\models\text{``$\kappa$ is ineffable''}
 \quad\text{or}\quad L\models\mathsf H(\kappa).
\]
In either case there is a $\lambda\leq\kappa$ such that
$L\models\mathsf H(\lambda)$.
\end{corollary}

\begin{proof}
If $L\models$ ``$\kappa$ is not ineffable'', take $T=\kappa$
in Theorem~\ref{lb:relative-bound}. Its hypotheses hold since
$\NAC(\kappa)$ means $\kappa\notin\A_\kappa^V$. Thus
$L\models\mathsf H(\kappa)$.
If $\kappa$ is ineffable in $L$, it is subtle there.
By \cite[Corollary~3.5]{CodyHoly}, applied inside $L$,
\[
 L\models\exists\lambda<\kappa\ \forall\eta<\lambda^+\quad
              \lambda\text{ is }\Pi^1_\eta\text{-indescribable}.
\]
The rank-one instance implies inaccessibility,
so this $\lambda$ satisfies $\mathsf H$ in $L$.
\end{proof}

\section{Stationary capture by generalized Hechler forcing}
\label{sec:capture}

We use $\Pi^1_1$ reflection to prove local domination by the Hechler
generic function. Applying this to the next-point functions of a
$C$-sequence gives stationary capture.

The reflection lemma, local domination theorem and capture corollary
in this section are \cite[Lemma~3.8, Theorem~3.10 and Corollary~3.12]{Fernandes},
respectively. Their proofs are included for completeness.
Generalized Hechler forcing and its basic properties were studied by
Cummings and Shelah \cite[Section~3 and Lemma~7]{CS}; we use the
presentation in \cite[Definition~48]{BBFM}.

\begin{definition}\label{def:hechler}
For a regular uncountable cardinal \(\kappa\), let
\(\D_\kappa\) be generalized Hechler forcing. Its conditions are
\(p=(s_p,h_p)\), where \(s_p\in{}^{<\kappa}\kappa\),
\(h_p\in{}^\kappa\kappa\), and \(s_p\subseteq h_p\).
We call \(s_p\) the \emph{stem}, put \(\ell(p)=\dom(s_p)\),
and call \(h_p\) the \emph{side function}. Thus \(q\leq p\)
means \(s_p\subseteq s_q\) and \(h_p\leq h_q\) pointwise.
If \(\beta\leq\kappa\) is regular uncountable and
\(p\in\D_\kappa\) satisfies \(\ell(p)<\beta\) and
\(h_p[\beta]\subseteq\beta\), put
\[
 p\restriction\beta=(s_p,h_p\restriction\beta).
\]
Then \(p\restriction\beta\in\D_\beta\).
\end{definition}

\begin{definition}
\label{def:eventual-domination}\label{def:generic-function}
If \(G\subseteq\D_\kappa\) is generic, its \emph{Hechler generic
function} is \(d_G=\bigcup_{p\in G}s_p\). In forcing statements,
\(d\) denotes the canonical name for this function.
\end{definition}

\begin{definition}\label{def:compatibility}
\label{def:predensity}
For a forcing order \(\mathbb P\), let
\(\Comp_{\mathbb P}(p,q)\) mean that \(p,q\in\mathbb P\) are
compatible. A family \(A\subseteq\mathbb P\) is \emph{predense
below \(p\)} if every \(q\leq p\) is compatible with a member of
\(A\).
We omit the subscript on \(\Comp\) when the order is clear.
\end{definition}

\begin{remark}\label{rem:compatibility}
For \(p=(s,h),q=(t,k)\in\D_\kappa\),
\begin{align}
 \Comp_{\D_\kappa}(p,q)\quad\Longleftrightarrow\quad
 &\bigl[s\subseteq t\ \land\
       (\forall\xi<\dom(t))\,h(\xi)\leq t(\xi)\bigr]\notag\\
 &{}\lor\bigl[t\subseteq s\ \land\
       (\forall\xi<\dom(s))\,k(\xi)\leq s(\xi)\bigr].
 \label{eq:compat}
\end{align}
A common extension must have a stem extending both \(s\) and \(t\),
so the stems are comparable and its values on the longer stem give
the corresponding inequality. Conversely, the longer stem and the
pointwise maximum of the side functions form a common extension;
the inequality ensures normalization. All these tests use only
coordinates below the length of the longer stem. Consequently, if
both restrictions to a regular uncountable \(\beta\leq\kappa\)
are defined, then
\[
 \Comp_{\D_\kappa}(p,q)
 \quad\Longleftrightarrow\quad
 \Comp_{\D_\beta}(p\restriction\beta,q\restriction\beta).
\]
\end{remark}

\begin{lemma}\label{lem:hechler-basic}
For a regular uncountable \(\kappa\), if
\(G\subseteq\D_\kappa\) is generic, then
\(d_G\) is a function from \(\kappa\) to \(\kappa\)
that eventually dominates every function \(f\) from
\(\kappa\) to \(\kappa\) in the ground universe. Moreover,
\[
 G=\{p\in\D_\kappa:s_p\subseteq d_G
                     \text{ and }h_p\leq d_G\},
\]
where \(\D_\kappa\) on the right is the ground-model order.
In particular, \(V[G]=V[d_G]\).
\end{lemma}
\begin{proof}
The eventual-domination assertion is standard
\cite[Lemma~7]{CS}; see also \cite[Section~4.2]{BBFM}.
For $\alpha<\kappa$ and a ground-model $f\in{}^\kappa\kappa$,
the sets
\[
 \begin{aligned}
 D_\alpha&=\{p:\ell(p)>\alpha\},\\
 D_f&=\{p:(\forall\xi\in[\ell(p),\kappa))\,h_p(\xi)\geq f(\xi)\}
 \end{aligned}
\]
are dense. To enter $D_\alpha$, extend the stem using $h_p$;
to enter $D_f$, replace $h_p(\xi)$ by
$\max\{h_p(\xi),f(\xi)\}$ for $\xi\geq\ell(p)$.
Genericity gives $d_G\in{}^\kappa\kappa$ and, for $p\in G\cap D_f$,
$f(\xi)\leq d_G(\xi)$ whenever $\xi\geq\ell(p)$.

If $p\in G$ and $\xi<\kappa$, choose $q\in G$ below $p$
with $\ell(q)>\xi$. Then
$h_p(\xi)\leq h_q(\xi)=s_q(\xi)=d_G(\xi)$.
Thus $s_p\subseteq d_G$ and $h_p\leq d_G$.
Conversely, if a ground-model $p$ satisfies these inequalities,
\eqref{eq:compat} gives
\[
 \forall q\in G\quad\Comp(p,q).
\]
The dense set $\{q:q\leq p\text{ or }q\perp p\}$ must be
met in its first part. Upward closure of $G$ yields $p\in G$.
This proves the formula for $G$ and hence $V[G]=V[d_G]$.
\end{proof}

\begin{samepage}
\begin{lemma}\label{lem:hechler-cc}
If \(\kappa\) is strongly inaccessible, then \(\D_\kappa\) is
\(\kappa^+\)-cc \cite[Section~4.2]{BBFM}.

\end{lemma}
\end{samepage}
\begin{proof}
Since $\kappa^{<\kappa}=\kappa$, there are $\kappa$ possible
stems. For a fixed stem $s$, every finite family
$\{(s,h_i):i<n\}$ has common extension
\[
 (s,h),\qquad h(\xi)=\max_{i<n}h_i(\xi).
\]
The function $h$ agrees with $s$ on its domain. Thus conditions
with the same stem form a centered family, and an antichain has
at most $\kappa$ members.
\end{proof}

\begin{lemma}\label{lem:hechler-closure}
For every regular uncountable \(\kappa\), the order
\(\D_\kappa\) is \(<\kappa\)-closed.
\end{lemma}
\begin{proof}
For later use, recall the lower bound witnessing closure in
\cite[Section~4.2]{BBFM}: for a decreasing sequence
\(\langle(s_i,h_i):i<\mu\rangle\), with \(0<\mu<\kappa\), put
\[
 s=\bigcup_{i<\mu}s_i,\qquad
 h(\xi)=\sup_{i<\mu}h_i(\xi)\quad(\xi<\kappa).
\]
Regularity bounds the domain and values below \(\kappa\), and
normalization fixes \(h_i(\xi)\) once \(\xi\) enters a stem.
Thus \((s,h)\) is a normalized lower bound. The empty sequence
has the greatest condition \((\varnothing,\mathbf0)\) as a
lower bound.
\end{proof}

For local domination we need predensity in \(\D_\beta\), including
conditions whose side functions are arbitrary maps from \(\beta\)
to \(\beta\). The next lemma obtains this by \(\Pi^1_1\)
reflection.

\begin{lemma}\label{lem:reflection}
Let \(\kappa\) be weakly compact, let \(p\in\D_\kappa\), and
let \(\langle A_\alpha:\alpha<\kappa\rangle\) be a sequence of
maximal antichains below \(p\). Choose surjections
\(e_\alpha:\kappa\to A_\alpha\), and write
\[
 a_{\alpha,i}=e_\alpha(i)
             =(s_{\alpha,i},h_{\alpha,i})
                    \qquad(\alpha,i<\kappa).
\]
Let \(\delta:\kappa\times\kappa\to\kappa\) be any function,
and write \(\delta_{\alpha,i}=\delta(\alpha,i)\).
The set of strongly inaccessible \(\beta<\kappa\) satisfying the
following conditions is stationary in \(\kappa\):
\begin{enumerate}[label=\textup{(\roman*)}]
\item \(\ell(p)<\beta\) and \(h_p[\beta]\subseteq\beta\);
\item for all \(\alpha,i<\beta\),
      \(\ell(a_{\alpha,i})<\beta\),
      \(\delta_{\alpha,i}<\beta\), and
      \(h_{\alpha,i}[\beta]\subseteq\beta\);
\item for every \(\alpha<\beta\), the family
\[
 A_\alpha^\beta=
 \{a_{\alpha,i}\restriction\beta:i<\beta\}
\]
is predense below \(p\restriction\beta\) in \(\D_\beta\).
\end{enumerate}
\end{lemma}
\begin{proof}
Fix a club \(B\subseteq\kappa\), and let \(K\) be the club
of uncountable beth fixed points below \(\kappa\). Code into
one predicate \(R\subseteq V_\kappa\), using finite tagged
tuples, the following data:
\[
 \ell(p),\quad s_p,\quad h_p,\quad
 \langle\ell(a_{\alpha,i}),s_{\alpha,i},
              h_{\alpha,i},\delta_{\alpha,i}:\alpha,i<\kappa\rangle,
 \quad K\cap B.
\]
Stems are coded as sets, while each side function is coded by
its graph as a section of \(R\). Require in the reflection
sentence that all these sections are total on their indicated
ordinal indices, that the coded stems exist and are normalized,
that every \(a_{\alpha,i}\) extends \(p\), and that the section
coding \(K\cap B\) is unbounded. These are first-order
requirements, with no parameter for the ordinal indexing the rank structure.

Add the assertion that the ordinal indexing the rank structure is a limit
ordinal greater than \(\omega\)
and the two universal second-order requirements below. Here
\(H\) ranges over predicates, \(\Ord\) denotes the ordinals
of the rank structure, and coded data are read from \(R\):
\[
\begin{split}
 (\forall H)(\forall\eta)\bigl[H:\eta\to\Ord
   \ \Longrightarrow\
    (\exists\zeta)(\forall\xi<\eta)\,H(\xi)<\zeta\bigr],
\end{split}
\]
\[
\begin{split}
 (\forall H)\bigl[H:\Ord\to\Ord\ \Longrightarrow\
 (\forall\alpha)(\forall s)\bigl[
   &(\exists\eta)\,s:\eta\to\Ord,
       \quad s\subseteq H,\quad s_p\subseteq s,\\
   &h_p\leq H
       \ \Longrightarrow\
       (\exists i)\,\Comp_R((s,H),a_{\alpha,i})\bigr]\bigr].
\end{split}
\]
In the antecedent of the second display the requirements are
conjoined. The compatibility test \(\Comp_R\) is the
first-order formula of \eqref{eq:compat}, using the coded stem
and side-function section. The stem \(s\) is a first-order
variable, whereas \(H\) is a second-order graph. All quantifiers
other than the displayed universal predicate quantifiers are
first-order. The resulting sentence is therefore \(\Pi^1_1\).

It holds over \((V_\kappa,\in,R)\): the first display follows
from regularity of \(\kappa\), and the second from maximality
of the antichains below \(p\). Reflect it to \(V_\beta\) by
weak compactness. Unboundedness of the \(K\cap B\) section
and closedness of \(K\cap B\) give \(\beta\in K\cap B\).
The first display says, with full second-order semantics, that
every map from an ordinal below \(\beta\) into \(\beta\) is
bounded. Hence \(\beta\) is regular; since it is an uncountable
beth fixed point, it is strongly inaccessible.

Totality of the reflected sections gives all bounds in (i)
and (ii). The coded stems in \(V_\beta\) are the original
stems, because they are entries of their tagged tuples.
Finally, every \((s,H)\leq p\restriction\beta\) in
\(\D_\beta\) has \(s\in V_\beta\) and graph
\(H\subseteq V_\beta\). The second display therefore applies
to every such condition, including arbitrary side functions
\(H:\beta\to\beta\), and yields an \(i<\beta\) witnessing
compatibility with \(a_{\alpha,i}\restriction\beta\).
This is (iii). Since \(B\) was arbitrary, the desired set is
stationary.
\end{proof}

\begin{definition}\label{def:domination}
For a regular uncountable \(\kappa\), a set
\(I\subseteq\Inacc(\kappa)\), functions
\(\vec f=\langle f_\beta:\beta\in I\rangle\) with
\(f_\beta:\beta\to\beta\), and a function \(d\) from \(\kappa\) to \(\kappa\), put
\[
 T(\vec f,d)=
 \{\beta\in I:d[\beta]\subseteq\beta
          \ \land\ f_\beta\leq^*d\restriction\beta\}.
\]
\end{definition}

\begin{theorem}\label{thm:domination}
Suppose that \(\kappa\) is weakly compact in \(V\), let
\(I=\Inacc^V(\kappa)\), and let
\(\vec f=\langle f_\beta:\beta\in I\rangle\in V\), where
\(f_\beta:\beta\to\beta\) for each \(\beta\in I\). Then
\[
 \D_\kappa\Vdash
 T(\vec f,d)\text{ is stationary in }\kappa.
\]
\end{theorem}
\begin{proof}
Work in \(V\). By Lemma~\ref{lem:hechler-closure}, names for
ordinal sequences of length less than \(\kappa\) can be decided
by a descending recursion with a lower bound. Thus \(\D_\kappa\)
adds no such sequences, preserves regularity of \(\kappa\), and
preserves the inaccessible cardinals below it.

Let \(p\in\D_\kappa\) force that \(\dot E\subseteq\kappa\)
is club. For each \(\alpha<\kappa\), every condition below
\(p\) has an extension deciding an ordinal above \(\alpha\)
that belongs to \(\dot E\), because \(p\) forces unboundedness.
Choose a maximal antichain
\(A_\alpha\) below \(p\) consisting of such conditions, enumerate
it as in Lemma~\ref{lem:reflection}, and choose
\(\delta_{\alpha,i}\) so that
\[
 \alpha<\delta_{\alpha,i}<\kappa,\qquad
 a_{\alpha,i}\Vdash\delta_{\alpha,i}\in\dot E.
\]
Choose a strongly inaccessible \(\beta\) satisfying the conclusions
of that lemma, and put \(\ell=\ell(p)\). Replace the side function
of \(p\) by
\[
 h^0(\xi)=
 \begin{cases}
 \max\{h_p(\xi),f_\beta(\xi)\},&\ell\leq\xi<\beta,\\
 h_p(\xi),&\xi<\ell\text{ or }\beta\leq\xi<\kappa.
 \end{cases}
\]
Then \(q_0=(s_p,h^0)\leq p\), and
\(h^0[\beta]\subseteq\beta\). In \(\D_\beta\), each
\(A_\alpha^\beta\), \(\alpha<\beta\), remains predense below
\(q_0\restriction\beta\).

Construct a decreasing sequence
\(\langle q_\alpha=(t_\alpha,k_\alpha):\alpha<\beta\rangle\)
with this initial condition, maintaining
\begin{equation}\label{eq:invariant}
 \dom(t_\alpha)<\beta,\qquad
 k_\alpha[\beta]\subseteq\beta,\qquad q_\alpha\leq q_0.
\end{equation}
At stage \(\alpha<\beta\), predensity gives \(i_\alpha<\beta\)
such that \(q_\alpha\restriction\beta\) is compatible with
\(a_{\alpha,i_\alpha}\restriction\beta\). By
\eqref{eq:compat}, the corresponding global conditions are
compatible as well. The longer stem and the pointwise maximum of their
side functions give a common extension with stem length below
\(\beta\) and side function mapping \(\beta\) into \(\beta\).
Extend its stem, if necessary, to
length at least \(\alpha+1\), using the same side function on
the new coordinates. Call the result \(q_{\alpha+1}\). It
satisfies \eqref{eq:invariant} and extends
\(a_{\alpha,i_\alpha}\).

At a limit \(\rho<\beta\), use the union of the previous stems
and the pointwise supremum of the previous side functions.
Regularity of \(\beta\) bounds the stem length and every value at
coordinates below \(\beta\) by \(\beta\). At coordinates at
least \(\beta\), regularity of \(\kappa\) bounds the values by
\(\kappa\). Once a coordinate belongs to a stem, its later
side-function values equal that stem value, as in
Lemma~\ref{lem:hechler-closure}; this gives normalization.
Hence the resulting condition satisfies \eqref{eq:invariant}.

Finally set
\[
 t^*=\bigcup_{\alpha<\beta}t_\alpha,\qquad
 k^*(\xi)=\sup_{\alpha<\beta}k_\alpha(\xi)
                   \quad(\xi<\kappa).
\]
The domain of \(t^*\) is exactly \(\beta\), since every stem
has length below \(\beta\) and
\(\dom(t_{\alpha+1})\geq\alpha+1\). To bound its values below
\(\beta\), fix \(\xi<\beta\) and choose \(\alpha<\beta\) with
\(\xi<\dom(t_{\alpha+1})\). The value
\(v=t_{\alpha+1}(\xi)=k_{\alpha+1}(\xi)\) is below
\(\beta\); all later side functions have this value at \(\xi\),
and all earlier ones have value at most \(v\).
Consequently
\[
 t^*(\xi)=k^*(\xi)=v<\beta.
\]
For \(\xi\geq\beta\), the supremum involves only
\(\beta<\kappa\) values, so \(k^*(\xi)<\kappa\).
Therefore \(q=(t^*,k^*)\) is a normalized condition in
\(\D_\kappa\) extending every \(q_\alpha\). It forces
\[
 d\restriction\beta=t^*,\qquad d[\beta]\subseteq\beta,
 \qquad(\forall\xi\in[\ell,\beta))\,
                       f_\beta(\xi)\leq d(\xi).
\]
Also \(q\leq a_{\alpha,i_\alpha}\) for each \(\alpha<\beta\),
so it forces all the ordinals \(\delta_{\alpha,i_\alpha}\) into
\(\dot E\). They are below \(\beta\) and cofinal in
\(\beta\), since \(\alpha<\delta_{\alpha,i_\alpha}\).
Closedness of \(\dot E\) now gives
\[
 q\Vdash\beta\in\dot E\cap T(\vec f,d).
\]
Since this construction applies below every condition forcing a
club, \(\D_\kappa\) forces \(T(\vec f,d)\) to be stationary.
\end{proof}

\begin{definition}\label{def:closure-points}
For a regular uncountable \(\kappa\) and a function \(d\) from \(\kappa\) to \(\kappa\),
put
\[
 E_d=\{\delta<\kappa:\delta>0\text{ is limit and }
                              d[\delta]\subseteq\delta\}.
\]
\end{definition}

\begin{corollary}\label{cor:capture}
Suppose that \(\kappa\) is weakly compact in \(V\), let
\(I=\Inacc^V(\kappa)\), and let
\(\langle C_\beta:\beta\in I\rangle\in V\) consist of clubs
\(C_\beta\subseteq\beta\). If \(G\subseteq\D_\kappa\) is
generic over \(V\) and \(d=d_G\), then in \(V[d]\) there are
a club \(D\subseteq\kappa\) and a stationary set
\(T\subseteq I\) such that
\[
 D\cap\beta\subseteq C_\beta\qquad(\beta\in T).
\]
In particular, every ground-model \(C\)-sequence whose domain
contains \(I\) is non-amenable in \(V[d]\).
\end{corollary}
\begin{proof}
In \(V\), for each \(\beta\in I\) form the next-point function
\[
 n_\beta(\xi)=\min(C_\beta\setminus(\xi+1))
                              \qquad(\xi<\beta).
\]
Unboundedness gives \(\xi<n_\beta(\xi)<\beta\). Apply
Theorem~\ref{thm:domination} to these functions. In \(V[d]\),
\(\kappa\) is regular and
\[
 S=\{\beta\in I:d[\beta]\subseteq\beta
                  \ \land\ n_\beta\leq^*d\restriction\beta\}
\]
is stationary. For \(\beta\in S\), let
\(\gamma_\beta<\beta\) be the least bound witnessing
\(n_\beta\leq^*d\restriction\beta\). The pressing-down lemma, applied
in \(V[d]\), gives a stationary \(T\subseteq S\) and an
ordinal \(\gamma<\kappa\) such that
\(\gamma_\beta=\gamma\) for all \(\beta\in T\).

The set \(E_d\) is closed: if \(\delta<\kappa\) is a limit
of its members, then for each \(\xi<\delta\) there is
\(\varepsilon\in E_d\cap\delta\) with \(\xi<\varepsilon\),
giving \(d(\xi)<\varepsilon<\delta\). It is also unbounded.
Above any prescribed bound, choose a strictly increasing sequence
\(\langle\eta_n:n<\omega\rangle\) such that
\(\sup(d[\eta_n])<\eta_{n+1}\). Regularity of \(\kappa\)
permits each choice and gives
\(\delta=\sup_{n<\omega}\eta_n<\kappa\); this \(\delta\)
belongs to \(E_d\). Therefore
\(D=E_d\setminus(\gamma+1)\) is club.

Fix \(\beta\in T\) and \(\delta\in D\cap\beta\). Then
\[
 \xi<n_\beta(\xi)\leq d(\xi)<\delta
                     \qquad(\gamma\leq\xi<\delta).
\]
For every \(\eta<\delta\), choose
\(\xi\in[\max\{\eta,\gamma\},\delta)\); this is possible
because \(\delta>\gamma\) is limit. The point
\(n_\beta(\xi)\) lies in \(C_\beta\cap\delta\) and is
greater than \(\eta\). Hence \(C_\beta\cap\delta\) is
unbounded in \(\delta\), and closedness of \(C_\beta\) gives
\(\delta\in C_\beta\). This proves the asserted inclusion.
The last assertion follows from $T\subseteq\operatorname{Cap}(D,\vec C)$.
\end{proof}

\section{Hechler iterations}\label{sec:iteration}

Throughout this section \(\kappa\) is strongly inaccessible, and all
iterations and cardinalities are computed in a fixed ground universe
\(V\). We use the full Hechler orders in the intermediate universes.

The factorization, closure, chain-condition, size, localization and
stationary-set preservation facts below are recorded in
\cite[Lemmas~4.3--4.7]{Fernandes}.
Closure of iterations and their quotients is standard
\cite[Propositions~7.9 and~7.12]{Cummings}.
Earlier uses of generalized Hechler iterations include
\cite[Section~4.2]{BBFM} and \cite[Theorem~4.7]{Schuerz};
the latter follows a preparation of a strongly unfoldable cardinal
by a long Hechler iteration.
For an analogous calculation of the size of a dense presentation in a
related iteration, see \cite[Lemma~3.10]{Schuerz}.

\begin{definition}\label{def:iteration}
For \(\lambda\leq\kappa^+\), let
\(\langle\PP_\alpha,\dot{\mathbb Q}_\alpha:\alpha<\lambda\rangle\)
be the \(<\kappa\)-support iteration with
\[
 \PP_\alpha\Vdash
       \dot{\mathbb Q}_\alpha=\D_\kappa^{V[G_\alpha]}.
\]
A condition is a partial function whose domain has size less than
\(\kappa\). Its domain is its support; omitted coordinates denote the
greatest condition. We require
\(p\restriction\alpha\Vdash p(\alpha)\in\dot{\mathbb Q}_\alpha\)
at supported coordinates. The order is
\[
 q\leq p\quad\Longleftrightarrow\quad
 \supp(p)\subseteq\supp(q)\ \text{ and }\
 (\forall\alpha\in\supp(p))\,
 q\restriction\alpha\Vdash q(\alpha)\leq p(\alpha).
\]
Trivial coordinates may be inserted and remain in the domain.
At limits of cofinality less than \(\kappa\), supports may be
unbounded. These presentations are preorders; taking the quotient by
mutual extension gives partial orders.

We choose full iterand names, so that any name forced below \(r\)
to be an iterand condition has a representative in the chosen name
domain which agrees with it below \(r\), without strengthening
\(r\). Such set-sized domains exist: mix with the greatest condition
off \(r\), then represent both component graphs by names using check
names for pairs of ordinals below \(\kappa\).
\end{definition}

\begin{definition}\label{def:decided-stems}
Let \(\mathcal E_\lambda\) consist of those \(p\in\PP_\lambda\)
for which, at every \(\alpha\in\supp(p)\), there is a ground-model
stem \(s_\alpha\in{}^{<\kappa}\kappa\) such that
\[
 p\restriction\alpha\Vdash
     p(\alpha)=\langle\check s_\alpha,\dot h_\alpha\rangle
\]
for some side-function name \(\dot h_\alpha\).
\end{definition}

\begin{lemma}\label{lem:factorization}
For \(\theta\leq\lambda\leq\kappa^+\), the inclusion
\(\PP_\theta\subseteq\PP_\lambda\) is complete. If
\(\PP_\theta\) adds no ordinal sequences of length less than
\(\kappa\), its quotient is forcing equivalent, in
\(V[G_\theta]\), to the full \(<\kappa\)-support Hechler iteration on
\([\theta,\lambda)\).
\end{lemma}
\begin{proof}
For $p\in\PP_\lambda$ and $r\leq p\restr\theta$, let
\[
 q=r\cup(p\restr[\theta,\lambda)).
\]
Induction on the coordinates gives $q\in\PP_\lambda$,
$q\leq p$ and $q\restr\theta=r$. Thus restriction is a
projection, proving completeness of the inclusion.
The recursive translation of coordinate names identifies
\[
 \PP_\lambda/G_\theta\simeq
 \text{the }<\kappa\text{-support iteration on }[\theta,\lambda)
 \text{ in }V[G_\theta]
\]
as in \cite[Proposition~7.11]{Cummings}. Since $\PP_\theta$ adds no
ordinal sequences of length less than $\kappa$, for every ordinal $\eta$,
\[
 ([\eta]^{<\kappa})^{V[G_\theta]}=([\eta]^{<\kappa})^V.
\]
Hence the translated quotient admits exactly the supports of the
full tail iteration, also when a support is unbounded in a limit
of cofinality below $\kappa$.
\end{proof}

\begin{lemma}\label{lem:iteration}
For every \(\lambda\leq\kappa^+\), \(\PP_\lambda\) is
\(<\kappa\)-closed and \(\kappa^+\)-Knaster, and
\(\mathcal E_\lambda\) is dense. It preserves the strong
inaccessibility of \(\kappa\) and its successor. Every tail is
forcing equivalent to a \(<\kappa\)-closed order in its intermediate
universe.
\end{lemma}
\begin{proof}
Let $\langle p_i:i<\rho\rangle$ be decreasing, with
$\rho<\kappa$, and put $S=\bigcup_{i<\rho}\supp(p_i)$.
Then $|S|<\kappa$. Construct $q$ on $S$ by recursion. Once
$q\restr\alpha$ is defined and below every $p_i\restr\alpha$,
choose $q(\alpha)$ so that
\[
 q\restr\alpha\Vdash
 q(\alpha)=\left(\bigcup_{i<\rho}s_{p_i(\alpha)},
                  \xi\mapsto\sup_{i<\rho}h_{p_i(\alpha)}(\xi)\right).
\]
Omitted coordinates denote the greatest condition. The maximum
principle supplies the name; Lemma~\ref{lem:hechler-closure}
proves normalization and $q\leq p_i$. Consequently
$\PP_\lambda$ is $<\kappa$-closed, adds no ordinal sequences
of length $<\kappa$, and preserves inaccessibility of $\kappa$.

For density of $\mathcal E_\lambda$, construct
$p_{n+1}\leq p_n$ so that
\[
 p_{n+1}\restr\alpha\Vdash s_{p_n(\alpha)}=\check s_{n,\alpha}
                      \qquad(\alpha\in\supp(p_n)).
\]
Each stage requires fewer than $\kappa$ strengthenings in the
corresponding initial segments; use the closure just proved at
limits. For the coordinatewise lower bound $q$ of this
$\omega$-sequence,
\[
 s_{q(\alpha)}=\bigcup\{s_{n,\alpha}:\alpha\in\supp(p_n)\}
\]
is a ground-model stem. Thus $q\in\mathcal E_\lambda$ and
$q\leq p_0$.

Given $\kappa^+$ conditions, first strengthen them into
$\mathcal E_\lambda$. The generalized $\Delta$-system lemma
and $\kappa^{<\kappa}=\kappa$ give $I\subseteq\kappa^+$,
$|I|=\kappa^+$, and a common root $R$ such that
\[
 \supp(p_i)\cap\supp(p_j)=R,\qquad
 s_{p_i(\alpha)}=s_{p_j(\alpha)}\quad(\alpha\in R,\ i,j\in I).
\]
For $i,j\in I$, define a common extension on the union of the
supports. At a coordinate outside $R$, use the unique entry.
At $\alpha\in R$, retain the common stem and use the name
\[
 h_{q(\alpha)}(\xi)=
       \max\{h_{p_i(\alpha)}(\xi),h_{p_j(\alpha)}(\xi)\}.
\]
The earlier restriction forces normalization and extension of
both entries. Thus $p_i,p_j$ are compatible, proving the
$\kappa^+$-Knaster property and preservation of $\kappa^+$.
Lemma~\ref{lem:factorization} identifies each quotient with
a full tail iteration, to which the same closure argument applies.
\end{proof}

\begin{lemma}\label{lem:iteration-size}
For \(\theta\leq\kappa^+\), the order \(\PP_\theta\) has a
\(<\kappa\)-closed dense suborder of size at most \(2^\kappa\).
\end{lemma}
\begin{proof}
Put \(\mu=2^\kappa\). Then
\(\mu^\kappa=\mu^{<\kappa}=\mu\) and \(\kappa^+\leq\mu\).
Recursively choose dense sets \(D_\alpha\) of size at most
\(\mu\) and families \(\mathcal H_\alpha\) of coordinate
names. Over \(D_\alpha\), a nice name for a function from
\(\kappa\) to \(\kappa\) is given by \(\kappa\) antichains,
each of size at most \(\kappa\), and their decided values. There
are at most \(\mu^\kappa=\mu\) such names. Together with the
\(\kappa\) possible decided stems, these supply
\(\mathcal H_\alpha\) of size at most \(\mu\). Use names
globally forced to be normalized conditions with the prescribed
stem, including the greatest condition.

At stage \(\eta\), let \(D_\eta\) contain the conditions whose
coordinate names belong to these families. Counting supports and
names gives
\[
 |D_\eta|\leq|[\eta]^{<\kappa}|\,\mu^{<\kappa}\leq\mu.
\]
Given a condition, first decide all its stems. At each coordinate,
mix its name below the initial restriction with a fixed normalized
condition having the same stem elsewhere. Replace the mixed name
by its representative in \(\mathcal H_\alpha\). Recursion on
coordinates gives a stronger member of \(D_\eta\), proving
density. Finally, a short decreasing sequence in \(D_\theta\)
has a lower bound in \(\PP_\theta\), which can be strengthened
into \(D_\theta\). Thus the dense suborder is itself closed.
\end{proof}

\begin{lemma}\label{lem:localization}
Every subset of \(\kappa\), and every \(C\)-sequence on
\(\kappa\), in a \(\PP_{\kappa^+}\)-extension belongs to a
bounded intermediate extension.
\end{lemma}
\begin{proof}
For a name $\dot X$ for a subset of $\kappa$, choose maximal
antichains $A_\xi$ deciding $\xi\in\dot X$, for $\xi<\kappa$.
The chain condition gives $|A_\xi|\leq\kappa$. Therefore
\[
 S=\bigcup\{\supp(p):\xi<\kappa,\ p\in A_\xi\},\qquad
 |S|\leq\kappa,\qquad \theta=\sup(S)+1<\kappa^+.
\]
Every condition in each $A_\xi$ is a $\PP_\theta$-condition,
including all dependencies of its coordinate names. Hence the
nice name
\[
 \{(\check\xi,p):\xi<\kappa,\ p\in A_\xi,
                         \ p\Vdash\xi\in\dot X\}
\]
is a $\PP_\theta$-name for $X$. For a $C$-sequence, apply the
same argument to its domain and the relation
$\{(\beta,\alpha):\alpha\in C_\beta\}$, coded by a
ground-model bijection $\kappa\times\kappa\to\kappa$.
Closedness and unboundedness of each fixed entry are absolute
between the intermediate and final extensions.
\end{proof}

\begin{lemma}\label{lem:stationarity}
If \(\kappa\) is regular uncountable, every \(<\kappa\)-closed
forcing preserves every stationary subset of \(\kappa\).
\end{lemma}
\begin{proof}
Let $S\subseteq\kappa$ be stationary and
$p\Vdash\dot C\subseteq\kappa$ is club. Recursively choose
\[
 p_{\alpha+1}\leq p_\alpha\leq p,\qquad
 \alpha<\gamma_\alpha<\kappa,\qquad
 p_{\alpha+1}\Vdash\gamma_\alpha\in\dot C,
\]
taking lower bounds at limit stages below $\kappa$.
The set
\[
 B=\{\delta<\kappa:\delta>0\text{ is limit and }
                  (\forall\alpha<\delta)\,\gamma_\alpha<\delta\}
\]
is club. Choose $\delta\in S\cap B$. Then
$\sup_{\alpha<\delta}\gamma_\alpha=\delta$ and
\[
 p_\delta\Vdash
 \{\gamma_\alpha:\alpha<\delta\}\subseteq\dot C\cap\delta
 \quad\Longrightarrow\quad\delta\in\dot C\cap\check S.
\]
Thus every name for a club is forced to meet $S$.
\end{proof}

\providecommand{\InfH}{\mathsf H}
\providecommand{\InfPrep}{\mathbb S}

\section{The consistency upper bound}
\label{sec:upper}\label{up:section}

We prepare $\kappa$ so that every proper initial segment of the full
Hechler iteration of length $\kappa^+$ preserves weak compactness.
For each $\theta<\kappa^+$, we obtain the required correctness for
$\EE_\theta$ from indescribability at a level below $\kappa^+$,
using the embedding lemma proved below.

A \(\kappa\)-model below is a transitive \(M\models\ZFCc\) with
\(\kappa\in M\), \(|M|=\kappa\), and
\({}^{<\kappa}M\subseteq M\) externally.
The well-ordering principle is included in \(\ZFCc\).
At an inaccessible \(\kappa\), such a model contains every element
of \(V_\kappa\), by induction and closure under sequences of length less than $\kappa$.
It also contains \(V_\kappa\) as a set, since Collection and
Replacement collect the actual \(V_\alpha\)'s for \(\alpha<\kappa\).
It correctly regards \(\kappa\) as inaccessible.

Use well-founded relational syntax graphs with at most \(\kappa\)
nodes, finite labels, and conjunctions and disjunctions with at most
\(\kappa\) immediate subformulas.
We use node occurrences: a shared graph can be unfolded into finite
root-to-node paths without exceeding \(\kappa\) nodes.
Branches are finite, but the graph rank can be any ordinal below
\(\kappa^+\).
There are finitely many variables requiring an assignment at the root,
of either order, and finitely many second-order parameters.
Definition 3.1 in \cite{CodyHigher} permits any number of first-order parameters;
these are fixed constant labels, not variables requiring assignments.
Each finite leaf uses finitely many constants. Since each node has
only finitely many quantifier ancestors, its variable assignments
are finite as well. Fixed second-order parameters can be replaced
by additional free predicate variables from the outset; correctness
for all assignments then includes their original values.

Distinguish the relational graph \(S\) from its restriction data \(B\).
The graph has nodes indexed by a subset of \(\kappa\), a root,
connective and variable labels, and local codes for constants.
The data \(B\) include node enumerations, tupling, and the bijections
\(b_{\kappa,\rho}:\kappa\to\rho\) for all limit ranks
\(\rho\geq\kappa\) used by the graph, with the local index maps.
Both have hereditary size at most \(\kappa\).
High ordinal rank labels belong to \(B\); they are not relational
labels required to be fixed by a collapse.

\begin{lemma}\label{up:graphcut}
Fix a formula \(S\), its local restriction data, and codes
by \(\kappa\) for its node occurrences.
There is a club \(D\subseteq\kappa\) such that, for every uncountable
regular \(\alpha\in D\), the graph of \(S|^\kappa_\alpha\) is the
induced graph on node codes below \(\alpha\), with ranks reindexed
by the canonical collapses.
\end{lemma}
\begin{proof}
Let $I\subseteq\kappa$ index the occurrences of subformulas.
For a conjunction or disjunction at $n\in I$ of complexity
$\rho_n\geq\kappa$, let $c_n(\zeta)$ be the index of its
$\zeta$th constituent. Include the maps
\[
 (n,i)\longmapsto c_n(b_{\kappa,\rho_n}(i)),\qquad
 (n,c_n(\zeta))\longmapsto b_{\kappa,\rho_n}^{-1}(\zeta).
\]
Also include the maps for immediate predecessors, quantifier
matrices, variable labels and constant indices. For each
subformula of complexity $\rho_n<\kappa$, its syntax and
first-order parameters have size less than $\kappa$; include
bounds for their indices, parameter ranks and $\rho_n+1$.
Combine these maps into one finitary function $F$ on $\kappa$.
Then
\[
 D_0=\{\alpha<\kappa:F''\alpha^{<\omega}\subseteq\alpha\}
\]
is club, and for $\alpha\in D_0$ and $n<\alpha$,
\[
 c_n(\zeta)<\alpha\quad\Longleftrightarrow\quad
                \zeta\in F^\kappa_{\rho_n}(\alpha).
\]
For each $n$, let $D_n$ ensure the identities of
\cite[Lemmas~2.7 and 2.10]{CodyHigher}. In particular,
\[
 F^\kappa_\zeta(\alpha)
        =F^\kappa_{\rho_n}(\alpha)\cap\zeta,\qquad
 f^\kappa_\zeta(\alpha)
        =\pi^\kappa_{\rho_n,\alpha}(\zeta)
\]
when $\zeta\in F^\kappa_{\rho_n}(\alpha)$ at a limit
complexity. Set $D=D_0\cap\triangle_{n<\kappa}D_n$.
Induction on the subformula relation identifies the restricted
syntax with the induced graph on $I\cap\alpha$. Closure under
immediate predecessors ensures that every such occurrence
has all its predecessors present.
\end{proof}

\begin{lemma}\label{up:normalization}
Every well-founded second-order syntax graph of size at most
\(\kappa\), with finitely many free variables and second-order
parameters, has a logically equivalent hierarchy normal form of level
less than \(\kappa^+\).
There is a tuple of normalization data of hereditary size at most
\(\kappa\), including rank functions and enumerations for the syntax
graphs and their ordinal bounds.
The equivalence holds in full semantics and in the internal
semantics of a model with Pairing and Separation, using its internal
collection of predicates.
\end{lemma}
\begin{proof}
First replace each first-order quantifier outside finite
formulas using singleton predicates. For example,
\[
 \begin{aligned}
 \forall x\,\varphi(x)&\equiv
       \forall X\,(\operatorname{Singleton}(X)\to\varphi^X),\\
 \exists x\,\varphi(x)&\equiv
       \exists X\,(\operatorname{Singleton}(X)\wedge\varphi^X),
 \end{aligned}
\]
where $\operatorname{Singleton}(X)$ is
$\exists x\,\forall y\,(y\in X\leftrightarrow y=x)$, and
each finite formula in $\varphi^X$ refers to that unique member.
Rename bound variables to avoid capture. Pairing and Separation
supply the same singleton predicates in the internal semantics.

Normalize a formula and its negation simultaneously by
well-founded recursion. Finite quantifier blocks and dualization
require finitely many successor levels. For an arbitrary
conjunction $\bigwedge_{i\in I}\varphi_i$, suppose the recursive
step gives $\varphi_i'\in\Pi^1_{r_i}$ equivalent to $\varphi_i$.
Choose distinct nonzero limits $\zeta_i>r_i$ and a limit
$\tau>\sup_{i\in I}\zeta_i$, all below $\kappa^+$. Let
$\top_\eta$ be a tautology in $\Pi^1_\eta$, and define
\[
 \begin{aligned}
 \widehat\varphi_i&=\bigwedge_{\eta<\zeta_i}\theta^i_\eta,
 &\theta^i_\eta&=
     \begin{cases}\varphi_i',&\eta=r_i,\\\top_\eta,&\eta\ne r_i,\end{cases}\\
 \varphi'&=\bigwedge_{\eta<\tau}\theta_\eta,
 &\theta_\eta&=
     \begin{cases}\widehat\varphi_i,&\eta=\zeta_i\ (i\in I),\\
                   \top_\eta,&\eta\notin\{\zeta_i:i\in I\}.
     \end{cases}
 \end{aligned}
\]
Then $\varphi'\in\Pi^1_\tau$ and
$\varphi'\equiv\bigwedge_{i\in I}\varphi_i$ in both semantics.
Tautologies of all required levels are obtained recursively
from $x=x$ by dummy quantifier blocks and conjunctions.
The disjunction case uses the dual construction with false
formulas at unused indices. A change from $\Sigma^1_r$ to
$\Pi^1_{r+1}$, or conversely, uses an empty quantifier block.

Because $|I|\leq\kappa$ and $\kappa^+$ is regular, all
chosen ranks remain below $\kappa^+$. Each original occurrence
introduces at most $\kappa$ occurrences, so the resulting
syntax has size at most $\kappa\cdot\kappa=\kappa$.
An infinite branch would either lie in one of the well-founded
added formulas or project to an infinite branch of the original
syntax; both are impossible. Rank functions, enumerations and
the chosen indices have hereditary size at most $\kappa$.
\end{proof}

\begin{lemma}\label{up:finitewrap}
Let \(S\) be a formula of level \(\rho<\kappa^+\), and let
$C[S]$ be obtained from finitely many copies of $S$ and its
dual by substitutions of finite formulas, Boolean connectives,
and finitely many first- and second-order quantifiers.
It has a hierarchy normal form of level less than
\(\rho+\omega^2<\kappa^+\), and can be normalized so that on a club
of inaccessible \(\alpha\),
\[
                  C[S]|^\kappa_\alpha
                        \ \equiv\ C[S|^\kappa_\alpha].
\]
The parameters, syntax graphs and diagram predicates restrict to
their corresponding interpretations over $V_\alpha$.
The original local restriction data of \(S\) are unchanged.
\end{lemma}
\begin{proof}
Normalize only the finitely many operations applied to the
copies of $S$. Substituting a finite formula leaves their
complexities unchanged. A finite quantifier block or change
of class increases a rank by finitely many successors.
The conjunction construction in Lemma~\ref{up:normalization},
applied to finitely many formulas, can be performed below
$\rho+\omega m$ for some $m<\omega$. Therefore there is
$n<\omega$ such that the resulting formula has rank
\[
 r<\rho+\omega n<\rho+\omega^2<\kappa^+.
\]

Only countably many offsets $\tau<\omega^2$ occur in the
added syntax. On a club, simultaneously for those offsets,
\[
 F^\kappa_{\rho+\tau}(\alpha)
       =F^\kappa_\rho(\alpha)\cup[\rho,\rho+\tau),\qquad
 f^\kappa_{\rho+\tau}(\alpha)=f^\kappa_\rho(\alpha)+\tau.
\]
If $\rho<\kappa$, take $\alpha>\rho+\omega^2$. Otherwise,
at each added limit $\xi$, choose $\alpha$ above all inverse
indices $b_{\kappa,\xi}^{-1}(\eta)$ for
$\eta\in[\rho,\xi)$. Lemma~2.7 of \cite{CodyHigher}
identifies the initial part with $F^\kappa_\rho(\alpha)$,
and Lemma~2.10 handles successor complexities. There are
countably many such clubs and bounds.

Apply Lemma~\ref{up:graphcut} to the copies of $S$ and to
the added syntax, and close under the maps identifying their
variables and constants. At an inaccessible $\alpha$ in the
resulting club, each added constituent has the prescribed
local complexity. Deleting tautologies, false disjuncts and
dummy quantifiers gives
\[
 C[S]|^\kappa_\alpha\equiv C[S|^\kappa_\alpha].
\]
The equivalence follows by induction on the finite sequence of operations.
\end{proof}

Represent an element of \(H_{\kappa^+}\) by a rooted well-founded
extensional relation on a subset of \(\kappa\), with root \(0\)
and all nodes reached from the root by finite predecessor chains.
Well-foundedness is first-order over \(V_\kappa\) from its relation
predicate, since all \(\omega\)-sequences of ordinals below the
regular \(\kappa\) belong to \(V_\kappa\).
Equality is isomorphism of rooted codes.
Membership is isomorphism onto the predecessor-hereditary segment
of a node immediately below the second root.
These are finite second-order tests.

Aggregate the relational codes for all the at most \(\kappa\)
first-order constants in a syntax into one predicate
\(\Gamma\subseteq\kappa^3\).
Each leaf accesses finitely many sections \(\Gamma_i\) through
first-order indices \(i<\kappa\).
The finitely many free variables use finitely many additional codes.
Thus only finitely many second-order parameters are introduced.

In a transitive $\kappa$-model, an internally
well-founded code on \(\kappa\) is actually well-founded, since an
external descending \(\omega\)-sequence would belong to the model.
Its internal Mostowski collapse therefore agrees with the external one.
The represented objects are exactly its internal \(H_{\kappa^+}\),
when this exists as a set.
No equality with the actual \(H_{\kappa^+}\) is assumed.

The embedding characterization underlying the next argument is given
in \cite[Theorem~1.3]{Hauser}; we prove the local strengthened
conclusion needed here.
\begin{lemma}\label{up:hauser}
Suppose \(\InfH(\kappa)\).

Let $S\in\Pi^1_\rho\cup\Sigma^1_\rho$, where $\rho<\kappa^+$,
let $B$ be its restriction data, and let $M$ be a
$\kappa$-model containing $S,B$.
There are a \(\kappa\)-model \(N\), an elementary \(j:M\to N\),
and \(e_M:\kappa\twoheadrightarrow M\) belonging to \(N\), such that
\[
                         \crit(j)=\kappa,\qquad j,M\in N.
\]
The sets \(V_\kappa^N\), \(\mathcal P^N(V_\kappa)\), and a satisfaction
table for \(S\) on these domains exist in \(N\).
For every assignment of its predicates belonging to \(N\), the
model \(N\) agrees with the universe about \(S\).
Here the syntax is the original \(S\in M\in N\), not \(j(S)\).
\end{lemma}
\begin{proof}
Suppose a fixed $M$ has no target as in the conclusion.
Choose a bijection $e:\kappa\to M$. Define predicates coding
membership, the Skolem functions and the complete elementary diagram:
\[
 \begin{aligned}
 E(i,j)&\iff e(i)\in e(j),\\
 F(n,\vec i,j)&\iff f_n^M(e(\vec i))=e(j),\\
 T(\ulcorner\varphi\urcorner,\vec i)&\iff
                              M\models\varphi(e(\vec i)).
 \end{aligned}
\]
Here $\langle f_n:n<\omega\rangle$ enumerates Skolem functions
for a fixed well-order of $M$. Mark the indices of
$\kappa,V_\kappa,S$, and the inverse indices of all
$x\in V_\kappa$. Close under the Skolem maps and
\[
 \xi\mapsto\operatorname{rank}(e(\xi))+1
       \quad(e(\xi)\in V_\kappa),\qquad
 \alpha\mapsto\sup(e^{-1}[V_\alpha])+1.
\]
There is a club $D$ on which
\[
 e''\beta\prec M,\qquad e''\beta\cap V_\kappa=V_\beta.
\]
Thus the marked $\kappa,V_\kappa$ collapse to $\beta,V_\beta$.

Let $\Phi(E,F,T,S)$ assert that this diagram presents a
well-founded extensional structure satisfying $\ZFCc$,
with the marked objects and the following external closure:
\[
 \forall\delta<\kappa\ \forall f:\delta\to\kappa\quad
                       \langle e(f(\xi)):\xi<\delta\rangle\in M.
\]
This is a finite first-order formula over $V_\kappa$ using
the diagram predicates. Explicitly, well-foundedness is
$\neg\exists f:\omega\to\kappa\ \forall n\,
E(f(n+1),f(n))$, and extensionality is
\[
 \forall i,j<\kappa\,
 [\forall k<\kappa\,(E(k,i)\leftrightarrow E(k,j))\to i=j].
\]
The predicates $T$ obey the atomic, negation and existential
satisfaction clauses, and $T(n)$ holds for every axiom code
$n$ of $\ZFCc$. The closure requirement uses the diagram's
index for the function with the displayed values. Each
$f:\delta\to\kappa$ belongs to $V_\kappa$, so this
quantification is first order at an inaccessible $\kappa$.
The marked copy of $S$ is required to equal the external
relational syntax, with ordinal rank labels kept in $B$.

For an auxiliary membership code $R$ and its full diagram $U$,
let $Q$ denote its transitive collapse. The formula
$\operatorname{Model}(R,U,S)$ asserts, using that diagram,
\[
 \begin{gathered}
 Q\models\ZFCc,\qquad \kappa\in Q,\qquad |Q|=\kappa,
 \qquad{}^{<\kappa}Q\subseteq Q,\\
 M,E,F,T,S\in Q,\qquad
 V_\kappa^Q,\ \mathcal P^Q(V_\kappa),\
   \operatorname{Sat}(V_\kappa^Q,\mathcal P^Q(V_\kappa),S)\in Q.
 \end{gathered}
\]
These conditions are specified by the same finite diagram
clauses. In particular the marked source is required to be
the transitive collapse of the actual code $E$.
For $x\in V_\kappa$, its index $a$ in $Q$ is determined
by a function $f:\tc(\{x\})\to\Ord$ satisfying
\[
 f(x)=a,\qquad
 R(b,f(y))\iff(\exists z\in y)\ b=f(z)
 \qquad(y\in\tc(\{x\}),\ b\in\Ord).
\]
Such $f$ belongs to $V_\kappa$ by regularity. This is a
finite first-order definition of the representation of $x$.
If $c$ is an index for a predicate, its external interpretation is
\[
 \{x\in V_\kappa:\exists a\,
       (R(a,c)\wedge a\text{ represents }x)\}.
\]
Use these definitions for the external parameters in $S$.
Let $C_S(R)$ assert, for every finite assignment $\vec A$
from $\mathcal P^Q(V_\kappa)$ (and first-order assignments
in the appropriate domain),
\[
 Q\models\operatorname{Sat}
       (V_\kappa^Q,\mathcal P^Q(V_\kappa),S,\vec A)
        \quad\Longleftrightarrow\quad V_\kappa\models S[\vec A].
\]
The right side is a copy of $S$ with the displayed finite
definitions substituted for predicate membership. The left
side is an assertion in $U$ about a set satisfaction relation.
Only finite assignments at each subformula are used.

Let $\operatorname{NoEmb}_R$ be the assertion in $U$ that
there are no $N,j,e_M\in Q$ such that
\[
 \begin{gathered}
 j:M\to N\text{ is elementary},\qquad\crit(j)=\kappa,\\
 N\text{ is transitive},\quad N\models\ZFCc,\quad
 \kappa\in N,\quad |N|=\kappa,\\
 \quad{}^{<\kappa}N\subseteq N,\quad j,M,e_M\in N,\\
 e_M:\kappa\twoheadrightarrow M,
 \end{gathered}
\]
and $N$ possesses the stated domains and satisfaction relation,
which agrees with $Q$'s relation for $S$ on all assignments
from $N$. These are first-order assertions in the set model
$Q$, and are therefore read from its complete diagram.
The assumed failure of the lemma implies
\begin{equation}\label{up:failure}
 \Phi(E,F,T,S)\ \wedge\
 \forall R,U\,
 [\operatorname{Model}(R,U,S)\wedge C_S(R)
                                      \to\operatorname{NoEmb}_R].
\end{equation}
Indeed, if $Q$ recognized such a target, its transitive sets
and diagrams would make the embedding actually elementary.
Its enumerations give the actual cardinalities, and the
closure of $Q$ under sequences of length less than $\kappa$ makes
the target closed under such sequences in $V$. Set satisfaction is absolute for the required
domains. Finally $C_S(R)$ would give external correctness for
$S$, contrary to the choice of $M$.

Sentence~\eqref{up:failure} is obtained by finitely many
logical operations from copies of $S$, its dual and finite
diagram formulas. By Lemma~\ref{up:finitewrap}, it has a
normal form of complexity $\eta<\rho+\omega^2<\kappa^+$.
Conjoin the $\Pi^1_1$ assertion of inaccessibility and the
unboundedness of $D$, and reflect to an inaccessible
$\beta\in D$, using $\mathsf H(\kappa)$.
The bounded-range assertion
\[
 \forall H\ \forall x\,
 [H:x\to\Ord\ \to\ \exists\zeta\,\forall y\in x\ H(y)<\zeta]
\]
together with uncountability, asserts that the ordinal indexing the rank
structure is a regular strong-limit cardinal; $H$ is second order. These are finitely many
additional operations. By Lemma~\ref{up:finitewrap}, the
reflected sentence is exactly the same construction with
$S_\beta=S|^\kappa_\beta$ and the restricted diagrams.

Collapsing $e''\beta$ gives a $\beta$-model $M^*$ and
\[
 i^*:M^*\to M,\qquad \crit(i^*)=\beta,\quad i^*(\beta)=\kappa.
\]
Both $M^*$ and $i^*$ belong to $M$, since
${}^{<\kappa}M\subseteq M$ in $V$ and $V_\kappa\subseteq M$. Let $\mathcal T_\beta$ denote the
full satisfaction relation for
$(V_\beta,\mathcal P(V_\beta),S_\beta)$; this set has size
and rank below $\kappa$, and belongs to $M$.
Choose $Y\prec M$ with $|Y|=\beta$, ${}^{<\beta}Y\subseteq Y$,
containing $V_\beta\cup M^*\cup i^*{}''M^*$ pointwise and
\[
 i^*,M^*,S_\beta,\mathcal P(V_\beta),\mathcal T_\beta,
             e_*:\beta\twoheadrightarrow M^*
\]
as elements. The enumeration $e_*$ is induced by the restricted
source diagram. Closure under sequences of length less than $\kappa$ puts these choices in $M$.
For $\pi:Y\cong N_\beta$,
\[
 \pi\restr(M^*\cup V_\beta)=\mathrm{id},\quad
 \pi(S_\beta)=S_\beta,\quad \pi(e_*)=e_*,\quad
 i=\pi(i^*):M^*\to N_\beta.
\]
The collapsed power set and $\mathcal T_\beta$ give the internal
satisfaction relation. Since each predicate in $Y$ is a
subset of the pointwise fixed $V_\beta$, its truth values
agree with full truth. Hence $N_\beta$ is correct for
$S_\beta$, and $\crit(i)=\beta$ with $i,M^*,e_*\in N_\beta$.

Take a second $Y^+\prec M$, with $|Y^+|=\beta$ and
${}^{<\beta}Y^+\subseteq Y^+$,
containing $\tc(\{N_\beta,i,M^*\})$ pointwise, together with
an enumeration of $N_\beta$, the restricted diagram codes,
$e_*,S_\beta,\mathcal P(V_\beta),\mathcal T_\beta$ as elements.
Its collapse $\mathcal R_\beta$ fixes these witnesses and
restricted codes and is again correct for $S_\beta$. Let
$R_\beta$ be a membership code for $\mathcal R_\beta$ on
$\beta$, and let $U_\beta$ be its complete diagram. Then
\[
 \operatorname{Model}(R_\beta,U_\beta,S_\beta)\wedge
 C_{S_\beta}(R_\beta)\wedge
 \neg\operatorname{NoEmb}_{R_\beta},
\]
since it recognizes $N_\beta,i,e_*$ as a target and witnesses
for $M^*$. This contradicts \eqref{up:failure} at $\beta$.
\end{proof}

\begin{remark}\label{up:level}
For a level-\(\rho\) syntax this proof uses only one
\(\Pi^1_\eta\)-indescribability assumption, for some
\(\eta<\rho+\omega^2\), uniformly in the source.
In particular \(\Pi^1_{\rho+\omega^2}\)-indescribability suffices,
by \cite[Corollary 5.15]{CodyHigher}.
\end{remark}

\begin{corollary}\label{up:Hembedding}
Suppose \(\InfH(\kappa)\).
Let \(\psi\) be a well-founded infinitary first-order formula
over \(H_{\kappa^+}\), with syntax of size at most \(\kappa\).
Let a \(\kappa\)-model \(M\) contain \(\psi\), a hierarchy normalization
\(S\) of its relation-code interpretation, and all coding and local
restriction data.
There are \(j:M\to N\) and \(e_M\) as above such that \(N\)
is correct for \(\psi\) on every assignment from its internal
\(H_{\kappa^+}\).
\end{corollary}
\begin{proof}
Apply Lemma~\ref{up:hauser} to the given normal form $S$.
Let $K^N=(H_{\kappa^+})^N$. For an assignment $\vec a\in K^N$,
choose its rooted relation codes $\vec A$ in $N$. The
normalization and relation-code interpretation give
\[
 \begin{aligned}
 (K^N,\in)\models\psi[\vec a]
 &\iff(V_\kappa^N,\mathcal P^N(V_\kappa))\models S[\vec A]\\
 &\iff(V_\kappa,\mathcal P(V_\kappa))\models S[\vec A]\\
 &\iff(H_{\kappa^+},\in)\models\psi[\vec a].
 \end{aligned}
\]
The middle equivalence is the correctness from that lemma.
Since $j(V_\kappa)=V_{j(\kappa)}^N\in N$ and $j(\kappa)$
is internally inaccessible, $K^N$ and its satisfaction
relations exist as sets in $N$.
\end{proof}

\begin{lemma}\label{up:translate}
Let \(R\) have size \(\kappa\), conditions in \(V_\kappa\), and
preserve \(\kappa,\kappa^+\).
For every infinitary first-order formula \(\psi\) over the future
\(H_{\kappa^+}\), of syntax size at most \(\kappa\), there is an
infinitary first-order formula \(F_R\psi\) over ground-model
\(K=H_{\kappa^+}\), of the same size bound, expressing its forcing
relation on names for relation codes.
The order of \(R\) is a parameter.
\end{lemma}
\begin{proof}
For $n<\omega$, a positive name for a relation on $\kappa^n$
is $\sigma\subseteq\kappa^n\times R$, with interpretation
\[
 \sigma^G=\{a\in\kappa^n:(\exists p\in G)\,(a,p)\in\sigma\}.
\]
Every such name belongs to $K$. For $a\in\kappa^n$,
\[
 p\Vdash\check a\in\sigma
 \quad\Longleftrightarrow\quad
 \forall q\leq p\ \exists r\leq q\ \exists t\in R\
                    ((a,t)\in\sigma\ \wedge\ r\leq t).
\]
Thus membership in a named relation is a finite first-order
test over $K$. Every future relation has a positive name by
the truth lemma. Represent objects by pairs
$\sigma=(\sigma_D,\sigma_E)$ naming the domain and membership
relation of a rooted extensional well-founded code.

Let $\operatorname{Candidate}(\sigma)$ mean exactly that
$\sigma_D\subseteq\kappa\times R$ and
$\sigma_E\subseteq\kappa^2\times R$. Define
$\operatorname{Code}(p,\sigma)$ to assert that $p$ forces this
pair to be a rooted extensional well-founded relation.
These are finite tests: well-foundedness quantifies over
positive names for nonempty subsets and requires a minimal
element; extensionality quantifies over indices; and generation
from the root quantifies over predecessor-closed subsets
containing it. Expand each named membership assertion by the
displayed density test. Quantifiers over future subsets use
their positive names, all belonging to $K$.
Similarly, $\operatorname{Eq}$ and $\operatorname{Mem}$ use
positive names for isomorphisms of the rooted codes or their
predecessor-closed subrelations. These names also belong to $K$.

Define the translation recursively, with
$\operatorname{Code}(p,\vec\sigma)$ conjoined at each step:
\[
 \begin{aligned}
 F_R(\sigma=\tau)(p)&=\operatorname{Eq}(p,\sigma,\tau),\\
 F_R(\sigma\in\tau)(p)&=\operatorname{Mem}(p,\sigma,\tau),\\
 F_R(\neg\varphi)(p,\vec\sigma)
   &=\forall q\leq p\ \neg F_R\varphi(q,\vec\sigma),\\
 F_R\left(\bigwedge_{i\in I}\varphi_i\right)(p,\vec\sigma)
   &=\bigwedge_{i\in I}F_R\varphi_i(p,\vec\sigma),\\
 F_R(\exists x\,\varphi)(p,\vec\sigma)
   &=\forall q\leq p\ \exists r\leq q\ \exists\tau\in K\
       (\operatorname{Candidate}(\tau)\wedge
        \operatorname{Code}(r,\tau)\wedge
        F_R\varphi(r,\vec\sigma,\tau)).
 \end{aligned}
\]
Well-founded induction proves equivalence to the forcing
relation for satisfaction in the future $H_{\kappa^+}$.
For a set conjunction, use its Boolean value and a maximal
antichain deciding that value; no directedness of $G$ of size
$\kappa$ is required. Each original occurrence adds finitely
many operations, so the size remains at most $\kappa$.
\end{proof}

\begin{lemma}\label{up:afterprep}
Suppose also that \(R\) is \(\kappa\)-cc.
Let \(N\) be a \(\kappa\)-model containing all conditions of $R$,
its order, an enumeration \(\kappa\twoheadrightarrow R\), and the
graph of \(\psi\).
Assume that its local \(H_{\kappa^+}\) and satisfaction sets exist,
for example because \(V_\lambda^N\in N\) for an internally
inaccessible \(\lambda>\kappa\).
If \(N\) is correct for
\begin{equation}\label{up:translated-formulas}
           \operatorname{Code},\quad F_R\psi,\quad F_R(\neg\psi).
\end{equation}
on all its parameters, then an actual \(R\)-generic \(G\) makes
\(N[G]\) correct for \(\psi\) in \(V[G]\).
A finite selector combines \eqref{up:translated-formulas} into one formula.
\end{lemma}
\begin{proof}
All conditions and the order of $R$ are elements of $N$.
For $A\in N$,
\[
 N\models\text{``$A$ is maximal in $R$''}
 \quad\Longleftrightarrow\quad
 A\subseteq R\text{ is a maximal antichain in }V.
\]
Thus $G$ is $N$-generic. Since $R$ is $\kappa$-cc,
\cite[Fact~14(3)]{Johnstone} gives
$({}^{<\kappa}N[G])^{V[G]}\subseteq N[G]$.
In particular internal well-founded codes on $\kappa$ are
actually well-founded. The internal enumeration of $R$ ensures
that its positive names belong to its internal $H_{\kappa^+}$.

For parameters represented by $\vec\sigma\in N$, the internal
truth lemma and correctness of the translated formula give
\[
 \begin{aligned}
 N[G]\models\psi(\vec\sigma^G)
 &\Longrightarrow
     (\exists p\in G)\ N\models F_R\psi(p,\vec\sigma)\\
 &\Longrightarrow
     (\exists p\in G)\ p\Vdash_V\psi(\vec\sigma)\\
 &\Longrightarrow V[G]\models\psi(\vec\sigma^G).
 \end{aligned}
\]
Applying the same implications to $\neg\psi$ proves the
converse. Correctness for $\operatorname{Code}$ supplies the
validity of all parameter representations in both applications.
Finally, if $\lambda>\kappa$ is internally inaccessible, then
$|R|^N\leq\kappa<\lambda$ preserves that inaccessibility in
$N[G]$, so the required internal domains and satisfaction
relations continue to exist.
\end{proof}

Fix an inaccessible $\kappa$ and put $K=H_{\kappa^+}$.
For the embedding argument, we use the following presentation of the
iteration from Lemma~\ref{lem:iteration}, defined by formulas over $K$.

A \emph{nice function-name code} over an order $E$ is a sequence
\[
 B=\langle(A_\zeta,v_\zeta):\zeta<\kappa\rangle,
\]
where each $A_\zeta$ is a maximal antichain of $E$ and
$v_\zeta:A_\zeta\to\kappa$. It denotes the globally total
function name whose value at $\zeta$ is $v_\zeta(a)$ when the
generic meets $A_\zeta$ at $a$; write this name as $\dot h_B$.

\begin{definition}\label{up:canonical}
Recursively define $\EE_\theta$ for $\theta\leq\kappa^+$.
Set $\EE_0=\{\varnothing\}$. A condition of $\EE_\theta$ is
a partial function $p\in K$ with support of size less than $\kappa$
in $\theta$. At every $\gamma$ in its support,
\[
 p(\gamma)=(s^p_\gamma,B^p_\gamma),\qquad
 s^p_\gamma\in{}^{<\kappa}\kappa,
\]
where $B^p_\gamma$ is a nice code over $\EE_\gamma$ and
$p\restriction\gamma$ is a valid condition forcing
$s^p_\gamma\subseteq\dot h_{B^p_\gamma}$.
All initial restrictions must be valid.
At an omitted coordinate use the empty stem and the constant zero
function. The empty function is the greatest condition.

For valid $p,q$, put $q\leq p$ if its support extends that of $p$
and, for every $\gamma\in\operatorname{supp}(p)$,
\[
 s^p_\gamma\subseteq s^q_\gamma,
 \qquad q\restriction\gamma\Vdash
           \dot h_{B^p_\gamma}\leq\dot h_{B^q_\gamma}
              \text{ pointwise}.
\]
\end{definition}

\begin{lemma}\label{up:canonical-nice}
The order $\EE_\theta$ is forcing equivalent to the full iteration
$\PP_\theta$, is $\kappa^+$-Knaster, and has all its conditions in
$K$. Every nice code and every family of at most $\kappa$ conditions
has hereditary size at most $\kappa$.
\end{lemma}
\begin{proof}
Inductively define $i_\theta:\EE_\theta\to\PP_\theta$ by
replacing each $(s^p_\gamma,B^p_\gamma)$ with
$(\check s^p_\gamma,\dot h_{B^p_\gamma})$, translating the
antichain conditions by the earlier maps. The definitions give
\[
 \begin{aligned}
 i_\theta(p)\restr\gamma&=i_\gamma(p\restr\gamma),\\
 q\leq_{\EE_\theta}p&\iff
                      i_\theta(q)\leq_{\PP_\theta}i_\theta(p),\\
 p\perp_{\EE_\theta}q&\iff
                      i_\theta(p)\perp_{\PP_\theta}i_\theta(q).
 \end{aligned}
\]
For density, strengthen a condition of $\PP_\theta$ to decide its
stems. At coordinate $\gamma$, mix its side-function name below
the current restriction with a fixed total function elsewhere,
and choose maximal antichains deciding all its values. Translate
these antichains by $i_\gamma$ and use the resulting nice code.
Recursion through the support gives
\[
 \forall p\in\PP_\theta\ \exists q\in\EE_\theta\quad i_\theta(q)\leq p.
\]
Thus the maps induce forcing equivalence and transfer the
Knaster property.

Each row has at most $\kappa$ members. By induction, the
transitive closure of every earlier condition has size at most
$\kappa$. A nice code has $\kappa$ rows, and a condition has
fewer than $\kappa$ coordinates, each below $\kappa^+$. Hence
\[
 |\tc(\{B\})|\leq\kappa\cdot\kappa\cdot\kappa=\kappa,
 \qquad |\tc(\{p\})|\leq\kappa,
 \qquad
 |\tc(A)|\leq\kappa\quad(|A|\leq\kappa,\ A\subseteq\EE_\theta).
\]
All the stated objects therefore belong to $K$.
\end{proof}

The coordinate forcing tests are finite once the earlier order is
given. For example, if $B$ has rows $(A_\zeta,v_\zeta)$, then
\[
 p\Vdash s\subseteq\dot h_B
 \quad\Longleftrightarrow\quad
 \forall\zeta<\operatorname{dom}(s)\ \forall a\in A_\zeta\,
 \bigl(\operatorname{Comp}(p,a)\longrightarrow
                          v_\zeta(a)=s(\zeta)\bigr).
\]
For two codes, the pointwise comparison test requires the
corresponding inequality of row values whenever $p$ and the two
row conditions have a common extension. Nice-code validity tests
maximality of every row in the earlier order.

\begin{lemma}\label{up:directed}
The canonical \(\EE_\theta\) is \(<\kappa\)-directed closed.
This holds internally whenever its local presentation sets exist.
\end{lemma}
\begin{proof}
Let $C\subseteq\EE_\theta$ be directed, with $|C|<\kappa$,
and let $S=\bigcup_{c\in C}\supp(c)$. Then $|S|<\kappa$.
For $\gamma\in S$, directedness makes the stems comparable;
put
\[
 s_\gamma=\bigcup_{c\in C}s^c_\gamma,
 \qquad \dot h_\gamma(\xi)=\sup_{c\in C}\dot h_{B^c_\gamma}(\xi),
\]
using the zero function at omitted coordinates. Regularity
forces $\dot h_\gamma:\kappa\to\kappa$, and the earlier
chain condition gives a nice code for this name.

Define $r$ on $S$ by these stems and codes. Assume inductively
that $r\restr\gamma\leq c\restr\gamma$ for all $c\in C$.
If $\alpha\in\dom(s_\gamma)$, choose $c_0\in C$ with
$\alpha\in\dom(s^{c_0}_\gamma)$. For each $c\in C$, choose
$d\in C$ below $c,c_0$. Then
\[
 r\restr\gamma\Vdash
 \dot h_{B^c_\gamma}(\alpha)
       \leq\dot h_{B^d_\gamma}(\alpha)=s_\gamma(\alpha).
\]
The term $c_0$ gives equality in the supremum. Hence
$r\restr\gamma\Vdash s_\gamma\subseteq\dot h_\gamma$, and
this coordinate is valid and extends every coordinate of $C$.
Induction gives $r\in\EE_\theta$ and $r\leq c$ for all
$c\in C$. Every step is internal when the presentation sets exist.
\end{proof}

\begin{lemma}\label{up:iteration-formula}
For every \(\theta<\kappa^+\), there are data of size \(\kappa\)
and one well-founded infinitary first-order formula \(\psi_\theta\)
over \(H_{\kappa^+}\), of size at most \(\kappa\), such that
correctness on all its internal parameters implies simultaneously,
for every \(\xi\leq\theta\):
\begin{enumerate}
\item for every object in the model's internal $H_{\kappa^+}$,
validity as an $\EE_\xi$-condition is absolute;
\item order and compatibility of internal conditions are absolute;
\item every internally maximal antichain belonging to the model is
predense in the actual \(\EE_\xi\).
\end{enumerate}
\end{lemma}
\begin{proof}
For $\xi\leq\theta$, let $\mathsf{Val}_\xi(p)$,
$\mathsf{Le}_\xi(q,p)$, $\mathsf{Comp}_\xi(p,q)$ and
$\mathsf{Max}_\xi(A)$ denote the formulas defined by the
following recursion over $K=H_{\kappa^+}$. At stage zero,
$\mathsf{Val}_0(p)$ means $p=\varnothing$.

Suppose the formulas for $\gamma<\xi$ are defined. For
$B=\langle(A_\zeta,v_\zeta):\zeta<\kappa\rangle$, the
nice-name requirement is
\[
 \mathsf{Nice}_\gamma(B)\quad\Longleftrightarrow\quad
 \forall\zeta<\kappa\,
   (\mathsf{Max}_\gamma(A_\zeta)\wedge v_\zeta:A_\zeta\to\kappa).
\]
For a stem $s$, define
\[
 \begin{split}
 \mathsf{Norm}_\gamma(r,s,B)\quad\Longleftrightarrow\quad
 \forall\zeta<\dom(s)\ \forall a\in A_\zeta\,
 (\mathsf{Comp}_\gamma(r,a)\longrightarrow v_\zeta(a)=s(\zeta)).
 \end{split}
\]
For $B' =\langle(A'_\zeta,v'_\zeta):\zeta<\kappa\rangle$, put
\[
 \begin{split}
 \mathsf{Dom}_\gamma(r,B,B')\quad\Longleftrightarrow\quad
 \forall\zeta<\kappa\ \forall a\in A_\zeta\ \forall a'\in A'_\zeta\,
 \bigl[&\exists t\in K\,
    (\mathsf{Le}_\gamma(t,r)\wedge\mathsf{Le}_\gamma(t,a)
                              \wedge\mathsf{Le}_\gamma(t,a'))\\
 &\longrightarrow v_\zeta(a)\leq v'_\zeta(a')\bigr].
 \end{split}
\]
The forcing tests following Definition~\ref{up:canonical}
identify these formulas with normalization and pointwise
domination of the named functions.

Now $\mathsf{Val}_\xi(p)$ asserts that $p$ is a partial
function on $\xi$ with $|\dom(p)|<\kappa$ and, for every
$\gamma\in\dom(p)$,
\[
 \begin{gathered}
 p(\gamma)=(s^p_\gamma,B^p_\gamma),\qquad
 s^p_\gamma\in{}^{<\kappa}\kappa,\\
 \mathsf{Val}_\gamma(p\restr\gamma),\qquad
 \mathsf{Nice}_\gamma(B^p_\gamma),\qquad
 \mathsf{Norm}_\gamma(p\restr\gamma,s^p_\gamma,B^p_\gamma).
 \end{gathered}
\]
After defining validity, define order, compatibility and
maximal antichains in that order:
\[
 \begin{aligned}
 \mathsf{Le}_\xi(q,p)\ \Longleftrightarrow\quad
 &\mathsf{Val}_\xi(p)\wedge\mathsf{Val}_\xi(q)
                \wedge\dom(p)\subseteq\dom(q)\\
 &{}\wedge\bigwedge_{\gamma<\xi}
 \bigl[\gamma\in\dom(p)\longrightarrow
  (s^p_\gamma\subseteq s^q_\gamma\wedge
   \mathsf{Dom}_\gamma(q\restr\gamma,B^p_\gamma,B^q_\gamma))\bigr],\\
 \mathsf{Comp}_\xi(p,q)\ \Longleftrightarrow\quad
 &\exists r\in K\,
             (\mathsf{Le}_\xi(r,p)\wedge\mathsf{Le}_\xi(r,q)),\\
 \mathsf{Max}_\xi(A)\ \Longleftrightarrow\quad
 &(\forall a\in A)\,\mathsf{Val}_\xi(a)\\
 &{}\wedge(\forall a\ne b\in A)\,\neg\mathsf{Comp}_\xi(a,b)\\
 &{}\wedge\forall p\in K\,
       (\mathsf{Val}_\xi(p)\longrightarrow
                        \exists a\in A\,\mathsf{Comp}_\xi(p,a)).
 \end{aligned}
\]
In the validity formula, the quantifier over supported coordinates
is likewise expanded as a conjunction indexed by $\gamma<\xi$.
All occurrences of $\gamma$ use fixed enumerations of the
initial segments of $\theta+1$, collected in one relation parameter.
The remaining quantifiers range over $K$.

Substitute the defining formulas at each occurrence. Apart from
the displayed finite sequence of definitions at one stage, each
substitution refers to a strictly smaller stage. Thus every
branch follows a finite decreasing sequence of ordinals at most
$\theta$. There are at most $\kappa$ such sequences and at
most $\kappa$ formula occurrences. A finite selector chooses
one of the four tests and an ordinal index chooses a stage;
the resulting single formula is $\psi_\theta$.

Correctness for all internal assignments transfers each of the
four displayed tests to its actual interpretation. Internal
maximal antichains have internal size at most $\kappa$, by
the chain-condition proof, and belong to the internal $K$.
The last universal quantifier of $\mathsf{Max}_\xi$ ranges
externally over every actual condition. Its truth therefore
gives actual predensity, including with respect to conditions
absent from the model. The syntactic construction depends only
on $\kappa,\theta$ and the ordinal enumerations, so it can be
chosen before the preparation.
\end{proof}

The lottery preparation is due to Hamkins \cite[Section~3]{HamkinsLottery}.
The lifting argument below follows its use with embeddings of small models
by Johnstone \cite[Section~5]{Johnstone}.
We restrict the lottery sums to the canonical Hechler iterations and
prove preservation for the present bounded-formula correctness hypothesis.

Let \(\InfPrep_\kappa\) be the Easton-support iteration of length
\(\kappa\) which, at an initially inaccessible \(\gamma\) still
inaccessible in the intermediate universe, uses the lottery sum
\[
                 \bigoplus_{\tau<\gamma^+}\EE^{(\gamma)}_\tau,
\]
and is trivial otherwise.
The local \(\gamma^+\) and all summands are computed there.
Each summand is \(<\gamma\)-closed, and the lottery sum has the
same closure: after its first condition distinct from the greatest condition, a decreasing chain
stays in one component.
The sum has size at most \(2^\gamma\) in the canonical presentation
and rank below the next initially inaccessible cardinal, when one
exists.
It is a progressively closed good iteration in the sense of
\cite[Definition 2.14]{CGLH}.
At a Mahlo \(\kappa\)
it has size \(\kappa\) and is \(\kappa\)-cc
\cite[Proposition 2.15]{CGLH}.
It preserves inaccessibility and, at a weakly compact \(\kappa\),
weak compactness \cite[Lemma 2.16]{CGLH}.
Their strategic-closure game has \(\gamma\) moves, for which
\(<\gamma\)-closure of the local iterand suffices.
Use a presentation with conditions in \(V_\kappa\).

\begin{theorem}\label{up:preservation}
Suppose \(\InfH(\kappa)\), and let \(G\subseteq\InfPrep_\kappa\)
be generic.
In \(U=V[G]\), every \(\EE_\theta\), \(\theta<\kappa^+\),
preserves the weak compactness of \(\kappa\).
\end{theorem}
\begin{proof}
For $\theta=0$ use the preparation's preservation of weak
compactness cited above. Fix $0<\theta<\kappa^+$, a name
$\dot A$ for a subset of $\kappa$, and a condition of the
composed forcing. Choose a sufficiently large $H_\Lambda$ and
\[
 X\prec H_\Lambda,\qquad |X|=\kappa,\qquad
 \kappa+1\subseteq X,\qquad{}^{<\kappa}X\subseteq X,
\]
containing the condition, $\dot A$, $\theta$, the preparation
and its order, and a name for the entire $\EE_\theta$.
Also include $\psi_\theta$, the three formulas
\eqref{up:translated-formulas}, a hierarchy normal form of
their relation-code interpretation, and all restriction data.
All these syntax objects are chosen before taking $X$.
Let $\pi:X\cong M$.

If $x\in X$ has hereditary size at most $\kappa$, a
surjection $\kappa\twoheadrightarrow\tc(\{x\})$ in $X$
shows that $\pi(x)=x$. Thus $\theta$, the preparation,
its conditions and the syntax data are fixed. The same
calculation in $X[G]$ fixes each canonical condition it contains.
The composed forcing is $\kappa^+$-cc, so each maximal
antichain in $X$ is a subset of $X$. Therefore actual generics
$G*g$ induce an $M$-generic $G*g_0$ with
\[
 A=\dot A^{G*g}\in M[G][g_0],\qquad
 Q_0=\EE_\theta^{M[G]}\in M[G].
\]
Elementarity and Lemma~\ref{up:directed} give
\[
 M[G]\models\text{``$Q_0$ is $<\kappa$-directed closed''}.
\]

Apply Corollary~\ref{up:Hembedding} to the chosen normal form.
It gives $j:M\to N$, with
\[
 \crit(j)=\kappa,\qquad j,M,e_M\in N,\qquad
 e_M:\kappa\twoheadrightarrow M.
\]
The preparation, its order and a $\kappa$-enumeration belong
to $N$. Lemma~\ref{up:afterprep} makes $N[G]$ correct for
$\psi_\theta$. Lemma~\ref{up:iteration-formula} transfers
validity, order, compatibility and maximal-antichain predensity
at every $\xi\leq\theta$.

Put $Q_N=\EE_\theta^{N[G]}$ and $g_N=g\cap Q_N$.
To verify directedness, for $p,q\in g_N$ take an internal
maximal antichain refining
\[
 \{r\in Q_N:r\leq p,q\text{ or }r\perp p\text{ or }r\perp q\}.
\]
It is actually predense. The actual generic meets it;
compatibility absoluteness excludes the last two cases, so
$g_N$ is directed. The same predensity statement for every
internal maximal antichain proves $N[G]$-genericity.

Set $\lambda=j(\kappa)$ and $\nu=(\kappa^+)^{N[G]}$.
A source surjection $\kappa\twoheadrightarrow\theta$ also
belongs to $N$, so $\theta<\nu$. In $j(\InfPrep_\kappa)$,
choose the $\theta$th component at stage $\kappa$:
\[
 j(\InfPrep_\kappa)\text{ at }\kappa
       =\bigoplus_{\tau<\nu}\EE_\tau^{N[G]}.
\]
Use $g_N$ as its generic. In $U[g]$, enumerate the dense sets
of the preparation tail strictly above $\kappa$ by
$\langle D_\alpha:\alpha<\kappa\rangle$. Recursively choose
$r_{\alpha+1}\leq r_\alpha$ in $D_\alpha$; at a limit
$\delta<\kappa$, closure under sequences of length less than $\kappa$
puts the preceding sequence in the model, where closure of the forcing
gives a lower bound. The resulting tail generic gives $J$ with
\[
 j''G\subseteq J,\qquad
 j_G:M[G]\longrightarrow W=N[J],\qquad
 j_G(\dot x^G)=j(\dot x)^J.
\]
The models have size $\kappa$ in $U[g]$. Closure of $N[G]$ under
sequences of length less than $\kappa$ is preserved by the
$<\kappa$-distributive forcing $\EE_\theta$ and then by evaluation
of names for the tail generic, using \cite[Fact~14(1)]{Johnstone}.
Thus $({}^{<\kappa}W)^{U[g]}\subseteq W$. Since $M$ regards $\kappa$ as Mahlo,
$N$ regards $\lambda$ as Mahlo. The internal good preparation
preserves its inaccessibility, so $W\models\lambda$ is
inaccessible.

The collapse identifies $Q_0$ with the conditions in
$X[G]$ from the actual order. Correctness in $N[G]$ gives
$Q_0\subseteq Q_N$, with the same order and compatibility, and
\[
 g_0=g_N\cap Q_0.
\]
Name evaluation in $W$ forms $j_G,M[G]$ as sets, and the
critical-stage generic gives $g_0\in W$. Replacement in $W$
therefore forms
\[
 D=j_G''g_0\in W.
\]
This set is internally directed. Evaluating the names listed
by $e_M$ gives an internal surjection
$\kappa\twoheadrightarrow M[G]$, so in $W$,
\[
 |D|\leq|Q_0|\leq\kappa<\lambda.
\]
Elementarity and the directed closure of $Q_0$ give
\[
 W\models\exists q^*\in j_G(Q_0)\ \forall p\in g_0\quad
                                      q^*\leq j_G(p).
\]
Starting below $q^*$, enumerate the dense sets of $j_G(Q_0)$
in $W$ by $\kappa$ and repeat the descending recursion above.
Closure of $W$ under sequences of length less than $\kappa$ in $U[g]$
and its internal $<\lambda$-closure give the limit conditions. This yields a $W$-generic $h$ and the lift
\[
 j_{G,g_0}:M[G][g_0]\longrightarrow W[h],\qquad
 j_G''g_0\subseteq h,\qquad\crit(j_{G,g_0})=\kappa.
\]
The source has size $\kappa$ and is closed under sequences of length
less than $\kappa$ in $U[g]$: use the $\kappa$-cc of the preparation,
the fact that the iteration adds no such sequences of names, and
\cite[Fact~14]{Johnstone} for evaluation. It contains the
arbitrary $A\subseteq\kappa$. The weak-compactness embedding
criterion \cite[Fact~4]{Johnstone} proves the conclusion.
\end{proof}

\begin{corollary}\label{up:final-wc}
Suppose $\mathsf H(\kappa)$. After
$\InfPrep_\kappa*\EE_{\kappa^+}$, the cardinal $\kappa$
is weakly compact.
\end{corollary}
\begin{proof}
Let $G*\langle g_\alpha:\alpha<\kappa^+\rangle$ be generic,
and set $U_\theta=V[G][g\restr\theta]$. The final $\kappa$
is inaccessible by the preparation and Lemma~\ref{lem:iteration}.
Code a final $\kappa$-tree $T$ and its level map on $\kappa$.
Lemma~\ref{lem:localization} gives $\theta<\kappa^+$ with
this code in $U_\theta$. For each $\alpha<\kappa$, a final
enumeration $e:\mu\twoheadrightarrow T_\alpha$, $\mu<\kappa$,
is a short ordinal sequence, hence belongs to $U_\theta$.
Therefore
\[
 U_\theta\models\text{``$T$ is a $\kappa$-tree''}.
\]
By Theorem~\ref{up:preservation}, there is a cofinal branch
$b\in U_\theta$ through $T$. It remains cofinal in the final
extension. Inaccessibility and the tree property give weak
compactness of $\kappa$ there.
\end{proof}

\begin{corollary}\label{up:NAC}
If $\mathsf H(\kappa)$ holds, then there is a forcing extension
in which $\kappa$ is weakly compact and $\mathrm{NAC}(\kappa)$
holds. Consequently
\[
 \operatorname{Con}(\mathrm{ZFC}+\exists\kappa\,\mathsf H(\kappa))
 \quad\Longrightarrow\quad
 \operatorname{Con}(\mathrm{ZFC}+\exists\kappa\,
                                  \mathrm{NAC}(\kappa)).
\]
\end{corollary}
\begin{proof}
Force with $\InfPrep_\kappa*\EE_{\kappa^+}$ and write
$U_\theta=V[G][g\restr\theta]$. Corollary~\ref{up:final-wc}
gives final weak compactness. If $\vec C$ is a final
$C$-sequence, Lemma~\ref{lem:localization} gives
$\theta<\kappa^+$ with $\vec C\in U_\theta$. This model
regards $\kappa$ as weakly compact. At the next full Hechler
coordinate, Corollary~\ref{cor:capture} gives
\[
 U_{\theta+1}\models\text{``$D\subseteq\kappa$ is club,
 $T\subseteq\kappa$ is stationary''},\qquad
 D\cap\beta\subseteq C_\beta\quad(\beta\in T).
\]
The remaining quotient is forcing equivalent to a
$<\kappa$-closed order. Lemma~\ref{lem:stationarity}
preserves $T$, and $D$ remains club. Thus in the final model
\[
 T\subseteq\operatorname{Cap}(D,\vec C),\qquad
 \operatorname{Cap}(D,\vec C)\notin\NS_\kappa.
\]
Since $\vec C$ was arbitrary, $\NAC(\kappa)$ holds.
\end{proof}

\section{Proofs of the main theorems}\label{sec:conclusion}

\begin{proof}[Proof of Theorem~\ref{thm:lower}]
Lemma~\ref{fd:SA} gives inaccessibility of $\kappa$ in $L$.
If $I_{\mathrm{ineff}}^L(\kappa)\nsubseteq\A_\kappa^V$,
choose
\[
 T\in I_{\mathrm{ineff}}^L(\kappa)\setminus\A_\kappa^V.
\]
Theorem~\ref{lb:relative-bound} yields
$I_\xi^L(\kappa)\subseteq\SA_\kappa^V$ for every
$1\leq\xi<(\kappa^+)^L$, proving \eqref{intro:ideal}.
If $\kappa$ is nonineffable in $L$, one may take $T=\kappa$,
because $\NAC(\kappa)$. Thus
\[
 L\models\text{``$\kappa$ is ineffable or $\mathsf H(\kappa)$''}.
\]
Corollary~\ref{lb:dichotomy} gives
$L\models\exists\lambda\leq\kappa\,\mathsf H(\lambda)$.
\end{proof}

Corollary~\ref{up:NAC} proves Theorem~\ref{thm:upper}, with weak
compactness of the final witness as an additional conclusion.

\begin{proof}[Proof of Theorem~\ref{thm:main}]
By Theorem~\ref{thm:lower}, every model of
$\mathsf{ZFC}+\exists\kappa\,\NAC(\kappa)$ has a
constructible inner model satisfying
$\mathsf{ZFC}+\exists\lambda\,\mathsf H(\lambda)$.
Theorem~\ref{thm:upper} gives the converse consistency
implication by set forcing. Hence
\[
 \operatorname{Con}(\mathsf{ZFC}+\exists\kappa\,\NAC(\kappa))
 \quad\Longleftrightarrow\quad
 \operatorname{Con}(\mathsf{ZFC}+\exists\lambda\,\mathsf H(\lambda)).
\]
Both implications are formalized by the usual inner-model
and forcing interpretations over $\mathsf{ZFC}$.
\end{proof}

By Corollary~\ref{up:final-wc}, the theory asserting a weakly compact
cardinal with $\NAC$ has the same consistency strength.
By \cite[Corollary~3.5]{CodyHoly}, if $\delta$ is subtle,
some $\lambda<\delta$ satisfies $\mathsf H$, and
$V_\delta$ is a transitive model of the corresponding existence
theory.

\section*{Acknowledgments}
The author thanks Assaf Rinot, who supervised the author's postdoctoral
research at Bar-Ilan University from 2019 to 2021, for discussions on
amenable $C$-sequences and for a substantial simplification of the proof
of Lemma~\ref{lem:iteration-size}.

\section*{Declaration of AI-assisted preparation}
OpenAI's Codex and ChatGPT were used to assist with literature
searches, exploration and checking of mathematical arguments,
organization, English-language revision, and preparation of
the \LaTeX\ manuscript.


\end{document}